\documentclass[12pt,a4paper]{amsart}
\usepackage{latexsym}
\usepackage{amsthm}
\usepackage{amssymb}
\usepackage{amsmath}
\usepackage{amsfonts}
\usepackage{mathrsfs}
\usepackage{enumerate}
\usepackage[retainorgcmds]{IEEEtrantools}
\usepackage{color}
\usepackage{tikz}
\usepackage{epsfig,amscd}
\usetikzlibrary{positioning,shapes,shadows}
\usetikzlibrary{arrows,automata}
\usepackage{float}
\usepackage{pgf}
\usepackage{epstopdf}
\usepackage[mathscr]{eucal}  %to use mathematical symbol
\usepackage{verbatim}
\def\newtheorems{\newtheorem{theorem}{Theorem}[section]
                 \newtheorem{cor}[theorem]{Corollary}
                 \newtheorem{prop}[theorem]{Proposition}
                 \newtheorem{lemma}[theorem]{Lemma}
                 \newtheorem{definition}[theorem]{Definition}
                 \newtheorem{notation}[theorem]{Notation}
                 \newtheorem{claim}[theorem]{Claim}
                 
                 \theoremstyle{definition}
                 \newtheorem{example}[theorem]{Example}
                 
                \theoremstyle{definition}
                \newtheorem{remark}[theorem]{Remark}
                 \newtheorem{question}{Question}[section]
                 }
\newtheorems

\newcommand{\CCn}{\mathfrak{C}_n}
\newcommand{\CCk}[1]{\mathfrak{C}_{#1}}

\newcommand{\CCnk}[2]{\mathfrak{C}_{#1,#2}}
\newcommand{\CCnr}{\CCnk{n}{r}}
\newcommand{\Ban}{{\bf U}_n}
\newcommand{\Banr}{{\bf U}_{n,r}}

\newcommand{\Aut}{\mathop{\mathrm{Aut}}\nolimits}
\newcommand{\Out}{\mathop{\mathrm{Out}}\nolimits}
\newcommand{\Inn}{\mathop{\mathrm{Inn}}\nolimits}

\newcommand{\Dec}{\mathop{\mathrm{Dec}}\nolimits}
\newcommand{\crit}{\mathop{\mathrm{crit}}\nolimits}
\newcommand{\Core}{\mathop{\mathrm{Core}}\nolimits}
\newcommand{\Sym}{\mathop{\mathrm{Sym}}\nolimits}

\newcommand{\Homeo}{\mathop{\mathrm{Homeo}}\nolimits}
\newcommand{\Image}{\mathop{\mathrm{Image}}\nolimits}
\newcommand{\Root}{\mathop{\mathrm{Root}}\nolimits}
\newcommand{\Range}{\mathop{\mathrm{Range}}\nolimits}
\newcommand{\LoopStates}{\mathop{\mathrm{LoopStates}}\nolimits}
\newcommand{\core}[1]{\mathop{\mathrm{Core}}\nolimits(#1)}

\newcommand{\T}[1]{\mathcal{#1}}

\newcommand{\Rn}{\mathcal{R}_{n}}
\newcommand{\Rnr}{\mathcal{R}_{n,r}}
\newcommand{\Anr}{\Aut(\Gnr)}
\newcommand{\Bnk}[2]{\mathcal{B}_{#1,#2}}
\newcommand{\Bnr}{\Bnk{n}{r}}
\newcommand{\Snr}{\mathcal{S}_{n,r}}
\newcommand{\Mnr}{\mathcal{M}_{n,r}}
\newcommand{\LBnr}{\mathcal{L}\mathcal{B}_{n,r}}
\newcommand{\PBnr}{\mathcal{H}\mathcal{B}_{n,r}}
\newcommand{\HBnr}{\mathcal{H}\mathcal{B}_{n,r}}
\newcommand{\Ln}{\mathcal{L}_n}
\newcommand{\Lnr}{\mathcal{L}_{n,r}}

\newcommand{\Hn}{\mathcal{H}_n}
\newcommand{\Hnr}{\mathcal{H}_{n,r}}
\newcommand{\On}{\mathcal{O}_n}
\newcommand{\Nn}{\mathcal{N}_{n}}
\newcommand{\SOn}{\widetilde{\mathcal{O}}_{n}}
\newcommand{\Outnr}{\Out(G_{n,r})}

\newcommand{\Onk}[2]{\mathcal{O}_{#1,#2}}
\newcommand{\Onr}{\Onk{n}{r}}

\newcommand{\groupG}[2]{G_{#1,#2}}

\newcommand{\Gnr}{\groupG{n}{r}}

\newcommand{\concat}{\kern1.83pt\hat{\ }\kern1.4pt}
\newcommand{\pair}[2]{\langle#1 ,#2\kern1.4pt\rangle}

\newcommand{\trpl}[3]{\{ #1, #2, #3 \}}
\newcommand{\setm}[2]{\{#1\mid #2\}}
\newcommand{\fsetn}[2]{\{\,#1,\ldots ,#2\}}

\newcommand{\veceta}{\vec{\eta}}
\newcommand{\vecnu}{\vec{\nu}}
\newcommand{\veczeta}{\vec{\zeta}}
\newcommand{\fnn}[3]{#1:{#2}\to{#3}}

\newcommand{\Gnrpcn}{G_{n,r,\mathrm{pcn}}}
\newcommand{\Gnrdcn}{G_{n,r,\mathrm{dcn}}}
\newcommand{\Hnrsim}{H_{n,r,{\sim}}}

\newcommand{\nset}{\{0,1,\ldots,n-1\}}
\newcommand{\Xn}{X_n}
\newcommand{\xn}{\Xn}
\newcommand{\rset}{\left\{\dot{0},\dot{1},\ldots,\dot{r-1}\right\}}

\newcommand{\Wne}{W_{n,\epsilon}}
\newcommand{\xns}{\Wne}
\newcommand{\Wn}{W_n}
\newcommand{\Wnre}{W_{n,r,\epsilon}}
\newcommand{\Wnr}{W_{n,r}}
\newcommand{\wnr}[2]{W_{#1,#2}}
\newcommand{\wnre}[2]{W_{#1,#2,\epsilon}}
\newcommand{\rd}{\mathbf{\dot{r}}}
\newcommand{\N}{\mathbb{N}}
\newcommand{\Z}{\mathbb{Z}}
\newcommand{\im}{\mbox{ im}}
\newcommand{\seteq}{:=}
\newcommand{\dgfrac}[2]{#1 / #2}

\newcommand{\bbN}{\mathbb{N}}

\newcommand{\nrestriction}{\kern-2.5pt\upharpoonright\kern-2.5pt}

\newcommand{\gen}[1]{\langle #1 \rangle}
\newcommand{\Skip}[3]{#1#2 \ldots#2 #3}

\def\fs#1{\mbox{\it #1\kern 1.3pt}}
\def\rfs#1{\mbox{\rm #1\kern 1.3pt}}
\def\bfs#1{\mbox{\bf #1\kern 1.3pt}}
\def\fss#1{\mbox{\scriptsize\it #1\kern 1.3pt}}
\def\fst#1{\mbox{\tiny\it #1\kern 1.1pt}}
\def\sifs#1{\mbox{\scriptsize\it #1\kern 1.3pt}}
\def\srfs#1{\mbox{\kern0.7pt\scriptsize\rm #1\kern 1.3pt}}
\def\sbfs#1{\mbox{\kern0.7pt\srbf #1\kern -0.6pt}}
\def\srbfs#1{\mbox{\kern0.7pt\srbf #1\kern -0.6pt}}
\def\spfs#1{\mbox{\kern0.7pt\scmu #1\kern 1.3pt}}
\def\sspfs#1{\mbox{\kern0.5pt\sscmu #1\kern 1.1pt}}
\def\ssbfs#1{\mbox{\kern0.7pt\ssbf #1\kern 1.3pt}}
\def\fsm#1{\mbox{\tiny\it #1\kern 1.0pt}}
\newcommand{\scirc}
{\raise1pt\hbox{\scriptsize\kern1.5pt$\circ$\kern1.5pt}}

\title[Automorphisms via dynamics]{The further chameleon groups of Richard Thompson and Graham Higman: Automorphisms via dynamics for the Higman-Thompson groups $\Gnr$}
\author{C. Bleak, P. Cameron, Y. Maissel, A. Navas, and F. Olukoya}

\RequirePackage[normalem]{ulem} %DIF PREAMBLE
\RequirePackage{color}\definecolor{RED}{rgb}{1,0,0}\definecolor{BLUE}{rgb}{0,0,1} %DIF PREAMBLE

\begin{document}

\maketitle
\par{\centering {\it In memory of Matayahu Rubin: researcher, teacher, and friend.}\par}
\vspace{.2 in}

{\flushleft {\it Abstract:}}\\
We describe, through the use of Rubin's theorem, the automorphism groups of the Higman-Thompson groups $\Gnr$ as groups of specific homeomorphisms of Cantor spaces $\CCnr$.  This continues a thread of research begun by Brin, and extended later by Brin and Guzm\'an: to characterise the automorphism groups of the `Chameleon groups of Richard Thompson,' as Brin referred to them in 1996.  The work here completes the first stage of that twenty-year-old program, containing (amongst other things) a characterisation of the automorphism group of $V$, which was the `last chameleon.'  The homeomorphisms which arise fit naturally into the framework of Grigorchuk, Nekrashevich, and Suschanski\u\i's \emph{rational group of transducers}: they are exactly those homeomorphisms which are induced by \emph{bi-synchronizing} transducers, which we define in the paper. This result appears to offer insight into the nature of Brin and Guzm\'an's \emph{exotic automorphisms}, while also uncovering connections with the theory of reset words for automata (arising in the Road Colouring Problem) and with the theory of automorphism groups of the full shift.

\tableofcontents

\section{Introduction}
In this article, we describe the automorphism groups $\Anr$ of the Higman-Thompson groups $\{\Gnr\}$, the first infinite family of finitely presented infinite (almost) simple groups to be discovered.

The description of the automorphism group of $V=G_{2,1}$  (and more generally of $G_{n,r}$) has remained a challenge to the community of researchers of the R. Thompson groups since Brin's 1996 article \cite{BrinCh}, which characterises the automorphism groups of $F$ and $T$, but which leaves $V$ as the last `chameleon.'  Brin and Guzm\'an's following article \cite{BGAut} explores many further properties of the automorphism groups of the generalised Thompson groups $F_{n,r}$ and $T_{n,r}$, which are subgroups of $G_{n,r}$. Properties  they discovered include the intriguing existence of `Exotic automorphisms' when $n>2$.  (The groups $F$ and $T$ correspond to $F_{2,1}$ and $T_{2,1}$ in a generalised notation introduced by Brown in \cite{BrownFinite}.)    However, Brin and Guzm\'an leave the mysterious groups of automorphisms of the groups $G_{n,r}$ untouched.  In \cite{BurilloClearyAut}, Burillo and Cleary further study the automorphism group of $F$, investigating some of its metric properties and providing a presentation for this group.

Our paper fills the main gap mentioned above.

To obtain our classification, we follow a three-step process. Firstly, we demonstrate through the use of Rubin's theorem \cite{Rubin} that the automorphism group $\Anr$ of $G_{n,r}$ can be naturally identified as an overgroup $\Bnr$ of $G_{n,r}$ in the group of homeomorphisms of a specific Cantor space $\CCnr$. Secondly, we show that any such homeomorphism must actually have a special property (\emph{finitely many local actions}) that is equivalent to the homeomorphism being representable by a finite state transducer as observed in \cite{GNSenglish}. Such homeomorphisms have inverses of the same sort, which means that we obtain $\Anr$ as a subgroup of the rational group $\Rnr \cong \Rn$ of Grigorchuk, Nekrashevych, and Suschanski\u{\i}.  (This group will reappear frequently in our discussion; so we abbreviate the authors as GNS.) Finally, we classify exactly which subgroup of $\Rnr$ we are getting: we show that the elements of $\Anr$ are precisely those homeomorphisms which can be represented by \emph{strongly synchronizing} transducers and with inverses also being strongly synchronizing (we call such homeomorphisms \emph{bi-synchronizing}).

As the reader can see from the above description, our characterisation of the automorphisms of the groups $\{\Gnr\}$ is similar to Brin's \cite{BrinCh} (and later Brin and Guzm\'an's \cite{BGAut}) for other Thompson-esque groups, in that we describe these automorphism groups as subgroups of homeomorphism groups of Cantor spaces through the use of Rubin's theorem, using the descriptions of the groups $\{\Gnr\}$ as groups of homeomorphisms of the relevant Cantor spaces.  In contrast to the earlier results, the total disconnected nature of Cantor space makes pursuing the outline of the papers \cite{BrinCh} and \cite{BGAut} more difficult.  Section \ref{sec:Rubin} of this paper represents the culmination of the previously taken approach; all of the later work is what is required to bridge the gap and analyse the results. We characterize $\Gnr$ as the set of elements fixing the equivalence classes of the relation that two points in $\CCnr$ are equivalent if they admit a common infinite suffix (in the natural labelling of the points in that space).  However, $\Gnr$ acts highly transitively on any given equivalence class under that relation (that is, $k$-transitively for all positive integers $k$).  Also, the points of any such equivalence class are spread densely throughout the Cantor space $\CCnr$. These features eventually lead to our representation by bi-synchronizing transducers.

The remainder of this work is concerned with several offshoots from our three-step outline mentioned earlier.  Firstly, we use our classification of the automorphism groups to obtain strong results about the outer automorphism groups of the groups $G_{n,r}$.  We are also able, for a fixed $n$, to assemble the set of groups $\{\Onr\mid 1\leq r<n\}$ into a larger group $\On$ which has elements entirely defined in a combinatorial fashion and with a synthetic multiplication operation, which group has subgroups connected to the theory of the group of automorphisms of the shift on $n$ letters.  We also explore properties of bi-synchronizing transducers in their own right as object of interest.

While we have not further investigated the automorphism groups of the groups $F_n$ and $T_n$ as given by Brin and Guzm\'an, the connection with transducers and the rational group appears to provide a very natural framework for exploring some of the questions raised in their paper.  Indeed, we hope the viewpoint taken here towards the groups of automorphisms that we study may help to further understand the `exotic' automorphisms found by Brin and Guzm\'an.

In the next subsections for given $1 \le r \le n$ we define the Cantor space $\CCnr$ and the groups $\Gnr$, $\Bnr$, and $\Rnr$ mentioned above (and some other groups of interest as well). After that we give precise statements of our chief results. We will also interleave some discussion of the state of current research on the questions answered here.

\subsection{Cantor spaces and groups}
For the remainder of this subsection, the symbols $r$ and $n$ will represent two natural numbers so that $1\leqslant r <n$.  It is fine to allow $r\geq n$ as well, but we do not as under the general form of the Higman definition of the groups $\Gnr$, the groups $G_{n,r}$ and $G_{n,(r+n-1)}$ are isomorphic.

Given such $r$ and $n$, the Cantor space $\CCnr$ is the space consisting of all infinite sequences defined as follows:
\[
\CCnr:=\{c a_1 a_2 a_3\ldots \mid a_i \in \nset, c\in \rd\},
\]
where $\rd$ is the set $\{\dot{0},\dot{1},\ldots, \dot{r-1}\}$ which is a set of $r$ symbols disjoint from $\nset$.  That is, $\CCnr$ can be thought of as a disjoint union of $r$ copies of the infinite $n$-ary Cantor space $\mathfrak{C}_n:=\nset^\omega$ (the standard topology on $\CCnr$ is the product topology, considering $\rd$ and $\{0,1,2,\ldots, n-1\}$ as finite discrete spaces).

Now, the group $\Gnr$ is then precisely the group generated by prefix replacement maps: one specifies two incomparable finite prefixes $c_1a_1a_2\ldots a_j$ and $c_2b_1b_2\ldots b_k$ (for some indices $j$ and $k$), and then `swaps' these prefixes  (here, two prefixes are incomparable if neither is a prefix of the other).  For example, a point $c_1a_1a_2\ldots a_ja_{j+1}a_{j+2}\ldots$ would map to $c_2 b_1b_2\ldots b_ka_{j+1}a_{j+2}\ldots$ while $c_2b_1b_2\ldots b_kb_{k+1}b_{k+2}\ldots$ would map to $c_1a_1a_2\ldots a_jb_{k+1}b_{k+2}\ldots$.  Note that one can think of this group as a group of piecewise affine transformations of the space $\CCnr$ which are locally orientation preserving.  The book \cite{Higmanfpsg} introduces these groups and is still a main source of information on the Higman-Thompson groups, which family of groups provides the first infinite source of infinite, finitely presented simple groups (the commutator subgroup of $\Gnr$ is always simple, and is equal to $\Gnr$ when $n$ is even, or is index two in $\Gnr$ when $n$ is odd).  The Higman-Thompson groups are much studied but are still of topical interest, retaining as they do some cloak of mystery.  Some investigations of these groups are \cite{Higmanfpsg,Pardo,MR1170379,Britaetal,Thumann}.

It is well known that in the case $n=2$ and $r=1$, the group $G_{2,1}$ is isomorphic to the R. Thompson group $V$, and in general, following the notation of Brown introduced in \cite{BrownFinite}, we will denote $G_{n,1}$ as $V_n$.

Similarly, there are subgroups $F_n<T_n<V_n$ (again adopting the notation of Brown).  The groups $F_n$, $T_n$ and $V_n$ naturally generalize the R. Thompson groups $F=F_2$, $T=T_2$, and $V=V_2$. (See Thompson's 1965 notes \cite{ThompsonNotes} or the oft-cited survey \cite{CFP} for more information on the R. Thompson groups.)  We will not discuss the groups $F_n$ and $T_n$ in any depth in this article, but we will relate the work done here to previous work carried out for those groups.

For the Cantor space $\mathfrak{C}_n$ there is the group $\mathcal{R}_n$ of homeomorphisms of $\mathfrak{C}_n$, called the \emph{rational group} (on $n$ letters) by its discoverers Grigorchuk, Nekrashevych, and Suschanski\u\i\, in \cite{ GNSenglish, GNSrussian}.  This is the group of homeomorphisms of $\mathfrak{C}_n$ that can be represented by finite (asynchronous) transducers inducing the appropriate transformations of the infinite sequences corresponding to points in $\mathfrak{C}_n$. A transducer is a directed edge-labelled graph with vertices called states, and with edge labels taken from a finite alphabet $\mathcal{A}$.  In normal usage, a (finite) transducer is considered to have an active state, and it reads an input letter  in the alphabet $\mathcal{A}$, transitions the active state using the directions provided by labels on the edges and writes an output word from the alphabet $\mathcal{A}$ according to the edge traversed.  The defining characteristic of a homeomorphism of $\mathfrak{C}_n$ which admits a finite asynchronous transducer to represent it is that such a homeomorphism has only finitely many local actions on the basic open sets of the relevant Cantor space; each basic open set maps to its image using a scaled version of one of these local actions.  The local actions then correspond to the states of the representative transducer.

%The fundamental aspects of the $\mathcal{R}_n$  theory are also developed in \cite{GNSenglish}.
By an essentially trivial modification of the groups $\mathcal{R}_n$, we introduce here the groups $\Rnr$, which are just like the groups $\mathcal{R}_n$ except that the resulting transducers process points in the Cantor spaces $\CCnr$.

%%%%%%%%%%%%%%%%%%%%%%%%%%%%%%%%%%%%

\subsection{Discussion and statements of results}

Recall that the automorphism groups of $F$ and $T$ are described in Brin's landmark paper \cite{BrinCh}.  In the later paper \cite{BGAut}, Brin and Guzm\'an go on to explore the automorphism groups of $F_n$ and $T_n$ for $n>2$, where they make the startling discovery of `exotic' automorphisms.

Perhaps surprisingly, the methods of \cite{BrinCh, BGAut}  fail to restrict possibilities sufficiently to create a meaningful description of the automorphisms of $V_n$.  We say just a few words on this here.  The groups $F_n$ and $T_n$ can be thought of as groups (under composition) of certain piecewise affine homeomorphisms of the spaces $\mathbb{R}$ or $\mathbb{S}^1$, respectively.  The approaches of Brin in \cite{BrinCh} and of Brin and Guzm\'an in \cite{BGAut} both make use of Rubin's theorem to understand an automorphism of a group $F_n$ or $T_n$ as a topological conjugation by a homeomorphism of the relevant space.

For a given $n>1$, the group $V_n$ can be thought of as consisting of the piecewise affine transformations of the Cantor space $\CCn$;   in some sense $V_n$ is the group of ``PL approximation homeomorphisms'' of the full group of homeomorphisms of $\CCn$.  As such, $V_n$ represents a very ``large'' group, and Rubin's theorem again applies.  However, $V_n$ takes full advantage of the totally disconnected and homogeneous nature of its relevant Cantor space.  In consequence, the groupoid of local germs of elements of $V_n$ turns out to be too flexible to provide sufficiently restrictive information on its own to characterize the automorphisms of $V_n$.

Our first theorem represents a resolution of the problem mentioned above.  Through a heavy use of some versions of transitivity of the action of $\Gnr$ on $\CCnr$ we are able to show that an automorphism of $\Gnr$ will only admit finitely many types of local action on $\CCnr$, so automorphisms of $\Gnr$ are representable by finite transducers.

The ``transitivity'' of the action of $\Gnr$ on $\CCnr$ also allows us to see that for a given automorphism, after finitely many steps, the active state of our representative transducer has to be in a specific state.

%, as the related but weaker term ``synchronizing'' is the common term established by research on the Road Colouring Problem and the  \v{C}ern\'y Conjecture .

Specifically, a transducer is \emph{strongly synchronizing at level $m$} if there is a natural number $m$ so that whenever the transducer reads an input word of length $m$, the resulting active state is then known, regardless of the initial active state.  A homeomorphism is representable by a \emph{bi-synchonizing transducer}  if there is a natural number $m$ so that the homeomorphism and its inverse are both representable by finite transducers which are strongly synchronizing at level $n$.  Note that there exist homeomorphisms representable by strongly synchronising transducers but where the inverse of the homeomorphism cannot be represented by a transducer with strong synchronization (see \cite{BleakDonovenJonusas} for examples of these sorts of homeomorphisms).

A transducer which has the property that after reading a specific string (a \emph{reset word}) from \emph{any} state one knows which of the states of the transducer has become the ``active state'' is called a synchronizing transducer (in the literature around the \v{C}ern\'y Conjecture and the road colouring problem (see, e.g., \cite{Trahtman09,Volkov2008})).  This established use of the adjective ``synchronizing'' motivates our choice of language.  Strongly synchronizing transducers are considered further \cite{BleakCameronOlukoya}.

Note there is also another unfortunate collision in nomenclature in the literature.  Transducers which transform input strings in a ``one letter in, one letter out'' fashion from each of their states are called synchronous transducers.  A typical example of such is the transducer whose states represent the standard generators of the Grigorchuk group.

\begin{theorem} \label{thm:MainTheorem}
Let $\Rnr$ represent the generalized GNS rational group of homeomorphisms of the Cantor space $\CCnr$ (that is, those homeomorphisms which are representable by finite initial transducers). The subgroup $\Bnr$ of $\mathcal{R}_{n,r}$ of homeomorphisms representable by bi-synchronizing finite transducers contains $\Gnr$, and is isomorphic to $\Anr$.\end{theorem}

\subsubsection{On outer automorphisms}
{ As mentioned above, the union  over all valid $r$ of the outer automorphism $\Onr$ groups of $\Gnr$ forms a group under an appropriately defined transducer product. Thus we have the following theorem.
\begin{theorem}\label{thm: union of outnr is a group}
 Let $r$ be a positive integer less than $n$, and denote by $\Onr$  the outer automorphism groups of $\Gnr$.  Then  $\On :=\bigcup_{1\le r< n} \Onr,$ with an appropriately defined binary operation extending multiplication in $\Outnr,$ forms a group.
 \end{theorem}}

 The path to the proof of this is perhaps of interest.  One can
 show that any given transducer $A_{q_0}$ representing an
 element of $\Bnr$ has a special sub-transducer (the \emph{core}
 of $A_{q_0}$), which precisely characterises the
 outer-automorphism class of the image of the homeomorphism
 represented by $A_{q_0}$ under the natural quotient to the
 outer automorphism group.  The set of (equivalence classes of)
 such core transducers admits an easily computed product operation under
 which it is the group $\Onr$.

%There have been those who held e.g., that large, highly transitive groups of homeomorphisms of homogeneous spaces should have small (finite) outer automorphism groups (in this context, \emph{highly transitive} means $k$-transitive for any $k$, but restricted over some dense subset of the underlying space).  In this case, our groups satisfy the general criteria, but for all $n\geq 2$, it is the case that $\On$ is actually an infinite group for all $n\geq 2$.  Thus, this paper provides another set of examples that the ``meta-theorem'' just proposed is indeed simply false (this was already known to be the case, e.g., for $T_n$ and $F_n$ from the work in \cite{BrinCh,BGAut}).

We have the following theorem.
 \begin{theorem} \label{thm:InfiniteOut}
 For $n\geqslant2$, and $1 \le r < n$ a positive integer, the group $\Onr$, and so $\On$, is infinite.
 \end{theorem}

\subsubsection{Further related groups of interest}
For given $n$ (and $r$), the group $\On$ has some very interesting subgroups.  One of these subgroups is $\Ln$, which is the image in $\On$ of those homeomorphisms representable by bi-synchronizing transducers that have locally constant Radon--Nikodym derivative.  In $\Bnr$, these elements form the subgroup of homeomorphisms with bi-Lipschitz action on the Cantor space $\CCnr$, which group we denote as $\LBnr$. For $1 \le r < n$, let $\Lnr$ denote the image of the group $\LBnr$ under the quotient to $\Onr$. In particular, $\Lnr = \Ln \cap  \Onr$.  The group $\Ln$ contains a (combinatorially defined) further subgroup $\Hn$ of great interest. Similar to the above, we set $\Hnr:= \Hn \cap \Onr$. The elements of $\Hn$ correspond to non-initial core transducers which are not only bi-synchronizing, but also synchronous (one letter read in becomes one letter written out, on each transition), and which have some state that acts as a homeomorphism of Cantor space.    As will be shown in \cite{BleakCameronOlukoya}, the group $\Hn$ embeds naturally as the subgroup of the automorphisms of the full shift $\Aut(\nset^{\Z},\sigma)$ which are given by sliding block codes which use no future information. However, it is not hard to show that the image group mentioned is actually isomorphic to $\Aut(\nset^\omega,\sigma)$, the automorphisms of the one sided shift on $n$ letters (as suggested to us by J. Hubbard).  It follows that $\mathcal{H}_2$ is actually cyclic of order two by a classic result of Hedlund \cite{Hedlund69}.  In \cite{BleakCameronOlukoya},  the authors provide a new proof of this classic result using a close analysis of the automorphism types of specific quotients of de Bruijn graphs.

 Throughout the paper (but, predominantly in Subsection \ref{s_lipschitz}), we provide examples of various group elements of the various groups, by giving representative transducers.  While the definitions are sufficient to immediately prove the first point of the following theorem, the remaining points are proven through demonstrations of the existence of transducers representing various group elements with appropriate properties.

%\newpage
\begin{theorem}\label{thm:Examples}
Let $1\leqslant r<n$ be integers.  We have
 \begin{enumerate}
\item $\Hn\leqslant\Ln\leqslant\On$,
\item there are elements of $\On$ which are not in $\Ln$,
\item there are elements of $\Ln$ which are not in $\Hn$,
\item for $n>2$ there are elements of $\Hn$ of infinite order,
\item there are non-bi-Lipschitz torsion elements of $\mathcal{O}_2$, and
\item there are elements of $\mathcal{L}_2$ of infinite order.
 \end{enumerate}
\end{theorem}
%%%%%%%%%%%%%%%%%%%%%%%%%%%%%%%%%%%%%%%%%%%%%%%%%%%%%%%
\subsubsection{An example transducer}

 For those readers already comfortable with transducers, the following initial transducer (initial state $q_0$) of Figure \ref{fig_smallTrans} represents an element of $\Bnk{3}{2}$ with natural image in $\mathcal{L}_{3,2}$ non-trivial.
\begin{figure}[htbp]
\begin{center}
\begin{tikzpicture}[->,>=stealth',shorten >=1pt,auto,node distance=2.3cm,on grid,semithick,
                    every state/.style={fill=red,draw=none,circular drop shadow,text=white}]
  \node [state] (A)                                                                   {$q_0$};
  \node [state] (B)  [below= of A]                                                    {$q_1$};
  \node [state] (C)  [above= of A]                                                    {$q_2$};
  \node [state] (D)  [below left= of B]                                               {$q_3$};
  \node [state] (E)  [below right= of B]                                              {$q_4$};

\path (A) edge node {$\dot{1}/\dot{1}1$} (C);
\path (A) edge node [swap]{$\dot{0}/\varepsilon$} (B);
\path (B) edge [bend right=30] node [swap] {$1/\dot{1}0$} (C);
\path (B) edge node [swap]{$0/\dot{0}$} (D);
\path (B)  edge     node {$2/\dot{1}2$} (E);
\path  (C)  edge [in=105,out=75,loop]        node [swap]{$0/1$}(\CCnr);
\path  (C)  edge [bend right=35]                        node [swap] {$1/2$} (D);
\path  (C)  edge [bend left=70]                        node {$2/0$}(E);
\path  (D)  edge [in=240,out=210, loop]     node [swap]{$1/2$} (D);
\path  (D)  edge [bend left=70]      node {$0/0$} (C);
\path  (D)  edge                         node{$2/1$} (E);
\path  (E)  edge [in=330,out=300,loop]     node [swap] {$2/0$} (E);
\path  (E)  edge [bend left]      node {$0/2$} (D);
\path  (E)  edge [bend right=35]                node[swap]{$1/1$} (C);
\end{tikzpicture}
\end{center}
\caption{\label{fig_smallTrans}
An element of $\Bnk{3}{2}$}
\end{figure}

One can verify, for instance, that the states $q_2$, $q_3$, and $q_4$ are the states of the core, as reading any fixed word of length two in the alphabet $\{0,1,2\}$ from any state (other than $q_0$, which only takes words with first letter from the alphabet $\{\dot{0},\dot{1}\}$) will result in the same fixed state.

\subsection{Future directions}

The paper concludes with some open problems on which further work is needed.

There are some related articles \cite{BleakCameronOlukoya} and \cite{OlukoyaOrder}.  In those papers, the authors investigate many further properties of the groups $\Hn<\Ln<\On$.  There remain many mysteries.  For instance, the author of \cite{OlukoyaOrder} gives necessary and sufficient conditions for elements in $\mathcal{H}_n$, and therefore, for elements of $\Aut(\nset^\omega,\sigma)$, to have finite order. These conditions, however, do not yield a decision procedure.

It seems likely that the work in this article can be used to investigate the `exotic' automorphisms which arise for the subgroups $F_n<G_{n,1}$ for various values of $n$.  Indeed, it is relatively easy to build asynchronous finite transducers $\CCn\to\CCn$ which represent non-PL maps $[0,1]\to[0,1]$ which conjugate $F_n$ to $F_n$ in ways which cannot be realized by any inner automorphism of $F_n$ (for $n\geq2$).

\subsection{Acknowledgements}

We wish to thank the organizers of the conference ``Automorphism groups of topological Spaces, Bar Sheva, 2010'' who created the opportunity for some of the authors of this article to meet.   This meeting was fundamental for the genesis of this work.  We wish to thank Matthew G. Brin and Matatyahu Rubin for many interesting and insightful discussions related to this work, as well as for very careful readings of the manuscript in its various stages, which led to numerous improvements.  The first author would like to thank Daniel Lanoue for many enjoyable hours spent together studying the possible automorphisms of $V$ in 2007, which was the beginning of the first author's efforts towards answering the questions addressed in this paper.  The first and second authors wish to acknowledge support from EPSRC grant EP/R032866/1 received during the editing process of this article. The fourth author would like to thank St. Andrews University for its hospitality during the Workshop on the Extended Family of Thompson's Groups in 2014, and acknowledges the support of DySYRF (Anillo Project 1103, CONICYT) and Fondecyt's project 1120131. Finally, the fifth author was partly supported by Leverhulme Trust Research Project Grant RPG-2017-159.

%%%%%%%%%%%%%%%%%%%%
\section{Language, Notation, and Groups}
In this section, we will more carefully construct the Cantor spaces $\CCn$ and $\CCnr$ for given $1\leqslant r < n\in \N$, and the groups $\Gnr$.  We will also develop notation and equivalence classes for sets of points in $\CCnr$ which will be valuable in later discussions.  First, here are some general conventions we will follow.

Given topological spaces $X$ and $Y$, we will denote by $\Homeo(X,Y)$ the full set of homeomorphisms from $X$ to $Y$.  We set $\Homeo(X):=\Homeo(X,X)$, which becomes a group under composition, the group of self-homeomorphisms of the space $X$ (we will also use $H(X)$ for $\Homeo(X)$ when notation is too heavy, e.g., if $G\leqslant\Homeo(X)$ we will write $N_{H(X)}(G)$ for the normalizer of $G$ in the full group $\Homeo(X)$).   If $G \leqslant \Homeo(X)$, we say $\pair{X}{G}$ is a \emph{space-group pair}, and we use right actions to denote the natural action of $G$ on $X$.  In particular, if $x\in X$, $U \subset X$, and $g$,$h\in G$, then we write
%\marginpar{further on, the right side notation for group actions is used...}
$xg$ or $x \cdot g$ for the image of $x$ under the map $g$, we write
\[
Ug:=\{ug\mid u\in U\},
\]we write $g^h:=h^{-1}gh$ and $[g,h] := g^{-1}h^{-1}gh$.  We will extend the right action language above without too much concern when it is natural to do so, e.g., if $h:X\to Y$ is a homeomorphism, we will write $g^h$ to represent the self-homeomorphism $Y\to Y$ given by the rule $y\mapsto yh^{-1}gh$.

Throughout the remainder of this article, we will use the right-action notation outlined above.  Note also that from now on, we will assume at random times that we have some given $1\leqslant r<n\in\N$, so that we can refer to a space $\CCnr$ or a group $\Gnr$ without comment.  Occasionally we might still explicitly instantiate these constants.
\subsection{Cantor spaces revisited}

Recall from above that we regard the Cantor space $\CCn$ as the set of all infinite
$\nset$-sequences under the standard product topology.  That is, give the set $\Xn\seteq\nset$ the discrete topology, and then we have
\[\CCn=\Xn^{\omega}\]
 under the product topology.

Further, set $$\Wne:=\Xn^*$$ (the set of finite or empty words (sequences) over the alphabet $\Xn$, where we will always use the symbol $\varepsilon$ to denote the empty word over any alphabet), and $$\Wn := \Xn^+$$ (the set of finite non-trivial words in the alphabet $\Xn$).  If $\eta,\nu\in \Wne$ are so that $\eta$ is a prefix of $\nu$, we denote this by $\eta \leqslant \nu$, and we also write $\eta<\nu$ if $\eta$ is a proper prefix of $\nu$.

\begin{definition}
Let $\nu,\eta\in \Wne$. We say  that $\nu$ and $\eta$ are \emph{incomparable}
if $\nu\nleq \eta$ and $\eta\nleq \nu$, and we denote this as  $\nu\bot \eta$.
%We say that $\nu$ and $\eta$ are \emph{strongly incomparable}
%\marginpar{not sure this notion is used}
%if $\nu$ and $\eta$ are incomparable and there exists $\alpha\in \Wn$
%not equal to either of them such that in the lexicographic order of $\Wn$
%$\alpha$ is between  $\nu$ and $\eta$.
\end{definition}

\begin{remark}
Observe from the definitions that if $w\in \Wne$, then $\varepsilon\leq w$.  Specifically, $\varepsilon\perp w$ is not true for any $w\in \Wne$. \end{remark}

  Give the set $\rd:=\rset$ the discrete topology and set $\CCnr:= \rd\times \CCn$ (so that $\CCnr\cong \sqcup_{i\in\rd} \CCn$ and hence $\CCnr$ is a Cantor space).  Further, set $\Wnr:=\rd\times\Wn$, and $\Wnre:=\rd\times\Wne$.  By an abuse of notation, we will consider $\Wnr\subset \Wnre$ as sets of nontrivial finite words with first letter from the alphabet $\rd$ and any latter letters from the alphabet $\Xn$, and we extend the meaning of $\leqslant, <$, and $\bot$ to $\Wnre$.

For $\eta\in \Wnre\cup \Wne$ and  $\nu \in \Wnre \cup \Wne \cup \CCnr \cup\CCn$, define $\eta \concat \nu$ to be the concatenation of the sequences $\eta$ and $\nu$. (In practice, if $\nu\in \Wnre\cup\CCnr$, then we will have  $\eta$ is the empty word $\epsilon$.)  For $\eta\in \Wnre$, we set $U_{\eta} := \setm{\eta \concat x}{x \in \CCn}$, and we set $U_{\epsilon} := \CCnr$.
We call $U_{\eta}$ the \emph{cone of $\eta$} or the {\emph{cone at $\eta$}}.  In these cases, we refer to $\eta$ as the address of the cone $U_\eta$.
%and note that $U_{\eta}\subset\CCnr$.
Set $\Banr:=\setm{U_\eta}{\eta \in \Wnre}\cup\{U_\epsilon\}$.   The set $\Banr$ is a clopen basis
of the Cantor space $\CCnr$.

\begin{definition}\label{def:rotation}
If a word $w\in \Wn$ admits two non-trivial subwords $g$ and $h$ so that $w=g\concat h$, then we say $w$ is a \emph{rotation of the word $h\concat g$}.
\end{definition}

\begin{notation} \label{notn:subtraction}For $\eta\in \Wne\cup\Wnre$ and  $\nu \in \Wne\cup\CCn\cup\Wnre\cup\CCnr$, let $\eta \leqslant \nu$ mean that $\eta$ is a (possibly empty) prefix of $\nu$.  Furthermore, if $\eta\leqslant \nu$, there is $\tau\in \Wne\cup\CCn\cup\Wnre\cup\CCnr$ such that
$\eta\concat \tau=\nu$, and in this situation we define
$$\nu - \eta=\tau.$$
\end{notation}

Note that in the situation described by Notation \ref{notn:subtraction}, if $\tau\in \Wnre\cup\CCnr$, then in our practice we will have $\eta=\epsilon$.

\begin{definition}
Let $\veceta = \trpl{\eta_0}{\ldots}{\eta_{k - 1}} $ be in $ \Wnr^k$ for some integer $k\geq 1$.
We call $\veceta$ a \emph{antichain in $\Wnr$} if for any distinct $i,j  \in \{0, \ldots, k-1 \}$,
$\eta_i \not\leqslant \eta_j$.  Such an antichain will be called a \emph{complete antichain} if for every $\zeta \in \Wnre$,
there is $i < k$ such that
$\eta_i \leqslant \zeta$ or $\zeta \leqslant \eta_i$.
\end{definition}
     For an antichain $\veceta$ as in the definition above, we call the integer $k$ the \emph{length of $\veceta$}.  In general, we may also call a complete antichain in $\Wnr$ a \emph{prefix code}.  While an antichain is a tuple (implying an ordering), we use braces in our notation as most of the time we consider them simply as subsets of $\Wnre$.  The ordering is relevant to the definition of a \emph{prefix code map}, defined below in Subsection \ref{sec:GnrDef}.

%%%%%%%%%%%%%%%%%%%%%%%%%%%%%%%%%%%%%%

\subsection{The groups $\Gnr$ and some interesting subgroups} \label{sec:GnrDef}
We are also interested in building up complex maps from simpler ones.  A foundation for this is provided as below.
\begin{definition}
If $\eta,\zeta \in \Wnre$
define $\fnn{g_{\eta,\zeta}}{U_{\eta}}{U_{\zeta}}$ by
$$
(\eta \concat x)  g_{\eta,\zeta} := \zeta \concat x
$$  for all $x\in \CCn$.  For given $\eta$ and $\zeta$ as above,
we will refer to the map $g_{\eta,\zeta}$ as \emph{the basic cone map from $U_\eta$ to $U_\zeta$.}
\end{definition}

Using the cone maps above, and complete antichains in the poset $\Wnr$, we can define a specific class of homeomorphisms on the Cantor spaces $\CCnr$.

\begin{definition} \label{def:PrefixCodeMap}
Let $\veceta = \trpl{\eta_{0}}{\ldots}{\eta_{k - 1}}$ and $\veczeta = \trpl{\zeta_{0}}{\ldots}{\zeta_{k - 1}}$
be complete antichains in $\Wnr$ of length $k$ for some fixed positive integer $k$.
Define
\[
g_{\veceta,\veczeta} := \coprod_{i < k} g_{\eta_i,\zeta_i}.
\]
\end{definition}

In the context above, we may call such a map $g_{\veceta,\veczeta}$ a \emph{prefix code map}.  It is well known (and easy to check) that compositions and inversions of prefix code maps are prefix code maps, so for fixed $r$ and $n$, the prefix code maps over $\CCnr$ form a group of homeomophisms.  We denote the resulting group as $\Gnr$.

\begin{notation} Let
\begin{IEEEeqnarray*}{rCl}
\Gnr :=  \{g_{\veceta,\veczeta}&\mid& \veceta,\veczeta
\mbox{ are complete antichains } \\ & \mbox{of} & \mbox{ $\Wnre$ of the same length}\}.
\end{IEEEeqnarray*}
\end{notation}

We note in passing that in Higman's framework, the groups $\Gnr$ are groups of automorphisms of free term algebras.  The definition we give above is a standard translation from the original Higman definition of the groups $\Gnr$ to a definition of these groups as groups of homeomorphisms of spaces.  For a detailed discussion of the history behind these multiple points of view, we refer the reader to section 4B of \cite{BrownFinite}.

Now, from the definition given above for the elements of $\Gnr$, for every $g \in \Gnr$ there is a minimal integer $k$ and two finite sets of cones
$$\fsetn{U_{\eta_0}}{U_{\eta_{k - 1}}}\subseteq \Banr$$
and
$$\fsetn{V_{\zeta_0}}{V_{\zeta_{k - 1}}}\subseteq \Banr,$$
each partitioning $\CCnr$, so that $g$ maps each  $U_{\eta_{i}}$  to $V_{\zeta_i}$ as the basic cone  map $g_{\eta_i,\zeta_i}$.

By considering the loops in the transducer in Figure \ref{fig_smallTrans}, the reader can see that the transformation it represents cannot be a prefix-code map (these loops re-write infinite strings to infinite strings using entirely different letters).  Specifically,  the induced homeomorphism  must project to a non-trivial outer automorphism of $G_{3,2}$.  However, there are easier examples of the general fact that the group $\Bnr$ has homeomorphisms which naturally project to non-trivial elements of $\mathcal{O}_{n,r}$.  Here we provide a simple family of such homeomorphisms for the reader to verify the claim.

\begin{definition} \label{definition:pi-twist}
Let $\sigma \in S_n$ be any permutation on the set $\Xn$.
Define the map $\widehat{\sigma}:\CCn\rightarrow \CCn$ by the rule
$(a_{1}a_{2} a_3 \ldots) \mapsto (a_1\sigma\, a_2\sigma\, a_3\sigma \ldots)$, and call this map the \emph{$\sigma\text{-twist}$ of $\CCn$}.  Further, define the map $\widehat{\sigma}_{n,r}:\CCnr\to\CCnr$ which is obtained by applying $\widehat{\sigma}$ to each cone  $U_k\cong \CCn$ for $k\in \rset$, and call this map the \emph{$\sigma$-twist of $\CCnr$.}
\end{definition}

We observe in passing that the $\sigma$-twists of $\CCn$ and of $\CCnr$ are homeomorphisms of the respective Cantor spaces. (Indeed, these homeomorphisms are obtained as induced actions on the boundary of the standard infinite trees used in the construction of the Cantor spaces under consideration, where the $\sigma$-twists actually represent infinite automorphisms of these trees.) The relevance of $\sigma$-twists is shown by the next remark.

%The reader can verify the following by direct computation.

%\marginpar{a little argument here}

\begin{remark}[Existence of non-trivial $\Out(G_{n,r})$]\label{OutExists}\label{PiTwistsInOut}
Let $1\leqslant r<n\in\Z$, and suppose $\sigma \in S_n$.  The map $\widetilde{\sigma}_{n,r}:G_{n,r}\to G_{n,r}$ defined by $g\mapsto \widehat{\sigma}^{-1}_{n,r}\cdot g\cdot \widehat{\sigma}_{n,r}$, for $\,\widehat{\sigma}_{n,r}$ the $\sigma$-twist of $\CCnr$ as defined above, is an automorphism of $\Gnr$.  Furthermore, $\widetilde{\sigma}_{n,r}$ belongs to $\Inn(G_{n,r})$ if and only if $\sigma$ is the trivial permutation.
\end{remark}

The fact that $\widetilde{\sigma}_{n,r}$ is an automorphism of $\Gnr$ is a direct computation that we leave to the reader.  If $\sigma$ is not trivial, then the conjugacy action of $\widehat{\sigma}_{n,r}$ cannot coincide with that of an element $g\in\Gnr$ as otherwise conjugation by the product $g^{-1}\widehat{\sigma}_{n,r}$ (which is not in $\Gnr$) would induce the trivial action on $\Gnr$, while it is easy to see that for any given non-trivial homeomorphism $\rho$ of the Cantor space $\CCnr$, there are elements of $\Gnr$ which fail to commute with $\rho$.

\begin{definition}
We say that $g \in \Homeo(\CCnr)$ is \emph{densely canonical} if for every nonempty
open set $U \subseteq \CCnr$ there are words $\eta$, $\zeta\in\Wnre$ with cones $U_{\eta} \subseteq U$ and $U_\zeta$ so that $g | _{U_{\eta}} = g_{\eta,\zeta}$.
\end{definition}

Clearly, the set $\Gnrdcn$ of densely canonical homeomorphisms of $\CCnr$ is a group under composition.

\begin{definition}
We say that $g \in \Homeo(\CCnr)$ is \emph{pointwise canonical}
if for every $x\in \CCnr$
there are $\eta_{1},\eta_{2} \in \Wnre$
and $y \in \CCn$ such that $x=\eta_{1}\concat y$
and $xg=\eta_{2}\concat y$.
\end{definition}

Again, the set $\Gnrpcn$ of pointwise canonical homeomorphisms of $\CCnr$ is a group under composition.

\begin{remark}
For all integers $1\leqslant r<n$, we have $\Gnr\leqslant \Gnrdcn\cap\Gnrpcn$.
\end{remark}

We now give another definition of $\Gnrpcn$.
Define a relation on $\CCnr$ by:
\[
x \sim y \Leftrightarrow \exists  \eta,\nu \in \Wnre, z\in\CCn \textrm{ so that } x=\nu \concat z \textrm{ and } y=\eta \concat z.
\]
This is an equivalence relation on $\CCnr$.
We denote the set of equivalence classes of $\sim$ by $\dgfrac{\CCnr}{{\sim}}$,
and the equivalence class of $x\in \CCnr$ by $[x]_{\sim}$.
Given $x$, $y\in \CCnr$, we say \emph{$x$ and $y$ have an equivalent tail} if and only if $y\in[x]_{\sim}$,  and we call $[x]_{\sim}$ the \emph{tail class of $x$}.
Let $H_{n,r,\sim}$ be the group of all homeomorphisms of $\CCnr$ that preserve $\sim$.
That is,
\[
H_{n,r,\sim}=\{g\in \Homeo(\CCnr)\,\mid\,y\in [x]_{\sim} \Leftrightarrow yg\in [xg]_{\sim}\}.
\]
Clearly, $\Gnrpcn$ is a subgroup of $H_{n,r,\sim}$.  By definition, $H_{n,r,\sim}$ has an induced action on $\dgfrac{\CCnr}{{\sim}}$, and from this point of view it is immediate that  $\Gnrpcn$ represents the kernel of the action of $H_{n,r,\sim}$ on $\dgfrac{\CCnr}{{\sim}}$.  That is, $\Gnrpcn\vartriangleleft H_{n,r,\sim}$, and
\[
\Gnrpcn= \{g\in H_{n,r,\sim}(\CCnr)\,\mid\, \forall x\in\CCnr,\, xg\in[x]_{\sim}\}.
\]
We say that \emph{$\Gnrpcn$ fixes $\sim$ pointwise}.

We observe the following easy lemma.

\begin{lemma}
For all $x\in\CCnr$, the group $\Gnr$ acts transitively on $[x]_\sim$.
\end{lemma}

\begin{proof}
For each $y\in [x]_{\sim}$ there are $\nu, \eta\in \Wnre$ and $z\in \CCn$ such that
$x=\nu\concat z$ and $y=\eta\concat z$. There exist two complete antichains
$\veceta = \trpl{\eta_{0}}{\ldots}{\eta_{k - 1}}$
and $\vec{\nu} = \trpl{\nu_{0}}{\ldots}{\nu_{k - 1}}$
such that $\eta=\eta_{0}$ and $\nu=\nu_{0}$, and hence $xg_{\vec{\nu},\veceta}=y$.
\end{proof}

%%%%%%%%%%%%%

\section{More on $\Gnrdcn$, $\Gnrpcn$, $\Gnr$ and $\Hnrsim$}\label{sec:Rubin}

If $X$ is a topological space and $x \in X$,
then we let $\rfs{Nbr}^X(x)$ denote the set of open neighbourhoods
of $x$ in $X$. A subset $A \subseteq X$ is {\it somewhere dense}, if for
some nonempty open set $U \subseteq X$, the intersection $A \cap U$
is dense in $U$.  The following is a version of the main result in \cite{Rubin}.

\begin{theorem}\label{t2.3} {\rm (M. Rubin)}
Let $\pair{X}{G}$ and $\pair{Y}{H}$ be space-group pairs.
Assume that $X$ is Hausdorff, locally compact, and without isolated
points, and that for every $x \in X$ and $U \in \rfs{Nbr}^X(x)$, the set
$\setm{xg}{g \in G \mbox{ and } g|_{(X - U)} =
\rfs{Id}|_{(X-U)}}$
is somewhere dense.
Assume that the same holds for $\pair{Y}{H}$.
Suppose {we have a group isomorphism} ${G}\cong_{\phi}{H}$.
Then there is $\varphi \in \Homeo(X,Y)$
such that $\varphi$ induces $\phi$.
That is,
$g\phi = g^{\varphi}$ for every $g \in G$.
\end{theorem}

Below, we refer to the homeomorphism $\varphi$ in the statement of Rubin's Theorem as a \emph{Rubin conjugator}.

\begin{cor}\label{cor:NormAut}
$\Anr\cong N_{H(\CCnr)}(\Gnr)$.
\end{cor}
\begin{proof}
An automorphism of $\Gnr$ is an isomorphism from $\Gnr$ to itself, so we set $X=Y=\CCnr$ in the statement of Rubin's Theorem, and we see that $\Anr$ is a quotient of $N_{H(\CCnr)}(\Gnr)$.  If $g\in H(\CCnr)$ is non-trivial, then there is some open set $U$ that is moved entirely off itself by $g$, so we can find an element $t$ of $\Gnr$ that acts simply as a transposition of two disjoint basic open cones within $U$.  It is immediate that $g$ and $t$ do not commute, and it follows that the kernel of the conjugation action of $H(\CCnr)$ on $\Gnr$ is trivial.  \end{proof}

%\vspace{.1 in}

The following lemma is standard in the literature which relies upon Rubin's theorem, for example \cite{brinHigherV,BleakLanoue,NekrashevychIMG}.
\begin{lemma} \label{lemma:tran}
Let $X$ be a topological space, $\sim$ an equivalence relation on $X$,
$H_\sim$ the subgroup of $\Homeo(X)$
consisting of all the homeomorphisms that preserve $\sim$ and
$G$ a subgroup of $H_\sim$ fixing all equivalence classes of $\sim$ and
acting transitively on each equivalence class.
Then  $\{ h\in \Homeo(X)\mid h^{-1}Gh\subseteq G \} \subseteq H_\sim.$
In particular,  $N_{H(X)}(G)\leqslant H_\sim$.
\end{lemma}

\begin{proof}
Let $h\in \Homeo(X)$ so that $G^h\subset G$
and suppose $x\in X$ and $y\in [x]_\sim$.
Since $G$ acts transitively on $[x]_\sim $, there exists $g\in G$ such that $xg=y$.
Calculating, we have  $xhg^h=xhh^{-1}gh=yh$.  As $g^h\in G$, we see that $xh\sim yh$,
and since $x$ is an arbitrary element of $X$, we conclude that $h\in H_\sim$.
\end{proof}

%Conversely, let $h\in H$, $g\in G$ and $x\in X$.
%Then $h^{-1}(x)\sim g(h^{-1}(x))$ so $h(h^{-1}(x))\sim h(g(h^{-1}(x)))$.
%Thus $x\sim hgh^{-1}(x)$. So $hgh^{-1}\in G$ and therefore$h\in N_{\Homeo(X)}(G)$ .

\begin{cor}
\label{cor:2} We have:
 \begin{enumerate}
\item\label{cor:tail-class-preserved}$\Anr\cong N_{H(\CCnr)}(\Gnr)\leqslant \Hnrsim $

\item $\Aut(\Gnrpcn)\cong \Hnrsim$.
 \end{enumerate}
\end{cor}

\begin{proof} Point $(1)$ follows directly from Corollary \ref{cor:NormAut} and Lemma \ref{lemma:tran}.
For Point (2), notice that by the Lemma \ref{lemma:tran} we have  $N_{H(\CCnr)}(\Gnrpcn)\leqslant \Hnrsim$.
Moreover, if $h\in\Hnrsim$, $g\in \Gnrpcn$ and $x\in \CCnr$, then $xh^{-1}\sim xh^{-1}g$ so
$xh^{-1}h\sim xh^{-1}gh=xg^h$.
Thus $x\sim xg^h$, and as $x$ is arbitrary in $\CCnr$, we see that $g^h\in \Gnrpcn$, and therefore $h\in N_{H(\CCnr)}(\Gnrpcn)$. \end{proof}

The above Corollary \ref{cor:2}.\ref{cor:tail-class-preserved} will be used in Section \ref{sec:finitely-many-local-actions} to find conditions on automorphisms of $\Gnr$ (as homeomorphisms of $\CCnr$), enabling us to show those particular homeomorphisms can be represented by finite transducers.  The following Lemma and Proposition are of independent interest, but are not used later.

\begin{lemma}
Let $g\in \Gnrpcn$ and $\nu,\eta\in \Wnr$.
Then the set
$$A_{g,\nu,\eta} := \setm {x\in \CCnr}{x=\nu\concat y \mbox{ and }
x g=\eta\concat y \mbox { for some } y\in \CCn}$$
is closed.
\end{lemma}

\begin{proof}
Suppose that $\{x_i \}_{i=0}^\infty \subseteq A_{g,\nu,\eta}$
is a sequence that converges to $x\in \CCnr$.
For every large index $i$ there is $y_i\in \CCn$ such that $x_i=\nu\concat y_i$
and $x_i g=\eta\concat y_i$. So as $x_i \rightarrow x$, we have $\nu\concat y_i \rightarrow \nu\concat y=x$
for some $y\in \CCnr$ such that $y_i \rightarrow y$.
But $g$ is continuous, so $x_i g \rightarrow x g$, that is $\eta\concat y_i \rightarrow x g$.
Therefore, $\eta\concat y = x g$, hence $x \in A_{g,\nu,\eta}$. \end{proof}

\begin{prop}
$ \Gnrpcn=\Hnrsim \cap \Gnrdcn$.
\end{prop}

\begin{proof}
Let $g\in \Gnrpcn$ and let $U\subseteq \CCnr$ be an open set.
For every $\nu,\eta \in \Wnr$ set
$A_{g,\nu,\eta}^U=A_{g,\nu,\eta}\cap U$. Since $A_{g,\nu,\eta}$ is closed in $\CCnr$, we have
that $A_{g,\nu,\eta}^U $ is closed in $U$. Since $g\in \Gnrpcn$, we also have
$U=\bigcup _{\nu,\eta \in \Wnr} A_{g,\nu,\eta}^U $. Now, $U$ has the Baire property
and there are countably many $A_{g,\nu,\eta}^U $'s in the above union so there exists some $\nu$ and $\eta$ so that
$A_{g,\nu,\eta}^U $ is somewhere dense. Let $U_{\zeta}\subseteq U$ be a basic cone so that
$A_{g,\nu,\eta}^U \cap U_{\zeta}$ is dense in $U_\zeta$. But $A_{g,\nu,\eta}^U \cap U_{\zeta}$
is closed in $U_{\zeta}$  so  $A_{g,\nu,\eta}^U \cap U_{\zeta}=U_\zeta$ and therefore
$U_{\zeta}\subseteq A_{g,\nu,\eta}^U $. Let $\tau \in \Wne$ be such that
$\nu \concat \tau =\zeta$ and set $\omega=\eta \concat \tau$.
For every $\zeta\concat x\in U_{\zeta}$ we have
 $$(\zeta\concat x)g=(\nu \concat \tau \concat x)g=
\eta\concat \tau \concat x=\omega\concat x.$$
Thus, we have found a cone $U_\zeta\subset U$ and $\omega \in \Wne$
so that $g |_ {U_{\zeta}} = g_{\zeta,\omega}$ and hence $g\in \Gnrdcn$. Therefore, $ \Gnrpcn \leqslant \Hnrsim \cap \Gnrdcn$.

Conversely, let $g \in \Hnrsim \cap \Gnrdcn$.  Since $g \in \Gnrdcn$, there exist $\nu,\eta \in \Wnr$ such that $g |_{ U_{\nu}} = g_{\nu,\eta}$.
As every tail class has a representative in $U_{\nu}$ and $g \in \Hnrsim$,  we have that $g$ fixes $\sim$ pointwise. Therefore,
$\Hnrsim \cap \Gnrdcn \leqslant \Gnrpcn$. \end{proof}

%%%%%%%%%%%%%%%%%%%%%%%%%%%%%%%%%%%%%%%%%%%%%%%%%%%%%%

\section{Local Actions}

Given $h\in \Homeo(\CCnr)$ and $U_{\nu}\in \Banr$, we would like
to investigate the action of $h$ on $U_{\nu}$. We begin by establishing some notation for
$h\in \Homeo(\CCn)$ and $U_\nu\in \Ban$.  We will then briefly discuss how to generalize what is
established for homeomorphisms of $\CCnr$.

For $U\subseteq \CCn$, define the \emph{root} of $U$ to be $\nu \in \Wne$
such that $U\subseteq U_{\nu}$ and $U_\nu$ is the minimal
element in $\Ban$ (with respect to inclusion) with this property.
Denote $$\nu := \Root(U).$$

\begin{definition}

Let $h:\CCn\longrightarrow \CCn$ be a continuous function.
The root function of $h$ is the function $\theta_h:\Wne \longrightarrow \Wne$ defined by:

$$\nu\mapsto \nu \theta_h=\Root(U_{\nu}h)$$
for all $\nu\in\Wne.$
\end{definition}

\begin{definition}
Let $h:\CCn\longrightarrow \CCn$ be continuous and injective, and $\nu\in \Wne$. The \emph{local action of $h$ on $\nu$}
is the injective continuous map $h_{\nu}:\CCn\longrightarrow \CCn$ defined by
$x \cdot h_{\nu}=y$, where  $(\nu\concat x) \cdot h=(\nu\cdot\theta_h) \concat y.$
\end{definition}

%We leave verification of the following points to the reader.

\begin{remark}
Given any homeomorphism $h:\CCn\to\CCn$ and $\nu\in \Wne$, it is easy to verify that:
 \begin{enumerate}
\item the local action of $h$ on $\nu$ is well defined,

\item $U_{\nu} \cdot h\subseteq U_{\nu\cdot\theta_h}$,

\item  $h_{\nu}$
is both continuous and injective,

\item  $h_{\nu}$ is surjective (and hence $h_{\nu}\in \Homeo(\CCn)$)
if and only if $U_{\nu}\cdot h=U_{\nu\cdot \theta_h}$, and

\item  if $\nu\concat \eta\in \Wne$,  then
$(\nu\concat \eta)\theta_{h}=\nu\theta_{h}\concat (\eta\theta_{h_{\nu}})$.
 \end{enumerate}
\end{remark}

We now discuss how to generalize these definitions to a homeomorphism $h$ from $\CCnr$ to itself.

  Let $P_h\subset \Wnre$ be the unique maximal set of strings such that:
 \begin{enumerate}
\item if $\nu\in P_h$, then $(U_\nu)h$ is contained in a specific cone $B_a\in \Banr$ for some $a\in \rd$, and
\item for any proper prefix $\mu$ of some $\nu \in P_h$, there are $a_1\neq a_2\in \rd$ and $x$, $y\in \CCn$ so that $\{a_1\concat x, a_2\concat y\}\subset (U_\mu)h$.
 \end{enumerate}
 In the case that $r=1$, we note in passsing that $P_h=\{\varepsilon\}$, since if there were other elements in $P_h$ we would not be able to satisfy the second property, as that are not two distinct elements of $\rd$.

Observe by the continuity of $h$ and compactness of $\CCnr$ that $P_h$ is a finite set which makes a complete antichain for $\Wnre$.  If $r>1$, then all members of $P_h$ are non-trivial words.  Now define, for any $\mu$ a proper prefix of an element of $P_h$, that $(\mu)\theta_h := \varepsilon$, so that the local action $h_\mu$ is a continuous injective map $h_{\mu}:\CCn\longrightarrow \CCnr$ (or from $\CCnr$ to $\CCnr$ if $\mu=\varepsilon$).  For each element $\nu\in P_h$, we set $(\nu)\theta_h$ to be the unique maximal common prefix of all the points in $(U_\nu)h$. We observe that for each element $\nu\in P_h$, the resulting string $(\nu)\theta_h$ will be a non-empty string beginning with a letter in $\rd$ and so there is an induced local action $h_\nu:\CCn\to\CCn$.
 Finally, for $\nu\in P_h$ and all $\xi = \nu\concat \rho$,  set $(\xi)\theta_h:=\nu\theta_h\concat \rho\theta_{h_\nu}$, where $h_\nu:\CCn\to\CCn$ is the local action of $h$ at $\nu$, as mentioned in the previous sentence.

Essentially, one needs a little care with the definition of local actions for a homeomorphism $h\in \Homeo(\CCnr)$, to handle the specific cases arising from the set of letters $\rd$. Specifically, there is a unique local action associated that is a map from $\CCnr$ to $\CCnr$, there are finitely many associated local actions which are maps from $\CCn$ to $\CCnr$, and finally the remaining associated local actions are maps from $\CCn$ to $\CCn$.

%%%%%%%%%%%%%%%%%%%%%%%%%%%%%%%%%%%%%%%%%%%%%%%%%%%%%%%%%%%%%%%%%

\section{On the Rational Group $\mathcal{R}_n$ and the Related Groups $\Rnr$}

In this section we will specify the group $\Rnr$, which is akin to the group $\Rn$ of GNS in \cite{GNSenglish}. For the most part, we follow the ideas and definitions of \cite{GNSenglish}.

\subsection{Defining standard transducers and associated notation}

\begin{definition} A \emph{transducer}
is a tuple $A=\left\langle  X_i,X_o,Q,\pi,\lambda \right\rangle$, where:
 \begin{enumerate}
\item $X_i$ is a finite alphabet, the \emph{input alphabet},

\item $X_o$ is a finite alphabet, the \emph{output alphabet},

\item $Q$ is a set, the \emph{set of states},

\item $\pi : X_i\times Q\longrightarrow Q$ is a mapping, the \emph{transition function}, and

\item $\lambda: X_i\times Q\longrightarrow X_o^*$ is a mapping, the \emph{output function}.
{
(Recall that $X_o^*$ is the set of all finite strings in the alphabet $X_o$; that is, ``*'' is the ``Kleene star'' operator.)}
 \end{enumerate}
An \emph{initial transducer} is a tuple $A_{q_0} =
\left\langle  X_i,X_o,Q,\pi,\lambda,q_0 \right\rangle$, where
$A=\left\langle  X_i,X_o,Q,\pi,\lambda \right\rangle$ is a transducer and
$q_0\in Q$.
\end{definition}

The word ``transducer'' is used since transducers are meant to model machines which transform inputs to outputs in a controlled fashion.  In particular, one is to imagine that there is always a specified ``active state'' (say $q\in Q$ for the purposes of this discussion) from which the transducer will process its input.  The transducer then ``processes from $q$'' as follows.  First, it reads a letter $a\in X_i$ from an input tape, and then it performs two actions.  These actions are:
 \begin{enumerate}
\item the active state changes to the state $\pi(a,q)$, and
\item the transducer writes the word $\lambda(a,q)$ to an output tape.
 \end{enumerate}
Thus, transducers transform input strings to output strings. An initial
transducer includes the specification of the initial state from which
processing starts.

The functions $\lambda$ and $\pi$ can be extended to the set $(X_i^*\backslash\{\varepsilon\})\times Q$
according to the following recurrence rules:
$$\pi(\mu\concat \nu,q) := \pi(\nu,\pi(\mu,q)) \quad \textrm{and} \quad
\lambda(\mu\concat \nu,q) := \lambda(\mu,q)\concat\lambda(\nu,\pi(\mu,q)),$$
where $\mu\in X_i$ and $\nu\in X_i^*$.  \label{infExtensions}Now further extend the definition of $\lambda$ so that it is defined on inputs from
 $X_i^{\omega}$ in the obvious manner.
Note that an infinite string might still be transformed to a finite string if while processing the infinite string we at some stage visit a state $q$ with no output on the next letter, and from then on, always visit states that give no output as we process the remaining letters of the input string.

In all that follows below, we will assume that our transducers are constructed in such a way as to never transform an infinite string to a finite string.  In this way, for any state $q\in Q$, we see that $\lambda(\cdot,q):X_i^\omega \to X_o^\omega$ will always represent a continuous map which map we will denote as $h_{A_q}:X_i^\omega \to X_o^\omega$.
%%%%%%%%%%%%%%%%%%%%%%%%%%%%%

\subsection{Some technicalities for transducers}
Let us now consider a fixed initial transducer $A_{q_0}\seteq\left\langle  X_i,X_o,Q,\pi_A,\lambda_A,q_0 \right\rangle$.
%We will now discuss some general aspects of $A$ and $A_{q_0}$.
For both $A$ and $A_{q_0}$, we say the transducer is \emph{finite} whenever $|Q|<\infty.$

Now consider the transducer $A_{q_0}$, and consider the continuous map $h = h_{A_{q_0}} :X_i^\omega\to X_o^\omega$ induced by $A_{q_0}$.  Recall that for any word $\tau\in X_i^*$ we use the notation $h_\tau$ to represent the local action of $h$ at $\tau$, while $\tau \theta_h = \Root(U_\tau h)\in X_o^*$.  Let $q\in Q$ and $\nu\in X_i^*$ be such that $\pi_A(\nu,q_0)=q$. If $\lambda_A(\nu,q_0)$ is a proper prefix of $\nu\theta_h$, then $q$ is considered a \emph{state of incomplete response}, as the infinite outputs from $q$ will have some non-empty guaranteed common prefix.
Also, if for some $q\in Q$ and for every $\nu \in X_i^*$
we have $\pi(\nu,q_0)\ne q$, then we say \emph{$q$ is an inaccessible state}, noting that such a state $q$ could never have any bearing on the definition of the map $h$.  Finally,
if two accessible states $q_1,q_2\in Q$ are so that $h_{A_{q_1}}=h_{A_{q_2}}$, then we say that $q_1$ and $q_2$ are \emph{equivalent states} (these states would be called \emph{$\omega$-equivalent} in \cite{GNSenglish}).

We say the transducer $A_{q_0}$ is \emph{minimal} if $A_{q_0}$ has no states of incomplete response, has all states accessible, and whenever $q_1,q_2\in Q$ are equivalent states, we have $q_1=q_2$.

Aside from the definition of $\omega$-equivalence above, there is also a more categorical definition of equivalence of transducers.  We say two transducers on a fixed alphabet $A$ are \emph{strongly isomorphic transducers}, denoted by $=_{si}$, if there is a bijection between the states of those transducers so that this bijection on states commutes with the transition and output functions of the two transducers. Note that this notion represents an equivalence of transducers which is independent of an initial state.  E.g., two strongly isomorphic transducers, both reading some input alphabet $A$, when assigned non-corresponding initial states would generally represent different functions on $A^{\infty}$.

We note that a general transducer is \emph{asynchronous} in the sense that reading any input letter from some state, we have no guarantee that the output of the function $\lambda$ will be a single letter.  In the special case that for all states $q\in Q$ of the transducer, and for all input letters $x\in X_i$, we have a guarantee that $|\lambda(x,q)|=1$, then we say the transducer is \emph{synchronous}.  Thus, if we do specify that a transducer is asynchronous, then this generally means that we are reminding the reader that (at that time) we have no guarantee that the output of the function $\lambda$ has length $1$.

One can represent transducers as directed labelled graphs as follows.  Let $A$ be a transducer as above.  Our graph $\Gamma_A$ will have vertex set the set $Q$, and a directed edge from $q$ to $\pi(a,q)$ for each pair $(a,q)\in X_i\times Q$, and where we label this edge with the string ``$a/w$'' where $w = \lambda(a,q)$ is a word in $X_0^*$.

\subsection{Transducers acting on $\CCnr$}\label{Subsection:transducersCCnr}
We will use transducers to model self-homeomorphisms of the spaces $\CCnr$.  This introduces various issues that will cause us to slightly modify our definition of transducers.  Above, we already extended the definition of $\lambda$ to receive infinite inputs.  We still need to discuss what needs to be done to represent the fact that the letters in $\rd$ only appear once in a string  representing a point in $\CCnr$.  Another technical issue arises in the case $r=1$, which would make the upcoming definitions extremely unwieldy.

Thus, we can specify an initial transducer $A_{q_0}$ as above with alphabets $X_i=X_o=\rd\sqcup\{0,1,2,\ldots,n-1\}$, and we can ask when this transducer induces a continuous map $h:\CCnr\to\CCnr$ as above (perhaps even a self-homeomorphism).  Such a transducer will actually map the Cantor space $X_i^\omega$ to $X_i^*\sqcup X_i^\omega$, which is a much bigger space than we need if we want to restrict to a map from $\CCnr$ to $\CCnr$.  We thus wish to consider a class of initial transducers which are found by restricting these larger transducers to their actions on $\CCnr$ (and only allow such transducers which turn such infinite inputs into infinite outputs).  Of course, we would like an internal characterisation of these restricted transducers, which we provide below.

An initial \emph{transducer for $\CCnr$} will be a tuple $$A_{q_0}=(\rd,\Xn,R,S,\pi,\lambda, q_0)$$ so that:
 \begin{enumerate}
\item $R$ is a finite set, and $S$ is a (finite or countably infinite) set disjoint from $R$; in the framework above,
$Q:= R\sqcup S$ is the \emph{set of states},

\item $q_0$ belongs to $R$ and is the \emph{initial state},

\item $\pi$ is piecewise defined function (the \emph{transition function}) taking an input letter and a state, and producing a state according to rules, and

\item $\lambda$ is a piecewise defined function (the \emph{output function}) taking an input letter and a state, and producing a (empty, finite, or infinite) output string.
 \end{enumerate}

The domain of $\pi$ and $\lambda$ is
$$\left( \rd\times\left\{q_0\right\}\right) \bigsqcup \left(\left\{0,1,\ldots,n-1\right\} \times \left( Q\backslash\left\{q_0\right\} \right)\right),$$ the range of $\pi$ is $Q\backslash\{q_0\}$, and the range of $\lambda$ is $\Wne\sqcup\CCn\sqcup\Wnre\sqcup\CCnr$.

The functions $\pi$ and $\lambda$
{(as well as the extended version of $\lambda$ as described above)}
will obey some further rules, as follows:
 \begin{enumerate}
\item whenever $\pi(x,q_1) = q_2 \in R$, we have $q_1\in R$ and $\lambda(x,q_1)=\varepsilon$,

\item whenever $q_1\in R$ and $\pi(x,q_1)=q_2$, with $q_2\in S$, we have $\lambda(x,q_1)\in \Wnr$,

\item if $q\in S$, then for all $x\in \Xn$ we have $\lambda(x,q)\in \Wne$ and $\pi(x,q)\in~S$,
%\vspace{.1 in}

%\centerline{{\it (The next two rules apply to the extended version of $\lambda$)}}

\item if $q\in Q$ and $w$ is a non-empty word so that $\pi(w,q)=q$, then $\lambda(w,q)\neq \varepsilon$, and

\item whenever $q\in S$ and $x\in \CCn$, we have $\lambda(x,q)\in \CCn$.
 \end{enumerate}

The  core idea of these conditions is that the letters in $\rd$ cannot appear in the output from a single input string more than one time, and then just at the beginning of the total output string given any initial input string.  Notice the last two conditions guarantee that $A_{q_0}$ will induce a function $h_{A_{q_0}}:\CCnr\longrightarrow\CCnr$, since no point (which by definition is given as an infinite string) will be transformed into a finite output string.  Also, the penultimate condition guarantees that the graph underlying the transducer admits no directed cycles in the states in $R$.  If either of the last two conditions fails, the transducer is \emph{degenerate}; we will call a transducer as above satisfying all of these conditions
\emph{non-degenerate}.

\begin{comment}
\begin{remark}
It is easy to check that any transducer for $\CCnr$ can be modified to become a transducer for $(\rd\sqcup\Xn)^\omega$ by adding a state $z$ to go to when one receives an input string with a second occurrence of a letter in $\rd$ (from any state in the transducer except $q_0$).  By using the identity function for the output of such a transition, and then from having the rules $\pi(x,z) = z$ and $\lambda(x,z) = x$ for all $x\in \rd\sqcup \{0,1,\ldots ,n-1\}$, the transducer so created will produce a homeomorphism for $(\rd\sqcup\Xn)^\omega$ whenever $h_{A_{q_0}}:\CCnr\to\CCnr$ is also a homeomorphism.
\end{remark}
\end{comment}

%%%%%%%%%%%%
\subsection{Transducers, continuity, and $\omega$-equivalence\label{omega}}

Suppose we are given a non-degenerate initial transducer $A_{q_0}$ of one of the two varieties above.  In the first case, our transducer induces a map
{$h_{A_{q_0}}: \CCn \to \CCn$,}
and in the second case, a map $h_{A_{q_0}}:\CCnr\to\CCnr$.  It is very easy, as mentioned above, to prove in either case that the map $h_{A_{q_0}}$ is continuous, and we say that $A_{q_{0}}$ represents the continuous map $h_{A_{q_0}}$.

Note that if a continuous function $h:\CCn\to\CCn$ or $h:\CCnr\to\CCnr$ can be represented by an initial transducer $A_{q_0}$ (and we will show below that it can, following a discussion of \cite{GNSenglish}), then $A_{q_{0}}$ is not unique amongst all of the transducers which represent $h$.  Thus, two initial transducers are said to be \emph{$\omega$-equivalent} if they represent the same continuous function.  Note that on occasion we wish to consider whether two states of (possibly distinct, and perhaps non-initial) transducers are themselves $\omega$-equivalent.  By this we mean taking these states as initial states of the associated transducers, and determining if the resulting maps are equal.

We now trace through the basics of a construction in \cite{GNSenglish} but in our context, to show how to build a tranducer representing a homeomorphism from $\CCnr$ to $\CCnr$.    The construction we describe can be used to build a transducer representing any continuous map from $\CCnr$ to itself, although in this more general case, the construction might result in a transducer with a state from which one might read in one input letter and write as output a point in Cantor space.  The construction below, even in the case of homeomorphisms, can result in a transducer where the set $S$ (and so, $Q$) is infinite.

Let $h\in \Homeo(\CCnr)$.  We inductively define a non-degenerate initial transducer  $$\widetilde{A}_{q_{0}}\seteq(\rd,\Xn,R,S,\pi,\lambda,q_0)$$ representing $h$ as follows:
 \begin{enumerate}
\item set $Q := \Wnre\sqcup \{\varepsilon\}$,
\item for $a\in \rd$ define $\pi(a,\varepsilon) := a$,
\item for $a\in \rd$ define $\lambda(a,\varepsilon) := a \theta_h$,
\item for $\nu\in Q\backslash\{\varepsilon\}$ and $x\in \Xn$ define $\pi(x,\nu) := \nu \concat x$,
\item for $\nu\in Q\backslash\{\varepsilon\}$ and $x\in \Xn$ define $\lambda(x,\nu) := (\nu \concat x) \theta_h -\nu\theta_h$,\label{item:lambdaDef}
\item set $R\subset Q$ to be the set of states $\nu\in \Wnre\sqcup\{\varepsilon\}$ for which $\nu\theta_h=\varepsilon$,\label{pt:defR}
\item set $S\seteq Q\backslash R$,
\item set $q_0 := \varepsilon\in Q$ as the initial state.
 \end{enumerate}
 In the case that $r=1$, replace point (\ref{pt:defR}) with the following:
\begin{enumerate}
\setcounter{enumi}{5}
\item set $R=\{\varepsilon\}$ (note that in this case, $\lambda(\dot{0},\varepsilon)$ has prefix $\dot{0}$),
\end{enumerate}
\begin{remark}
The following points are worth mentioning.  The first is important notationally, whilst the second is a technical point relating to our transition from the general rational group to the ``rational group acting on $\CCnr$.''    The third point will be discussed more fully in the paragraphs to follow, while the fourth point is to support the reader in recalling what happens over the more standard space $\CCn$.
 \begin{enumerate}
\item It is immediate by construction that the initial transducer $\widetilde{A}_{q_0}$ represents $h$.  That is, $h=h_{\widetilde{A}_{q_0}}$.
\item Using the inductive extension of $\lambda$ to the domain $\Wnre\sqcup \{\varepsilon\}$, we have
$$\lambda(x,\nu) = (\nu\concat x)\theta_h -\lambda(\nu,\varepsilon)$$ and also $$\lambda(\nu\concat x,\varepsilon)=(\nu\concat x)\theta_h$$ (we take $\lambda(\varepsilon,\varepsilon)=\varepsilon$ to found the induction).
\item Even though $\widetilde{A}_{q_0}$ is a non-degenerate transducer with no inaccessible states and no states of incomplete response, it might not be minimal, and it might not be synchronous.
\item In the case of transducers representing functions from $\CCn$ to $\CCn$, there would be no set of letters $\rd$ and no set $R$ of states, so that the resulting transducer would be listed as a quintuple, with $Q=\Wne$.
 \end{enumerate}
\end{remark}

We discuss the third point above.  First, note as before that there is no general algorithm to transform the above transducer $\widetilde{A}_{q_0}$ to a synchronous $\omega$-equivalent transducer.  For instance, many homeomorphisms of $\CCnr$ are not induced by automorphisms of a forest of $r$ infinite $n$-ary rooted trees.  In general, though, we can still reduce $\widetilde{A}_{q_0}$ to a minimal transducer $A_{q_0}$ as in the subsection \ref{s_reduction} below.
\begin{remark}
There is a bijection between the states of the reduced transducer $A_{q_0}$ and the set of local actions of $h$ (including the local action of $h$ at the empty input). Thus, while $\widetilde{A}_{q_0}$ is an infinite transducer, $A_{q_0}$ might be finite.  Finally, we note that if $\widetilde{A}_{q_0}$ represents a homeomorphism of $\CCnr$, then given a letter-state pair $(x,q)$ in the domain of $\lambda$, we have that $\lambda(x,q)$ is always a finite word. (Note that, for instance, for a continuous function $f:\CCn\to\CCn$ which sends an entire cylinder set $U_\mu$ to an eventually periodic point, the transducer produced by following the construction above may be finite, whilst the guaranteed output after some finite input might still be an infinite word; but such a transducer would not represent an injective function.)\end{remark}

The following theorem is direct consequence of Theorem 2.5 of \cite{GNSenglish}.

\begin{theorem} [GNS]\label{GrigThm}
Let $1\leqslant r<n$ be integers. Any homeomorphism $h:\CCnr\longrightarrow \CCnr$ can be represented by a finite non-degenerate initial transducer if and only if the set of local actions of $h$ is finite.
\end{theorem}

%%%%%%%%%%%%%%%%%%%%%%%%%%%%%%%%%%%%%%%%%%%%%%%

\subsection{Reducing transducers\label{s_reduction}}

When given a non-degenerate initial transducer $$A_{q_{0}}=(Q_A,X,\pi_A,\lambda_A,q_0)$$ recall that we say that $A_{q_0}$ is \emph{minimal}
if all of its states are accessible, it has no states of incomplete response, and any two distinct states represent distinct local maps.

In this section we describe the process of reduction that turns an initial transducer into a minimal transducer, following the discussion of \cite{GNSenglish}.  Note that the proofs in \cite{GNSenglish} are given with a view to transducers representing self-homeomorphisms of $\CCn$, but the processes are effectively unchanged in our context of self-homeomorphisms of $\CCnr$, so the steps described below take the view that $A_{q_0}$ is representing a continuous function from $\CCnr$ to $\CCnr$.

\vspace{.1 in}

\noindent {\bf{Step 1:}} (Removing states  of incomplete response) Let $h_{A_{q_0}}:\CCnr\to\CCnr$ be the continuous map induced by $A_{q_0}$.  {\it To simplify notation below, we will simply refer to this map as $h$.}  Recall further that for any word $\tau\in \Wnre$ the notation $h_{\tau}$ represents the local action of $h$ at $\tau$, while $\tau\theta_h = \Root(U_\tau h)\in\Wnre$.  Let $q\in Q$ and $\nu\in \Wnre$ be such that $\pi_A(\nu,q_0)=q$. Recall further that if $\lambda_A(\nu,q_0)$ is a proper prefix of $\nu\theta_{h}$, then $q$ is considered a state of incomplete response. That is, by the time one traverses the transducer from $q_0$ to $q$, the output word $\lambda(\nu,q)$ is only a proper prefix of the word $\nu\theta_h$ that is eventually guaranteed to be output as a prefix of the ultimate result on processing any point in $\CCnr$ with initial prefix $\nu$.  In particular, if we were to continue processing with any sufficiently large input, the transducer will first write the suffix $\tau= \nu\theta_h - \lambda(\nu,q)$, before writing any other output.  Thus, we strip the impending $\tau$ output from the next few transitions away from $q$, and add it as a suffix to all current outputs $\lambda(x,p)$ where $x$ is a letter and $p$ is a state so that $\pi(x,p) = q$.  After this is done, we have a  transducer $\omega$-equivalent to our original transducer, while the state $q$ is no longer a state of incomplete response (see Subsection \ref{omega} for the definition of $\omega$-equivalence).  Thus, one carries out this process for each word $\nu\in \Wnre$, taken in shortlex ordering, to inductively remove all states of incomplete response.

Note that in \cite{GNSenglish}, the first step is to add a new initial state $q_{-1}$, which has transitions described by a modified transition function $\pi$ copying the transitions from $q_0$, and with the first guaranteed outputs computed to this new initial state.  By this method the algorithm of \cite{GNSenglish} avoids creating a circuit of modifications (so that the resulting algorithm might never stop).  However, as our transducers act on $\CCnr$, we do not need to do this: the state $q_0$ is the unique state that reads letters from $\rd$, and will never admit incoming transitions from elsewhere in the transducer, as all of our input strings start with such a letter but also only have one such letter.  In particular, we can simply execute the algorithm of \cite{GNSenglish} using the original $q_0$ instead of a new state $q_{-1}$, thus this step of the algorithm will result in an $\omega$-equivalent transducer using the same set $Q_A$ of states as for the original transducer $A_{q_0}$.
\vspace{.1 in}

\noindent {\bf{Step 2:}} (Removing inaccessible states) Remove from $Q$ every state that can not be reached
from the initial state $q_0$. That is, if for some $q\in Q$ and for every $\nu \in \Wnre$
we have $\pi(\nu,q_0)\ne q$, then remove $q$ from $Q$.  When $Q$ is finite, we can detect in a finite number of steps if at some length $k$, the words in $\Wnre$ of length $k$ only cause the transitions of $\pi$ from $q_0$ to hit states already seen by following transitions from $q_0$ for shorter words, at which point any unvisited states will never be visited, and can be removed.  Thus, our new transducer may have fewer states than the original.
\vspace{.1 in}

\noindent {\bf{Step 3:}} (Identifying states representing the same local actions) If for two accessible states $q,p\in Q\backslash \{q_0\}$  (in the case of a map on $\CCn$, do this across all of $Q$) we have that for every $\nu\in \Wn$ the equation $\lambda(\nu,q)=\lambda(\nu,p)$ holds, then we can identify $p$ and $q$ as a single state.  Note that if we have removed states of incomplete response, then as discussed by \cite{GNSenglish}, any two $\omega$-equivalent states will satisfy this equivalence across finite words.

 After these potential identifications, our new transducer can have no more states than the original transducer $A$.

Grigorchuk et al. prove (in their context, but the proof is essentially the same in ours) that these steps produce from the transducer $A_{q_0}$ a reduced
non-degenerate initial transducer in its $\omega$-equivalence class, which is unique up to isomorphism of transducers (see Proposition 2.8 of \cite{GNSenglish}).

%%%%%%%%%%%%%%%%%%%%%%%%%%%%%%%%%%
\subsection{The groups $\Rn$ and $\Rnr$}

Given transducers $$A=(X_n,Q_A,\pi_A,\lambda_A) \mbox{ and } B=(X_n,Q_B,\pi_B,\lambda_B)$$ one can form $A*B$, their \emph{product transducer},  as follows.
$$
A*B=(X_n,Q_A\times Q_B,\pi_{A*B},\lambda_{A*B})$$

where $$\pi_{A*B}(x,(p,q))=(\pi_A(x,p),\pi_B(\lambda_A(x,p),q))$$ and
$$\lambda_{A*B}(x,(p,q))=\lambda_B(\lambda_A(x,p),q).$$

If $A$ and $B$ are assigned initial states $p_0$ and $q_0$ respectively, then the initial state of the product transducer $A_{p_0}*B_{q_0}$ is taken as $(p_0,q_0)$.  The reader can easily verify that if $A_{p_0}$ and $B_{q_0}$ are transducers representing continuous functions $f$ and $g$ respectively, then we have  $fg=h_{A_{p_0}}h_{B_{q_0}}=h_{A_{p_0}*B_{q_0}}$.

In \cite{GNSenglish} GNS show that a homeomorphism $f:\CCn\to\CCn$ representable by a finite initial transducer $A_{q_0}$ has its inverse $f^{-1}$ also representable by a finite transducer.  We give an exposition of their argument in Appendix \ref{appendixInversion}, as the inversion process plays an essential role in what follows.

We are now in position to define the {\em rational groups} $\Rn$ and $\Rnr$.

\begin{definition}
Define $\Rn\subset \Homeo(\CCn)$ and $\Rnr\subset \Homeo(\CCnr)$ to be the sets of homeomorphisms which can be represented by minimal transducers which are finite.
\end{definition}

We now have the following lemma.
\begin{lemma}\label{RnGroup}
For integers $1\leqslant r<n$, each of the sets $\Rn$ and $\Rnr$ forms a group under composition.
\end{lemma}

Lemma \ref{RnGroup} is proven in the case of $\Rn$ in \cite{GNSenglish} and the details of that proof work just as well for $\Rnr$.  Appendix \ref{appendixInversion} describes inversion in detail.

\begin{remark} The group $\Gnr$ coincides with a subgroup of $\Rnr$.  In particular, for $g\in\Homeo(\CCnr)$, we will have $g\in\Gnr$ if and only if $g$ is in $\Rnr$ and if, when $g$ is represented by a minimal non-degenerate transducer $A_{q_0}=(\rd,\Xn,R,S,\pi,\lambda,q_0)$, then there is some constant $m$ and a state $q\in S$ so that  for any word $w\in\Wnr$ with $|w|\geq m$, we have $\pi(w,q_0)=q$, and for all $x\in\Xn$, we have $\lambda(x,q)=x.$
\label{remark:trivial-core}\end{remark}

The following is a general proposition for transducers transforming a Cantor space into itself. For instance, it will apply to elements of the various types of rational groups we are discussing here.  We give the statement for the GNS form of the rational group.
\begin{prop}\label{prop:AllStatesHomeosToSynch}
Suppose $A=\left\langle \Xn,Q,\pi_A, \lambda_A\right\rangle$ is a finite transducer.  If, for each state $q\in Q$, the induced map $h_{A_q} : \CCn\to\CCn$ is a self-homeomorphism of $\CCn$, then $A$ is synchronous, has no states with incomplete response, and for each state $q$, the restricted map $\lambda_A(\cdot,q):\Xn\to\Xn$ is a permutation.
\end{prop}

\begin{proof}
%Suppose $A=\left\langle \Xn,Q,\pi_A, \lambda_A\right\rangle$ is a finite transducer with no states of incomplete response so that for each state $q\in Q$ we have $A_q:\CCn\to\CCn$ is a self-homeomorphism of $\CCn$.
Recall that the map $A_q:\CCn\to\CCn$ is defined as the map on Cantor space induced as the infinite extension of the map $\lambda_A(\cdot,q):\Xn^*\to \Xn^*$.  Now, for each letter $a\in\Xn$, we have a word $w_a\seteq\lambda_A(a,q)$ and a state $q_a\seteq \pi(a,q)\in Q$.

By assumption, the image of $A_q$ is the whole of the Cantor space, and $A_q$ is also injective.  Therefore, we must have
%\[\begin{matrix}\CCn=\Image(A_q)=\\ \\ \bigsqcup _{a\in \Xn} w_a\concat \Image(A_{q_a}) =\\ \\ \bigsqcup _{a\in \Xn} w_a\concat \,\CCn. \end{matrix} \]
$$\CCn
=
\Image(h_{A_q}) \hspace{0.2cm}
=
\bigsqcup _{a\in \Xn} w_a\concat \Image(h_{A_{q_a}}) \hspace{0.2cm}
=
 \bigsqcup _{a\in \Xn} w_a\concat \,\CCn.$$
We know that for any $a\in \Xn$, we have $\Image(h_{A_{q_a}})$ is all of Cantor space because the map $h_{A_{q_a}}$ is a homeomorphism for each state $q_a$
(and so, $w_a$ must represent a complete response for the state $q$ with input $a$). Moreover, we know the union above is also a disjoint union of these cylinder sets:
 if there are $a,b\in \Xn$ with $a\neq b$ but $w_a\concat \, \CCn\cap w_b\concat \, \CCn\neq \emptyset$, then the map $h_{A_q}$ would not be injective.

But this means that the set of words $\left\{w_a\mid a\in\Xn\right\}$ is a complete antichain (with cardinality $n$) for the poset of words $\Xn^*$, and the only such antichain is the set of $n$ distinct words of length $1$.  In particular, for all $a\in\Xn$, we have $|\lambda(a,q)|=1$, and the restricted map $\lambda_A(\cdot,q):\Xn\to\Xn$ is actually a permutation.
\end{proof}

%\marginpar{If Theorem 5.4 i assumed then this is rather obvious (?)}
%\marginpar{\color{red} I agree.  Proof commented out.}
%We provide a brief discussion below and leave further details to the reader.
%To prove this lemma for the set $\Rnr$, we need to show that if $h, g\in \Rnr$ then $h^{-1}\in \Rnr$ and $hg\in \Rnr$.  In more detail, the inversion algorithm for elements of $\Rn$ given in \cite{GNSenglish} takes a finite transducer $A_{q_0}$ representing a homeomorphism  $h~\in~\Homeo(\CCn)$, and produces a new finite transducer $B_{(q_0,\varepsilon)}$ representing $h^{-1}$.  This algorithm requires, for each state of $A_{q_0}$, a finite number of states corresponding bijectively with the number of words (including the empty word) which when processed from that state, will produce no output.  The algorithm works also for the homeomorphisms in $\Homeo(\CCnr)$ in precisely the same way.  Also, for closure under the group product of composition, the homeomorphism corresponding to the composition of the initial two homeomorphisms is representable by a transducer with state set the direct product of the state sets of the two finite automata representing the homeomorphisms in the product, and the construction of this transducer is again unaffected by the transition to homeomorphisms of $\CCnr$.

%%%%%%%%%%%%%%%%%%%%%%%%%%%%%%%%%%%%%%%%

\section{Automorphisms Admit Few Local Actions}\label{sec:finitely-many-local-actions}

This section is devoted to the proof of the following result.

\begin{theorem}\label{finiteActions}
Let $\phi \in \Anr$, and let $\varphi : \CCnr \to\CCnr$ be the Rubin conjugator representing $\phi$.  Then the set
\[
\mathcal{L A}_\varphi = \{f:\CCn\to\CCn\mid \exists \nu\in \Wnr, f=\varphi_{\nu} \}
\]
is finite.
\end{theorem}

%\begin{proof} Define an equivalence relation $\approx$ on $\Wnr$ by the rule $\mu\approx\nu$ whenever $\varphi_\mu = \varphi_\nu$.
%Following the procedure of \cite{GNSenglish}, we can build a minimal initial transducer
%$A_{\varphi} = (\rd,\Xn, R,S,\lambda,\pi,q_0)$ of $\CCnr$
%representing $\hat{\tau}$.  We will use $Q:=R\sqcup S$ as in previous sections.
%Also, by construction, the set of states $Q$ of $A_\tau$ then corresponds bijectively with the the set $\mathcal{A}_\tau$.

%\marginpar{change and use HIgman's lemma here}
%Now suppose the set $\mathcal{A}_\tau$ is infinite.  Then, there is an infinite word $w\in\CCnr$ where
%he local action at the node corresponding to each prefix is distinct from the local
%actions at all the other prefixes of $w$ (otherwise, we can cover the Cantor space $\CCnr$ with infinitely many open cones in such
%a way that there is no finite sub-cover).
%\marginpar{I do not see where this discussion is used}

This paragraph contains an informal description of the steps we carry out to show this theorem.  First, we argue that any homeomorphism $h$ of $\CCnr$ which preserves $\sim$ has the property that under any two given basic cones there is a pair of ``parallel'' cones (cones at the same relative address) where the two local actions of $h$ are identical.  (Recall that we have already shown that any automorphism of $\Gnr$ must preserve $\sim$.) Then, we observe that if we compose a homeomorphism which is acting the same way on a pair of cones, with another homeomorphism which after a fixed finite depth under the minimal clopen decomposition of the images acts the same way everywhere (elements of $\Gnr$ act in this way), then the composition will have the same local actions on deep enough parallel cones.  Finally, by thinking of elements of $\Gnr$ as acting on the conjugator $h$, we can pair off arbitrary local actions of $h$ and $h^{-1}$.  Since the result of the conjugation must be in $\Gnr$ these local actions must cancel each other out in finite time to become the identity action. This implies that we can already witness all of the interesting local actions of $h$ within the child cones of our initial two cones, and within a bounded distance: immediately implying the existence of finitely many local actions.

 Note in advance that in this section, if we refer to an antichain for a subset $A$ of the set $\Banr$ of cones, we are referring to an antichain for the partial order on $\Banr$ induced by the ``containment'' partial order of cones: $U_\eta\leq U_\zeta$ if and only if  $U_\eta \supseteq U_\zeta$, which happens if and only if $\eta\leq \zeta\in \Wnre$.  This isomorphism of partial orders will occasionally be used without further comment.

\begin{definition}
For any clopen set $U\subset \CCnr$ there is a minimal finite antichain $A\subseteq \Wnre$
such that $U=\bigcup_{U_\eta\in A} U_\eta$. We denote this antichain by $\Dec(U)$,
and we call it the {\em{decomposition of $U$}}.
 \end{definition}

Note that if $A$ is an antichain in $\Banr$ such that $U=\bigcup A$, then for every
$U_{\eta}\in A$, the set Exp$_\eta(A)\seteq(A\backslash\{U_{\eta}\})\cup \setm{U_{\eta\concat l}}{0\leqslant l\leqslant n-1}$
is also an antichain in $\Banr$ that satisfies $U=\bigcup A$.  Let us call the resulting antichain Exp$_\eta(A)$ the \emph{basic expansion of $A$ at $\eta$}.  We will not allow basic expansions for subsets of $\Banr$ which are not antichains, and we will only define basic expansions for an antichain at the addresses of cones in that antichain.

\begin{definition} The antichain $\Dec_{bal}(U)$ is the minimal antichain $A$ in $\Banr$ satisfying
$U=\bigcup A$ and such that $|\eta_{1}|=|\eta_{2}|$ holds for every $U_{\eta_{1}},U_{\eta_{2}}\in A$.
\end{definition}

We note that from any finite antichain $\Dec(U)$ one can find a finite sequence of basic expansions to create the finite antichain $\Dec_{bal}(U)$.

Although the next two lemmas are somewhat obvious, we provide complete proofs.

\begin{lemma}\label{lemma:1}
Let  $g,h:\CCnr\rightarrow \CCnr$ (or respectively $g,h:\CCn\rightarrow \CCn$) be
two different continuous functions. Then for any $j\in\N$ there exists $\nu\in\Wnr$ (resp. $\nu\in\Wn$)
such that $\Root(U_{\nu}g)\perp \Root(U_{\nu}h)$ and with $|\nu|>j,$ $|\Root(U_{\nu}g)|>j$, and $|\Root(U_{\nu}h)|>j$.
\end{lemma}

\begin{proof}
The proofs for the two cases are the same in concept, so we only give the argument for the case over the Cantor space $\CCnr$.

Let  $g,h:\CCnr\rightarrow \CCnr$ be
two different continuous functions, $j\in\N$, and $\zeta\in \CCnr$ such that $\zeta g\ne \zeta h$.  Set $p\seteq \zeta g$ and $q\seteq \zeta h$.  As $p\neq q$ there are $\mu_p,\mu_q\in\Wnr$, which are non-trivial finite prefixes of $p$ and $q$ respectively, so that $\mu_p\perp \mu_q$. Consequently, $U_{\mu_p}\cap U_{\mu_q}=\emptyset$. (Of course $p\in U_{\mu_p}$ and $q\in U_{\mu_q}$.)  As $g$ and $h$ are continuous, the preimage neighbourhoods $N_{\zeta,g} \seteq U_{\mu_p}g^{-1}$ and $N_{\zeta,h}\seteq U_{\mu_q}h^{-1}$ of $\zeta$ are both open.  Hence, there is a long (non-trivial) prefix $\nu$ of $\zeta$ so that $U_\nu\subset N_{\zeta,g}\cap N_{\zeta,h}$.  We may assume that $|\nu|>j$.  By construction we now have $U_{\nu}g\cap U_{\nu}h\subset U_{\mu_p}\cap U_{\mu_q}=\emptyset,$ and more specifically, $\mu_p\leq \Root(U_\nu g)$ while $\mu_q\leq \Root(U_\nu h)$, so these roots must be incomparable.  By choosing a long enough string $\nu$ we may also insist that $|\Root(U_{\nu}g)|>j$ and $|\Root(U_{\nu}h)|>j$.
\end{proof}

\hspace{0.2cm}

\begin{lemma}\label{lemma:beta}
Let $h\in \Homeo(\CCnr)$ and $U_{\tau},U_{\eta}\in \Banr$.
Then for every $U_{\tau^{\prime}}\subseteq U_{\tau}\, h$ there exist $\eta^{\prime}\in\Wnr$ and $\chi \in \Wn$ so that $\Root(U_{\eta}\,h)<\eta^{\prime}$ and $\tau^{\prime}<\Root(U_{\tau \concat \chi}\, h)$, $\eta^{\prime}<\Root(U_{\eta \concat \chi}\, h)$.
\end{lemma}

\begin{proof}
Let $h$, $\tau$, $\eta$, and $\tau'$ be given so as to satisfy the hypotheses of the statement (so, e.g., $U_{\tau'}\subseteq U_\tau h$).  Let $p\in U_\tau$ so that $p \,  h\in U_{\tau'}$.  Define $\delta\seteq p-\tau$.  By continuity, for a long enough non-trivial finite prefix $\widetilde{\chi}$ of $\delta$, we have $\tau^{\prime}<\Root(U_{\tau\concat\widetilde{\chi}} \,  h)$.  Set $q\seteq(\eta\concat \delta)\,h\in U_\eta  \, h$, and determine $\eta'$ a long enough non-trivial finite prefix of $q$ so that $\Root(U_{\eta} \,  h)<\eta'$.  Again by continuity, there is a long enough non-trivial finite prefix $\widehat{\chi}$ of $\delta$ so that $\eta^{\prime}<\Root(U_{\eta\concat \widehat{\chi}} \,  h)$.  Now set $\chi$ to be the longer of the two words $\widetilde{\chi}$ and $\widehat{\chi}$.  By construction, the conclusions of the lemma statement are satisfied.

\end{proof}

The following represents a key step in the proof of Theorem \ref{finiteActions}.  It shows (via Corollary \ref{cor:2}.\ref{cor:tail-class-preserved}) that any automorphism $h$ of $\Gnr$ has the property that in any two disjoint cones of $\CCnr$ there are subcones found at the same relative address where $h$ has the same local action.

\begin{prop}
\label{pro:ind}
Let $h\in \Homeo(\CCnr)$ and let $U_{\tau},U_{\eta}\in \Banr$.  Suppose $h\in H_{n,r,\sim}$. Then there is $\chi\in\Wne$ so that $h_{\tau\concat \chi}=h_{\eta\concat \chi}$.
\end{prop}

\begin{proof}
%Let $h\in \Homeo(\CCnr)$ and $U_{\tau},U_{\eta}\in \Banr$ so that for every $\chi \in \Wnre$, $h_{\tau\concat \chi}\ne h_{\eta\concat \chi} $.
We will prove the contrapositive statement.

Let $h\in \Homeo(\CCnr)$ and suppose that $U_{\tau}, U_\eta\in\Banr$.  Suppose further that for all $\chi\in\Wne$ we have the local maps $h_{\tau\concat \chi}$ and $h_{\eta\concat \chi}$ are different.  We will show $h\not\in H_{n,r,\sim}$.

Firstly, let $\{ ( \mu_{i} , \nu_{i} ) \}_{i=0}^{\infty} \subset \Wnr \times \Wnr$ be a sequence which visits every pair $(\alpha,\beta)\in \Wnr\times\Wnr$ infinitely often.  That is, for every $\alpha,\beta\in \Wnr$ we have
$$ \big| \{ i \mid  (\mu_{i},\nu_{i} ) = (\alpha,\beta) \} \big| =\aleph_{0}.$$

We will construct by induction four convergent sequences of members of $\Wnre$, namely
$\{\tau_{i}\}_{i=0}^{\infty}$,
$\{\eta_{i}\}_{i=0}^{\infty}$,
$\{\tau_{i}^{\prime}\}_{i=0}^{\infty}$ and
$\{\eta_{i}^{\prime}\}_{i=0}^{\infty}$.  The sequences will have the following properties:

\begin{enumerate}
\item $\tau_{i}\rightarrow \zeta$,
\item $\eta_{i}\rightarrow \rho$,
\item $\tau_{i}^{\prime}\rightarrow \zeta^{\prime}$,
\item $\eta_{i}^{\prime}\rightarrow \rho^\prime$,
\end{enumerate}
for some $\zeta, \rho, \zeta^\prime,$ and $\rho^\prime$ in $\CCnr$.
Furthermore, we will have by construction that $\tau_0 = \tau$ and $\eta_0 = \eta$, and for every $i\in \N$,
the following properties will also hold:
 \begin{enumerate}
\item [($P_a$)]  $\tau_{i}< \tau_{i+1}$,
$\eta_{i}< \eta_{i+1}$ ,
$\tau_{i}^{\prime}< \tau_{i+1}^{\prime}$
and $\eta_{i}^{\prime}< \eta_{i+1}^{\prime},$

\item [($P_b$)] $\tau_{i+1}-\tau_{i}=\eta_{i+1}-\eta_{i},$

\item [($P_c$)]
$$\Root(U_{\tau_{i}}h)<\tau_{i}^{\prime}<\Root(U_{\tau_{i+1}}h)$$
and
$$\Root( U_{\eta_{i}}h)<\eta_{i}^{\prime}<\Root(U_{\eta_{i+1}}h), \textrm{ and}$$

\item [($P_d$)] if $\mu_i< \tau_{i}^{\prime}$ and $\nu_i< \eta_{i}^{\prime}$,
then $(\tau_{i}^{\prime}-\mu_i)\perp (\eta_{i}^{\prime}-\nu_i)$.
 \end{enumerate}

(As will be clear from the proof of Claim \ref{claim:brokentails} just below, these four properties are precisely the properties required so that two points from the same equivalence class under $\sim$ are sent to two points which cannot be in the same equivalence class.  The property $(P_d)$, and the defining property of the sequence $(\mu_i,\nu_i)$, allow us sufficient control so that after removing two arbitrary prefixes from the image points, we can still determine that the resulting suffixes are incomparable.)
\vspace{.2 cm}

\begin{claim} \label{claim:brokentails}
If we can construct the four sequences as above, then $h\notin H_{n,r,\sim}$.
\end{claim}

\begin{proof}[Proof of Claim:]

Assume for a while that we have constructed these sequences. Then we have
 \begin{enumerate} [(i)]
\item $\bigcap_{i=0}^{\infty} U_{\tau_{i}}=\{ \zeta\}$ and
$\bigcap_{i=0}^{\infty} U_{\eta_{i}}=\{ \rho\}$,

\item $\bigcap_{i=0}^{\infty} U_{\tau_{i}^{\prime}}=\{ \zeta^{\prime}\}$ and
$\bigcap_{i=0}^{\infty} U_{\eta_{i}^{\prime}}=\{ \rho^{\prime}\}$, and

\item$ \zeta-\tau_{0}=\rho-\eta_{0}$,
 \end{enumerate}
 where we have (i) and (ii) follows from $(1)$--$(4),$ while (iii) follows from $(P_b)$.  Property (iii) now yields $\zeta\sim \rho$,  while Properties (i), (ii), and ($P_c$) give us that
$$\{\zeta h\} =  \big( \bigcap_{i=0}^{\infty} U_{\tau_{i}} \big)h =
\bigcap_{i=0}^{\infty} \left(U_{\tau_{i}}h\right)\subseteq
\bigcap_{i=1}^{\infty} U_{\tau_{i-1}^{\prime}}=\{\zeta^{\prime}\},$$
$$\{\rho h\} =  \big( \bigcap_{i=0}^{\infty} U_{\eta_{i}} \big)h =
\bigcap_{i=0}^{\infty} \left(U_{\eta_{i}}h\right)\subseteq
\bigcap_{i=1}^{\infty} U_{\eta_{i-1}^{\prime}}=\{\rho^{\prime}\}.$$

\noindent Thus, $\zeta h=\zeta^{\prime}$ and $ \rho h=\rho^{\prime}$. However,
we now show that $\zeta^{\prime}\nsim \rho^{\prime}$, so that we can conlude $h\notin H_{n,r,\sim}$.

Indeed, if there are $\alpha,\beta\in \Wnr$ and $\delta\in \CCn$ such that
$\zeta^{\prime}=\alpha\concat \delta$ and $\rho^{\prime}=\beta\concat \delta$,
then choosing $i\in \bbN$ satisfying $( \mu_{i}, \nu_{i} ) = (\alpha, \beta)$
and also large enough so that
$\alpha < \tau_{i}^{\prime}$ and $\beta < \eta_{i}^{\prime}$, then we have
that both $\tau_{i}^{\prime}-\alpha$ and $\eta_{i}^{\prime}-\beta$ are both non-trivial initial segments of $\delta$ and so must be compatible. But
we have $\tau_{i}^{\prime}-\alpha=(\tau_{i}^{\prime}-\mu_{i})\perp (\eta_{i}^{\prime}-\nu_{i})=\eta_{i}^{\prime}-\beta$, a contradiction.

\end{proof}

We will now construct our sequences.  The main tools will be Lemmas \ref{lemma:1} and \ref{lemma:beta}.  The main point to control is the growth of the sequences $(\tau_i)$ and $(\eta_i)$, noting that the extension of one forces the extension of the other.  Note further that we must maintain sufficient data so as to guarantee that the resulting image points under $h$ are not in the same equivalence class under $\sim$.  In fact, all of the difficulties are faced in the construction at the base level, although we still give the parallel arguments while carrying out the inductive step.

To begin with the construction of the sequences satisfying ($P_a$), ($P_b$), ($P_c$) and ($P_d$) above,
carry out the following steps:
 \begin{enumerate}
\item Set $\tau_{0} := \tau$, $\eta_{0} := \eta$,
$\chi_{\tau_{0}} := \Root(U_{\tau_{0}}h)$ and
$\chi_{\eta_{0}} := \Root(U_{\eta_{0}}h)$.  We now work to define $\tau'_0$ and $\eta'_0$ (which will be extensions of $\chi_{\tau_0}$ and $\chi_{\eta_0}$).

\item  Compare
$(\mu_{0}, \nu_{0} )$ with $(\chi_{\tau_{0}}, \chi_{\eta_{0}})$.

 \begin{enumerate}
\item If $\mu_{0}\nleqslant\chi_{\tau_{0}}$
then choose an $\alpha^{\prime}\in\Wn$ so that
$U_{\alpha^{\prime}}\subset (\CCn)h_{\tau_{0}}$, $\mu_{0}\perp \chi_{\tau_{0}}\concat \alpha^{\prime}$ (note that if $\chi_{\tau_0}<\mu_0$ then by the definition of the function $\Root$ we can choose $\alpha'$ so that incomparability is achieved on the first letter of $\alpha'$, and if $\chi_{\tau_0}\perp\mu_0$ then any non-trivial $\alpha'$ so that $U_{\alpha^{\prime}}\subset (\CCn)h_{\tau_{0}}$ will do).

\item Else, if $\nu_{0}\nleqslant\chi_{\eta_{0}}$ choose any $\beta^{\prime}\in\Wn$  so that
$U_{\beta^{\prime}}\subset (\CCn)h_{\eta_{0}}$
and $\nu_{0}\perp \chi_{\eta_{0}}\concat \beta^{\prime}$, noting that such non-trivial $\beta^\prime$ will exist as in the previous point for choosing $\alpha^{\prime}$.

\item If $\mu_{0}\leqslant \chi_{\tau_{0}}$ and $\nu_{0}\leqslant \chi_{\eta_{0}}$, then in these next four subcases we will choose an $\alpha^{\prime}\in\Wn$ or a $\beta^{\prime}\in\Wn$ to force a local incomparability condition:

 \begin{enumerate}
\item If $\chi_{\tau_{0}}-\mu_0 = \chi_{\eta_{0}}-\nu_0$,
then as $h_{\tau_{0}}\ne h_{\eta_{0}}$, by Lemma \ref{lemma:1}
there exists $\nu\in\Wn$ so that $$\Root(U_{\nu}h_{\tau_{0}})\perp  \Root(U_{\nu}h_{\eta_{0}}).$$  Now choose some $\alpha'\in\Wn$ with $U_{\alpha^{\prime}}\subset U_{\nu}h_{\tau_{0}}$.
\item Else if $\chi_{\tau_{0}}-\mu_{0}< \chi_{\eta_{0}}-\nu_{0}$, then
set $\theta:=(\chi_{\eta_{0}}-\nu_{0})-(\chi_{\tau_{0}}-\mu_{0})$ (observe $\theta\neq \varepsilon$) and choose
any $\alpha^{\prime}\in \Wn$ so that
$U_{\alpha^{\prime}}\subset (\CCn)h_{\tau_{0}}$ and with
$\theta\perp\alpha^{\prime}$ (since $\chi_{\tau_{0}} = \Root(U_{\tau_{0}}h)$ we can even insist that the first letter of $\alpha^{\prime}$ is different from the first letter of $\theta$).

\item Else if $\chi_{\tau_{0}}-\mu_{0}> \chi_{\eta_{0}}-\nu_{0}$, then
set $\theta:=(\chi_{\tau_{0}}-\mu_{0})-(\chi_{\eta_{0}}-\nu_{0})$ (again, $\theta\neq \varepsilon$) and choose
any  $\beta^{\prime}\in\Wn$ so that
$U_{\beta^{\prime}}\subset (\CCn)h_{\eta_{0}}$ such that
$\theta\perp\beta^{\prime}$  (since $\chi_{\eta_{0}} = \Root(U_{\eta_{0}}h)$ we can even insist that the first letter of $\beta^{\prime}$ is different from the first letter of $\theta$).

\item Else we must have $(\chi_{\tau_{0}}-\mu_{0}) \perp (\chi_{\eta_{0}}-\nu_{0})$.  Choose any $\alpha^{\prime}\in\Wn$ so that
$\Root((\CCn)h_{\tau_{0}})<\alpha^{\prime}$.

 \end{enumerate}

 \end{enumerate}

\item
If we have chosen a value for $\alpha^{\prime}$, then it is nontrivial and so setting $\tau_{0}^{\prime}:=\chi_{\tau_{0}}\concat \alpha^{\prime}$ we have $\chi_{\tau_0}=\Root(U_{\tau_{0}}h) <\tau_{0}^{\prime}.$
By Lemma \ref{lemma:beta}, there exists $\eta^{\prime}\in\Wnr$  and
$\chi \in \Wn$ so that $$\Root(U_{\tau_{0}}h) <\tau_{0}^{\prime}<\Root(U_{\tau_{0}\concat \chi}h),$$ and
 $$\Root(U_{\eta_{0}}h)<\eta^{\prime}<\Root(U_{\eta_{0}\concat \chi}h).$$ Now we set $\eta_{0}^{\prime}:=\eta^{\prime}$,
$\tau_{1}:=\tau_{0}\concat \chi$ and
$\eta_{1}:=\eta_{0}\concat \chi$.

\item
If instead we have chosen a value for $\beta^{\prime}$, then it is nontrivial and so setting $\eta_{0}^{\prime}:=\chi_{\eta_{0}}\concat \beta^{\prime}$ we have $\Root(U_{\eta_{0}}h)<\eta_0^{\prime}$.
By Lemma \ref{lemma:beta}, there exists $\tau^{\prime}\in\Wnr$  and $\chi \in \Wn$
such that $$\Root(U_{\tau_0}\,h)<\tau^{\prime}<\Root(U_{\tau_{0}\concat \chi}\,h)$$ and

$$\Root(U_{\eta_{0}}h)<\eta_0^{\prime}<\Root(U_{\eta_{0}\concat \chi}h).$$  Now we set $\tau_{0}^{\prime}:=\tau^{\prime}$,
$\tau_{1}:=\tau_{0}\concat \chi$ and
$\eta_{1}:=\eta_{0}\concat \chi$.

 \end{enumerate}

\noindent We have

\noindent $(a)$
$\tau_{0}< \tau_{1}$, $\eta_{0}< \eta_{1}$ and
$\tau_{1}-\tau_{0}=\eta_{1}-\eta_{0}$,

\noindent $(b)$

$$\Root(U_{\tau_{0}}h) <\tau_{0}^{\prime}<\Root(U_{\tau_{0}\concat \chi}h)$$ and

$$\Root(U_{\eta_{0}}h)<\eta_0^{\prime}<\Root(U_{\eta_{0}\concat \chi}h),$$ which together imply $(P_c)$ for $i=0$.

\noindent $(c)$ if $\mu_{0}< \tau_{0}^{\prime}$ and $\nu_{0}< \eta_{0}^{\prime}$ (by construction these cases may only arise from (2).(c)), then $(\tau_{0}^{\prime}-\mu_{0})\bot (\eta_{0}^{\prime}-\nu_{0})$ by our choice of $\alpha^{\prime}$ or of $\beta^{\prime}$,

\noindent $(d)$
$h_{\tau_{1}}\ne
h_{\eta_{1}}$.

\noindent Note that $(d)$ is guaranteed by our general hypotheses that for all $\chi\in\Wne$ we have $h_{\tau\concat \chi}\neq h_{\eta\concat\chi}$.  We have now verified the properties ($P_*$) for the case $i=0$ except the second clause of ($P_a$), which we cannot check yet as $\tau_1^{\prime}$ and $\eta_1^{\prime}$ are not yet defined.

Next, assume that we have already defined
$$\fsetn{\tau_{0}}{\tau_{k}},
\fsetn{\eta_{0}}{\eta_{k}},
\fsetn{\tau_{0}^{\prime}}{\tau_{k-1}^{\prime}}
\mbox{ and } \fsetn{\eta_{0}^{\prime}}{\eta_{k-1}^{\prime}}$$ for some $k\geq 1$ in such a way that the following holds:

\noindent $(a)$ for every $0\leqslant i< k$, we have
$\tau_{i}< \tau_{i+1}$,
$\eta_{i}<\eta_{i+1}$, $\tau_{i+1}-\tau_{i}=\eta_{i+1}-\eta_{i}$,
and for $0 \leqslant i < k-2$, we have $\tau_{i}^{\prime}< \tau_{i+1}^{\prime}$
and $\eta_{i}^{\prime}< \eta_{i+1}^{\prime}$,

\noindent $(b)$

$$\Root(U_{\tau_{i}}h) <\tau_{i}^{\prime}<\Root(U_{\tau_{i+1}}h)$$ and

$$\Root(U_{\eta_{i}}h)<\eta_i^{\prime}<\Root(U_{\eta_{i+1}}h),$$
 for every $0 \leqslant i < k$,

\noindent $(c)$ if $\mu_{i}< \tau_{i}^{\prime}$ and $\nu_{i}<\eta_{i}^{\prime}$,
then $(\tau_{i}^{\prime}-\mu_{i})\perp (\eta_{i}^{\prime}-\nu_{i})$,

\noindent $(d)$
$h_{\tau_{k}}\ne
h_{\eta_{k}}$.

\noindent We want to define $\tau_{k+1},\eta_{k+1},\tau_{k}^{\prime}$ and $\eta_{k}^{\prime}$.
We can do this by following the logic of the definitions of $\tau_1$, $\eta_1$, $\tau^{\prime}_0$ and $\eta^{\prime}_0$ from the foundational definitions of $\tau_0$ and $\eta_0$.

Specifically, we carry out the following steps:
 \begin{enumerate}
\item Set
$\chi_{\tau_{k}}:=\Root(U_{\tau_{k}}h)$ and
$\chi_{\eta_{k}}:=\Root(U_{\eta_{k}}h)$, and note that
$\tau_{k-1}^{\prime}< \chi_{\tau_{k}}$ and
$\eta_{k-1}^{\prime}< \chi_{\eta_{k}}$.

\item  Compare
$(\mu_{k}, \nu_{k})$ with $(\chi_{\tau_{k}}, \chi_{\tau_{k}})$.

 \begin{enumerate}
\item If $\mu_{k}\nleqslant \chi_{\tau_{k}}$ or $\nu_{k}\nleqslant \chi_{\eta_{k}}$ then:

 \begin{enumerate}
\item If $\mu_{k}\nleqslant \chi_{\tau_{k}}$, then choose any  $\alpha^{\prime}\in\Wn$ so
$$\Root((\CCn)h_{\tau_{k}})<\alpha^{\prime}$$
and such that $\mu_{k}\perp \chi_{\tau_{k}}\concat \alpha^{\prime}$, just as in (2).(a) for the case $k=0$.

\item Else if $\nu_{k}\nleqslant \chi_{\eta_{k}}$, then choose any $\beta^{\prime}\in\Wn$ so that
$\Root((\CCn)\,h_{\eta_k})<\beta^{\prime}$ and
so that $\nu_{k}\perp \chi_{\eta_{k}}\concat \beta^{\prime}$, just as in (2).(b) for the case $k=0$.

 \end{enumerate}

\item If $\mu_{k}\leqslant \chi_{\tau_{k}}$ and $\nu_{k}\leqslant \chi_{\eta_{k}}$, then:

 \begin{enumerate}
\item If $\mu_{k}= \chi_{\tau_{k}}$ and $\nu_{k}= \chi_{\eta_{k}}$,
then as $h_{\tau_{k}}\ne h_{\eta_{k}}$, by Lemma \ref{lemma:1} there exists $\nu\in\Wn$ so that  $$\Root(U_{\nu}h_{\tau_{k}})\perp \Root(U_{\nu}h_{\eta_{k}}),$$ so choose any $\alpha^{\prime}\in\Wn$ with
$$\Root(U_{\nu}h_{\tau_{k}})<\alpha^{\prime}.$$

\item Else if $\chi_{\tau_{k}}-\mu_{k}< \chi_{\eta_{k}}-\nu_{k}$, then
set $\theta:=(\chi_{\eta_{k}}-\nu_{k})-(\chi_{\tau_{k}}-\mu_{k})$ (observe $\theta\neq \varepsilon$) and choose
any $\alpha^{\prime}\in\Wn$ so that
$\Root((\CCn)h_{\tau_{k}})<\alpha^{\prime}$ and with
$\theta\perp\alpha^{\prime}$.

\item Else if $\chi_{\tau_{k}}-\mu_{k}> \chi_{\eta_{k}}-\nu_{k}$, then
set $\theta:=(\chi_{\tau_{k}}-\mu_{k})-(\chi_{\eta_{k}}-\nu_{k})$ (again, $\theta\neq \varepsilon$) and choose
any $\beta^{\prime}\in\Wn$ so that $U_{\beta^{\prime}}\subset (\CCn)h_{\eta_{k}}$ and such that
$\delta\bot\beta^{\prime}$.

\item Else we must have $( \chi_{\tau_{k}}-\mu_{k} ) \bot (\chi_{\eta_{k}}-\nu_{k} )$.  Here, choose any $\alpha^{\prime}\in\Wn$ so that  $\Root((\CCn)h_{\tau_{k}})<\alpha^{\prime}.$

 \end{enumerate}

 \end{enumerate}

\item
If we have chosen a value for $\alpha^{\prime}$, then it is nontrivial and so setting $\tau_{k}^{\prime}:=\chi_{\tau_{k}}\concat \alpha^{\prime}$ we have $\Root(U_{\tau_{k}}h) <\tau_{k}^{\prime}.$
By Lemma \ref{lemma:beta}, there exists $\eta^{\prime}\in\Wnr$  and
$\chi \in \Wn$ so that $$\Root(U_{\tau_{k}}h) <\tau_{k}^{\prime}<\Root(U_{\tau_{k}\concat \chi}h),$$ and
 $$\Root(U_{\eta_{k}}h)<\eta^{\prime}<\Root(U_{\eta_{k}\concat \chi}h).$$ Now we set $\eta_{k}^{\prime}:=\eta^{\prime}$,
$\tau_{k+1}:=\tau_{k}\concat \chi$ and
$\eta_{k+1}:=\eta_{k}\concat \chi$.

\item
If instead we have chosen a value for $\beta^{\prime}$, then it is nontrivial and so setting $\eta_{k}^{\prime}:=\chi_{\eta_{k}}\concat \beta^{\prime}$ we have $\Root(U_{\eta_{k}}h)<\eta_k^{\prime}$.
By Lemma \ref{lemma:beta}, there exists $\tau^{\prime}\in\Wnr$  and $\chi \in \Wn$
so that $$\Root(U_{\tau_k}\,h)<\tau^{\prime}<\Root(U_{\tau_{k}\concat \chi}\,h)$$ and
$$\Root(U_{\eta_{k}}h)<\eta_k^{\prime}<\Root(U_{\eta_{k}\concat \chi}h).$$  Now we set $\tau_{k}^{\prime}:=\tau^{\prime}$,
$\tau_{k+1}:=\tau_{k}\concat \chi$ and
$\eta_{k+1}:=\eta_{k}\concat \chi$.

 \end{enumerate}
 Note that in the special case of $k=1$ in the above construction, we will construct $\tau_2$, $\eta_2$, $\tau_1^{\prime}$ and $\eta_1^{\prime}$, and specifically, we will have

 $$\chi_{\tau_1}=\Root(U_{\tau_1}\,h)<\tau_1^{\prime}<\Root(U_{\tau_{2}}\,h)$$ and

$$\chi_{\eta_1}=\Root(U_{\eta_{1}}h)<\eta_1^{\prime}<\Root(U_{\eta_{2}}h).$$  Recall that we have already determined some shorter strings, and in particular we can extend these two relationship strings to the left as:

 $$\chi_{\tau_0}=\Root(U_{\tau_0}\,h)<\tau_0^{\prime}<\Root(U_{\tau_1}\,h)<\tau_1^{\prime}<\Root(U_{\tau_{2}}\,h)$$ and

$$\chi_{\eta_0}=\Root(U_{\eta_{0}}h)<\eta_0^{\prime}<\Root(U_{\eta_{1}}h)<\eta_1^{\prime}<\Root(U_{\eta_{2}}h).$$

From this we have $\tau_0^{\prime}<\tau_1^{\prime}$ and $\eta_0^{\prime}<\eta_1^{\prime}$ completing the verification of the properties $(P_*)$ at the base level.

Now, the verification of the desired properties $(P_*)$ at the $k^{th}$ level works in similar fashion to the verification at the base level.

This then completes the construction of the sequences, and thus the proof of the proposition.
\end{proof}

\vspace{0.2cm}

\begin{cor} \label{cor:exis}
If $h\in H_{n,r,\sim}$ then there exist $U_\nu,U_{\eta}\in \Banr$ such that
$\nu$ and $\eta$ are incomparable, $U_\nu \cup U_{\eta}\ne \CCnr$ and
$h_{\nu}=h_{\eta}$.
\end{cor}

\vspace{0.2cm}

\begin{definition}
For $h\in \Homeo(\CCnr)$ and $U_{\nu},U_{\eta}\in \Banr$, we say that $h$
acts on $U_{\nu}$ and $U_{\eta}$
\emph{in the same fashion}  provided that
${h}_{\nu}= {h}_{\eta}$.

\begin{comment}For two clopen sets $U,W\subseteq \CCnr$ we say that $h$
acts on $U$ and $W$ \emph{in the same fashion everywhere}  provided that $h$ acts on $U_{\nu}$
and $U_{\eta}$ in the same fashion for all $U_{\nu}\in \Dec(U)$ and $U_{\eta}\in \Dec(W)$.
%\marginpar{modulo a bijection of the sets of the cones}
\end{comment}
\end{definition}

\begin{definition}
Let $h\in \Homeo(\CCnr)$ and $U$ and $W$ clopen sets in $\CCnr$. We say that $h$ acts on $U$ and $W$
\emph{almost in the same fashion} provided that ${h}_{\nu \concat \chi}= {h}_{\eta \concat \chi}$
holds for all $\chi\in \Wn$ with $|\chi|\geq k$  and all  $U_{\nu}\in \Dec(U)$, $U_{\eta}\in \Dec(W)$.

When  $h$ acts on $U$ and $W$ almost in the same fashion and $k$ is the minimal natural
number which satisfies  the condition in the above definition, we define \emph{$\crit_{h}(U,W) := k$},
which we refer to as  \emph{the critical level of $U$ and $W$ with respect to $h$.}
\end{definition}

\begin{definition}We say that $h$ acts on clopen sets $U$ and $W$
\emph{in the same fashion uniformly}  provided that for every
$U_{\nu}\in \Dec(U)$ and $U_{\eta}\in \Dec(W)$, we have
%$h$ acts on $U_{\nu}$ and $U_{\eta}$ in the same fashion and
 ${h}_{\nu} = {h}_{\eta \concat \zeta}$
for any $\zeta \in W_{n,\varepsilon}$.

We say that $h$ acts on $U$ and $W$
\emph{almost in the same fashion uniformly}  provided that
there exists  $k\in \bbN$ such that for every
$U_{\nu}\in \Dec(U)$, $U_{\eta}\in \Dec(W)$ and $\chi,\zeta \in \Wn$ with both $|\chi|$, $|\zeta|\geq k$, we have ${h}_{\nu \concat \chi}= {h}_{\eta \concat \zeta}$.
\end{definition}

\begin{remark}
\label{claim:1}
Let $g\in \Gnr$ and let $U$ and $W$ be any two clopen sets.
Then $g$ acts on $U$ and $W$ almost in the same fashion uniformly.
Indeed, local actions associated to small-enough cones are all the identity map.
\end{remark}

The following lemma should be obvious but the two that follow after will require some explanation.

\begin{lemma}
\label{lem:sameFashion}
Let $g,h\in \Homeo(\CCnr)$ and let $U_{\nu},U_{\eta}\in \Banr$.
Suppose that $U_{\nu}g=U_{\nu^{\prime}}\in \Banr$,
$U_{\eta}g=U_{\eta^{\prime}}\in \Banr$,
$g$ acts on $U_{\nu}$ and $U_{\eta}$
in the same fashion
and $h$ acts on $U_{\nu^{\prime}}$ and $U_{\eta^{\prime}}$
in the same fashion. Then $gh$ acts on $U_{\nu}$ and $U_{\eta}$
in the same fashion.
\end{lemma}

\begin{lemma}
\label{lem:gh}
Let $g,h\in \Homeo(\CCnr)$ and let $U_{\nu},U_{\eta}\in \Banr$.
Suppose that
$g$ acts on $U_{\nu}$ and $U_{\eta}$
in the same fashion
and $h$ acts on $U_{\nu}g$ and $U_{\eta}g$
almost in the same fashion uniformly. Then $gh$ acts on
$U_{\nu}$ and $U_{\eta}$ almost
in the same fashion.
\end{lemma}

\begin{proof}
Suppose $g,h\in \Homeo(\CCnr)$, $U_{\nu},U_{\eta}\in \Banr$, that
$g$ acts on $U_{\nu}$ and $U_{\eta}$
in the same fashion and also that $h$ acts on $U_{\nu}g$ and $U_{\eta}g$
almost in the same fashion uniformly.

We will show that the composition $gh$ acts on $U_{\nu}$ and $U_{\eta}$ almost
in the same fashion by giving a precise calculation of the action of the composition for ``parallel'' cones that are deep enough in $U_{\nu}$ and $U_\eta$ so that the images under $g$ of these deep cones are themselves contained in parallel cones that are then deep enough so that $h$ is acting in the same fashion at both of those latter cones.  As the total local actions of the composition on our initial deep cones are (almost) given as the composition of the same pair of local actions, they must then be the same.  We now formalise this discussion.

There is $k\in\N$ so that for all $\chi,\zeta\in\Wn$ with both $|\chi|$, $|\zeta|\geq k$, and for all $\xi$, $\theta$ with $U_\xi\in\Dec(U_\nu g)$ and $U_\theta\in\Dec(U_{\eta} g)$, we have the equality $h_{\xi\concat\chi}=h_{\theta\concat\zeta}$.

Observe that there is a minimal length $m\in\N$ so that for all $\mu\in\Wnre$ of length at least $m$, such that either  $U_\mu\subset U_\nu g$ or
$U_\mu\subset U_\eta g$,  we have $\mu = \gamma\concat \chi$ for some $\chi$ with $|\chi|>k$ and for some $\gamma$ with either $U_\gamma\in \Dec(U_{\nu} g)$ or $U_\gamma\in \Dec(U_{\eta} g)$ (respectively). Therefore for words $\mu, \phi \in \Wnre$ of length at least $m$ such that $U_{\mu} \subset U_\nu g$ and $U_{\phi} \subset U_\eta g$ we have $h_{\mu} = h_{\phi}$.

Further observe that as $g$ acts on $U_{\nu}$ and $U_{\eta}$ in the same fashion, for all long enough $\chi\in\Wn$, two things must happen.  First, $g_{\nu\concat \chi}=g_{\eta\concat\chi}$, and second, both $|\Root(U_{\nu\concat \chi} g)|>m$ and $|\Root(U_{\eta\concat\chi} g)|>m$.  Note that there is a minimal value $s\in\N$ so both of these things happen  for all $\chi\in \Wn$ of length at least $s$.

For the next part of the argument, we extend the concatenation operator ``$\concat$''so that it may be applied to sets of strings as follows.  If $A$ and $B$ are sets of strings then we define $A\concat B\seteq \{x\concat y\mid x\in A, y\in B\}$.  We also want to allow the operator to join, via concatenation, a single string to a set of strings, or a set of strings to a string, in the obvious ways.

Let $\chi\in\Wn$ so that $|\chi|\geq s$. Notice that $$U_{\nu \concat\chi}g =((\nu \concat \chi)\theta_{g})\concat (\CCn g_{\nu\concat\chi})$$
and likewise $U_{\eta\concat\chi}g = ((\eta \concat \chi)\theta_{g})\concat (\CCn g_{\eta\concat\chi})$.
Thus $$U_{\nu\concat\chi} gh = \{ (((\nu\concat \chi)\theta_{g}) \theta_{h}) \concat ((\CCn g_{\nu\concat\chi}) h_{(\nu\concat \chi)\theta_{g}})\}$$ likewise,
$$U_{\eta \concat \chi} gh = \{ (((\eta\concat \chi)\theta_{g}) \theta_{h})\concat ((\CCn g_{\eta\concat\chi}) h_{(\eta\concat \chi)\theta_{g}})\}.$$

We will now study how the composition $gh$ moves individual points from the sets $U_{\nu\concat\chi}$ and $U_{\eta \concat \chi}$.  Recall (by definition), that $$\Root(U_{\nu \concat \chi} g) = (\nu\concat \chi)\theta_{g} \mbox{ and } \Root(U_{\eta \concat \chi} g) = (\eta\concat \chi)\theta_{g},$$ so we have, by the discussion of third and fourth paragraphs of this proof, that $h_{(\nu\concat \chi)\theta_{g}} =h_{(\eta\concat \chi)\theta_{g}}$.
  Now let $\rho$ be the greatest common prefix of the set $(\CCn g_{\nu\concat\chi} )h_{(\nu \concat \chi)\theta_{g}}$. Therefore, as $g_{\nu\concat\chi} = g_{\eta  \concat \chi}$,
$\rho$ is also the greatest common prefix of the set $(\CCn g_{\eta\concat\chi}) h_{(\eta\concat \chi)\theta_{g}}$.
Thus, we have $\Root(U_{\nu \concat \chi} gh) = ((\nu\concat \chi)\theta_{g}) \theta_{h}  \concat \rho$ and
$\Root(U_{\eta \concat \chi}gh) = ((\eta\concat \chi)\theta_{g}) \theta_{h} \concat \rho$.
It therefore follows that $(gh)_{\nu \concat \chi} = (gh)_{\eta \concat \chi}$ since for any $y \in \CCn$ we have

$$(\nu \concat\chi\concat y)gh = (((\nu\concat \chi)\theta_{g}) \theta_{h}) \concat (((y) g_{\nu\concat\chi}) h_{(\nu\concat \chi)\theta_{g}})$$

and

$$(\eta \concat\chi\concat y)gh = (((\eta\concat \chi)\theta_{g}) \theta_{h})\concat (((y) g_{\eta\concat\chi}) h_{(\eta\concat \chi)\theta_{g}})$$

where $((y) g_{\nu\concat\chi}) h_{(\nu \concat \chi)\theta_{g}}$ is equal to $((y) g_{\eta\concat\chi}) h_{(\eta\concat \chi)\theta_{g}}$ and has $\rho$ as a prefix.

In particular, we can conclude that $gh$ acts on
$U_{\nu}$ and $U_{\eta}$ almost
in the same fashion.
\end{proof}

The following represents a ``fundamental exercise'' for those developing their understanding within the body of theory relating to the Higman--Thompson groups.

\begin{lemma}
\label{lemma:g}
Let $\nu_{1},\nu_{2},\eta_{1},\eta_{2}\in \Wnr$ such that $\nu_{1}\bot\nu_{2}$
, $\eta_{1}\bot\eta_{2}$
, $U_{\nu_{1}}\cup U_{\nu_{2}}\ne \CCnr$ and
$U_{\eta_{1}}\cup U_{\eta_{2}}\ne \CCnr$.
Then there exists  $g\in \Gnr$ such that $g \! \mid_{U_{\nu_{1}}} = g_{\nu_{1},\eta_{1}}$
and  $g \! \mid_{U_{\nu_{2}}} = g_{\nu_{2},\eta_{2}}$.
\end{lemma}
\begin{proof}
Suppose $\nu_{1},\nu_{2},\eta_{1},\eta_{2}\in \Wnr$ such that $\nu_{1}\bot\nu_{2}$
, $\eta_{1}\bot\eta_{2}$
, $U_{\nu_{1}}\cup U_{\nu_{2}}\ne \CCnr$ and
$U_{\eta_{1}}\cup U_{\eta_{2}}\ne \CCnr$. There is a natural $k>2$ (with $k\equiv r\mod (n-1)$) so that we can find two complete prefix codes
$ \vecnu\seteq\{\alpha_1,\alpha_2,\ldots,\alpha_k\}, \veceta\seteq \{\beta_1,\beta_2,\ldots,\beta_k\}\subset \Wnr$ with $\alpha_1= \nu_1$, $\alpha_2=\nu_2$, $\beta_1=\eta_1$, and $\beta_2=\eta_2$.  In particular, the prefix code map $g_{\vecnu,\veceta}$ (see Definition \ref{def:PrefixCodeMap}) is an element of $\Gnr$ satisfying the conclusion of the lemma.
\end{proof}

Relating Lemma \ref{lemma:g} to the ideas above we obtain the following.
\begin{cor}
\label{cor:fashion}
Let $h\in \Homeo(\CCnr)$ such that $h^{-1} \Gnr h \subseteq \Gnr$.
Then for every $U_{\nu},U_{\eta}\in \Banr$
such that $ U_{\nu}\cup U_{\eta}\ne \CCnr$, the map
$h$ acts on $U_{\nu}$ and $U_{\eta}$ almost in the same fashion.
\end{cor}

\begin{proof}
By Lemma \ref{lemma:tran}, we have $h\in H_{n,r,\sim}$, so by Corollary \ref{cor:exis}
there exist $U_{\nu^{\prime}},U_{\eta^{\prime}}\in \Banr$ such that $h$
acts on $U_{\nu^{\prime}}$ and $U_{\eta^{\prime}}$ in the same fashion.
Moreover, we can
choose these $\nu^{\prime}$ and $\eta^{\prime}$ such that $\nu^{\prime}\bot\eta^{\prime}$
and $ U_{\nu^{\prime}}\cup U_{\eta^{\prime}}\ne \CCnr$.

Let $U_{\nu},U_{\eta}\in \Banr$ be such that $ U_{\nu}\cup U_{\eta}\ne \CCnr$. Assume
first that $\nu\bot\eta$. By Lemma \ref{lemma:g} there exists  $g\in \Gnr$ such that
$g\! \mid_{U_{\nu}} = g_{\nu,\nu^{\prime}}$ and $g\! \mid_{U_{\eta}} = g_{\eta,\eta^{\prime}}$.
So, by Lemma \ref{lem:sameFashion}, $gh$ acts on
$U_{\nu}$ and $U_{\eta}$ in the same fashion.
Letting $f := h^{-1} g h \in \Gnr$, we have $hf = gh$.
By Remark \ref{claim:1} and Lemma \ref{lem:gh},
$h = ghf^{-1}$ acts on $U_{\nu}$ and $U_{\eta}$ almost in the same fashion.

Next, assume that $\nu\leqslant\eta$ or $\eta\leqslant\nu$.
Then there exists $\beta\in \Wnr$ such that $\nu\bot\beta$, $\eta\bot\beta$
and $U_{\beta}\cup U_{\nu}\ne \CCnr \ne U_{\beta}\cup U_{\eta}$.
So by the first case,
$h$ acts on $U_{\nu}$ and $U_{\beta}$ almost in the same fashion and
$h$ acts on $U_{\eta}$ and $U_{\beta}$ almost in the same fashion.
That is, there exists  $k_{1}\in \bbN$ such that for every
$\zeta \in \Wn$ with $|\zeta| \geq k_{1}$, we have that
$h$ acts on $U_{\nu\concat \zeta}$ and $U_{\beta\concat \zeta}$
in the same fashion, and there exists  $k_{2}\in \bbN$ such that for every
$\chi\in \Wn$ with $|\chi|\geq k_{2}$, we have that
$h$ acts on $U_{\eta\concat \chi}$ and $U_{\beta\concat \chi}$
in the same fashion. Taking $k := max\{k_{1},k_{2}\}$, we get that for every
$\mu \in \Wn$ with $|\mu|\geq k$, the map $h$ acts on  $U_{\nu\concat \mu}$,
$U_{\beta\concat \mu}$ and $U_{\eta\concat \mu}$ in the same fashion.
Hence, $h$ acts on $U_{\nu}$ and $U_{\eta}$ almost in the same fashion. \end{proof}

\

We are finally in position to complete the proof of Theorem \ref{finiteActions}, that
is, if $h\in N_{H(\CCnr)}(\Gnr)$, then $h$ admits only finitely many types of local actions (recall this implies that $h$ can be represented by a transducer with only finitely many states).  Notice
that the statement below is slightly more general in that we also allow homeomorphisms $h$ which conjugate $\Gnr$ into $\Gnr$, even if the image of $\Gnr$ under this general conjugation is a proper subset of $\Gnr$.  (Such homeomorphisms play a role in \cite{BleakDonovenJonusas}, for instance.)

\begin{cor} \label{cor:finite types}
Let $h\in \Homeo(\CCnr)$ be such that $h^{-1} \Gnr h \subseteq \Gnr$.
Then $h$ uses only finitely many types of local action.
\end{cor}

\begin{proof}
Let us fix $A$ a complete antichain having at least three elements. For instance,
if $n=2$, we can take $A := \{ \dot{0}00,\dot{0}01,\dot{0}10,\dot{0}11\}$, and if $n>2$, we can take either
$A:=\rd\times \Xn$ or $A:=\rd \times \Xn^2$. In all cases, $A$ can be taken so that all of its words
are length three.

By Corollary \ref{cor:fashion}, for every $\nu,\eta\in A$ and $a \in \{0,\ldots,n-1\}$, the map $h$ acts on $U_{\nu}$ and
$U_{\eta}$ almost in the same fashion, and it also acts on $U_{\nu}$ and $U_{\nu\concat a}$ almost in the same fashion. Set
\begin{IEEEeqnarray*}{rCl}
k := \max
\big\{ \crit_h (U_{\nu},U_{\eta})\mid &\nu,&\eta \in A \mbox{ or } \nu\in A \mbox{ and } \\ &\eta&=\nu\concat a \mbox{ for some } a \in \Xn \big\}.
\end{IEEEeqnarray*}
Now, let $B:=\{\nu_0,\ldots,\nu_{n^{k} -1}\}$ be the set of all the $\nu\in \Wn$ with $|\nu|=k$,
ordered by the lexicographic order of $\Wne$.

We claim  that for every $\eta\in \Wnr$ with  $|\eta|\geq k+3$,
there exists $\nu \in B$ such that $h_{\eta}=h_{\tau\concat\nu}$ holds for every $\tau\in A$.

To show the claim above, we proceed by induction on $m := |\eta|$. If $m=k+3$, then $\eta=\chi\concat\nu$ for some $\chi\in A$
and $\nu\in B$. Now, by the choice of $k$, we have $h_{\chi\concat \nu}=h_{\tau\concat \nu}$ for every $\tau\in A$, as desired.

Assume now that for
every $\eta\in \Wnr$ with $|\eta|=m$ the claim is true.
Let $\eta\in \Wnr$ be such that  $|\eta|=m+1$. Write
$\eta=\chi\concat b \concat \eta^{\prime}$ for $\chi\in A$
and $b \in \{ 0,1,\ldots ,n-1\}$. Then $|\eta^{\prime}|\geq k$, so by the choice of $k$,
we have that
$h_{\chi\concat b \concat\eta^{\prime}}=h_{\chi\concat \eta^{\prime}}$.
By  the induction hypothesis, for $\chi\concat\eta^{\prime}$
there exists  $\nu \in B$ such that $h_{\chi\concat \eta^{\prime}}=h_{\tau\concat \nu}$
holds for every $\tau\in A$. Thus,
for every $\tau\in A$, we have $h_{\eta}=h_{\chi\concat b \concat\eta^{\prime}}
=h_{\chi\concat \eta^{\prime}}=h_{\tau\concat \nu}$, as desired. \end{proof}

Thus, for $h\in N_{H(\CCnr)}(\Gnr)$, we now have that $h$ admits only finitely many local actions, and so $h\in\Rnr$ and can be represented by a minimal initial transducer $${A}_{q_{0}}\seteq(\rd,\Xn,R,S,\pi_A,\lambda_A,q_0).$$  Now, for $\nu\in\Wnr$ minimality ensures us that the local map $h_\nu:\CCn\to\CCn$ is representable by the initial transducer $A_{\pi_A(\nu,q_0)}$ (since $\pi_A(\nu,q_0)$ is not a state of incomplete response).  That is, we have $h_\nu=h_{A_{\pi_A(\nu,q_0)}}$.  We now introduce simplified notation, reflecting the perspective that the local maps of $h$ are determined by the states of $A_{q_0}$.
\begin{notation}\label{notn-localMapAtState}
Suppose $h\in\Rnr$ is represented by the minimal initial transducer $${A}_{q_{0}}\seteq(\rd,\Xn,R,S,\pi_A,\lambda_A,q_0)$$ and that $q\in Q=R\sqcup S$.  By the notation $h_q$ we will mean the local map $h_{\nu}$ where $\nu\in\Wnre$ is such that $\pi_A(\nu,q_0)=q$.
\end{notation}
Note that we will also use the notation $h_q$ to represent the local map $h_\nu:\CCn\to\CCn$ where $\nu\in \Wne$ and $A_{q_0}=\left\{\Xn,Q,\pi_A,\lambda_A,q_0\right\}$ is a finite minimal initial transducer representing a rational endomorphism $h:\CCn\to\CCn$.
%Notice that in the above setting the  actions in the set
%$\{ h_{\dot{0}0^{a-1}\concat \nu_i} \mid \nu_i \in B\}$ are the only actions of $h$
%which appear in levels of $\Wnre$ greater than $k+3$.

%Indeed, we have rediscovered that a minimal transducer representing $h$ will have synchronization  at level $k$, but from a different perspective!
%\marginpar{really ?}

%%%%%%%%%%%%%%%%%%%%%%%%%%%%%%%%%%%%%%%%%%%%%

\section{Finding our Place in the Rational Group $\Rnr$}\label{sec:synchronizing}
Corollary \ref{cor:NormAut} shows that $\Anr\cong N_{H(\CCnr)}(\Gnr)$ where $H(\CCnr)$ is the full group of homeomorphisms $\Homeo(\CCnr)$.  Meanwhile, Theorem \ref{finiteActions} shows that any element $\phi\in N_{H(\CCnr)}(\Gnr)$ is a homeomorphism of $\CCnr$ that admits only finitely many types of local actions.  Thus, Theorem \ref{GrigThm} then implies  that any such homeomorphism $\phi$ is actually an element of $\Rnr$, since it is a homeomorphism that can be represented by a (non-degenerate) finite transducer.  In this section we complete the proof of Theorem \ref{thm:MainTheorem}.  That is, given $\phi\in\Homeo(\CCnr)$, then $\phi\in\Anr$ if and only if $\phi$ is representable by a finite bi-synchronizing transducer $A_{q_0}$.

We note that the arguments of Section \ref{sec:finitely-many-local-actions} can be strengthened, using the ability of  $\Gnr$ to move small basic open sets freely, to show that if $\phi\in\Anr$ is represented by a finite transducer $A_{q_0}$ then in fact $A_{q_0}$ is strongly synchronizing.  We give a combinatorial argument here, with the purpose of preparing the reader for the fully combinatorial discussions occurring in our analysis of the outer automorphism groups $\Onr$.

\subsection{Getting in sync}

Here we formally define the strongly synchronizing property of transducers.  The central idea of the following definition is that given any large enough input word, the active state resulting from reading the input word is fully determined only by that initial word.
\begin{definition}
Let $A=\langle X_i,X_o,Q_A,\pi_A,\lambda_A\rangle$ be a transducer, $m$ be a natural number, and
$\mathfrak{s} :X_i^m\to Q_A$ be a function so that if $w\in X_i^m$ and $q\in Q_A$, then $\pi_A(w,q) = \mathfrak{s} (w)$.  In this case we call $\mathfrak{s}$ a \emph{synchronizing map for $A$} and we say that \emph{$A$ is synchronizing at level $m$} or simply \emph{strongly synchronizing}.  If $A$ is initial and $A$ represents a homeomorphism $h$ on $\CCn$ and $h^{-1}$ also admits a representative initial transducer $B_{q_0}=\langle X_o,X_i, Q_B, \pi_B,\lambda_B,q_0\rangle$ which is  synchronizing at level $m$, then we say \emph{$A$ is bi-synchronizing (at level $m$).}
\end{definition}

Note that our language is not parallel in the sense that we use ``bi-synchronizing'' (instead of, e.g., ``bi-strongly-synchronizing'') for a transducer  with the property that it and its inverse are both strongly synchronizing.  We have chosen this language as it seems easier on the reader.
In agreement with the use in automata theory we say that a word $w \in X_i^m$ is a {\it synchronizing word} or alternatively a \emph{reset word},  if $\pi (w, q)$ does not depend on the state $q$.
Thus, $A$ is synchronizing at level $m$ whenever all words of length $m$ are reset words.

\emph{The above definitions extend to transducers acting on $\CCnr$ in the obvious way, using all valid strings of length $m$} (if a string begins with a letter from $\rd$, then it must be processed from the initial state).

\begin{remark}The transducer in Figure \ref{fig_smallTrans} of the Introduction is bi-synchronizing at level $2$ and represents a homeomorphism of $\mathfrak{C}_{3,2}$.
\end{remark}

The reader can also easily verify the points of the following remark for general transducers and also our more specific transducers which take inputs from $\CCnr$.
\begin{remark}
 \begin{enumerate}
\item Suppose $A$ is a transducer and $m$ is a natural number so that $A$ is synchronizing at level $m$.  Then for all natural $n>m$, we have that $A$ is synchronizing at level $n$.
\item Suppose $A=(\rd,\Xn, R, S, \pi,\lambda,q_0)$ represents a homeomorphism $h\in \Homeo(\CCnr)$ and is synchronizing at level $m$ for some positive integer $m$ with synchronizing map $\mathfrak{s}$ taking all inputs of length  $m$ to states in $Q := R\cup S$.  Then we have the following:
 \begin{enumerate}
\item for all
{$q_1 \in Q$}
and $q_2\in \Image (\mathfrak{s})$, there is a nontrivial word $w$ so that $\lambda(w,q_1) = q_2$,
\item for all $q\in \Image (\mathfrak{s})$ and all $w\in \Wn$, we have $\pi (w,q)\in \Image (\mathfrak{s})$, and
\item $\Image (\mathfrak{s})\subseteq S$.
 \end{enumerate}
 \end{enumerate}
\end{remark}

The image of the synchronizing map is therefore an inescapable set of states in its transducer.
In the remainder of this section, our initial transducers will all be given as acting on some space $\CCnr$, but all of the definitions apply similarly in the context of the original rational groups $\{\Rn\}_{n>1}$ of GNS.

Again, let $A_{q_0}=(\rd,\{0,1,\dots,n-1\},R,S,\pi,\lambda,q_0)$ be a transducer that is synchronizing at level $m$.  We call the maximal sub-transducer of $A_{q_0}$ which uses $\Image(\mathfrak{s})$ as its set of states the \emph{core of $A$}, which we denote as $\Core(A)$. (Note  that we drop the reference to the initial state $q_0$ in this notation, as the core is independent of initial state if $A_{q_0}$ is strongly synchronizing.)  Observe that $\Core(A)$ is a non-initial transducer in its own right with transition function $\widetilde{\pi}$ and output function $\widetilde{\lambda}$ where these functions are defined as the restrictions of the functions $\pi$ and $\lambda$ to the domain $\Xn\times \Image (\mathfrak{s})$.  That is, we have the induced transducer $$\Core(A) := \langle \Xn,\Image (\mathfrak{s}),\widetilde{\pi},\widetilde{\lambda}\rangle.$$
 Observe that if $A_{q_0}$ synchronizes at level $n$ for some integer $n$, then there is some $m\leqslant n$ so that $\Core(A)$ synchronizes at level $m$.

\begin{remark}
Note that for a strongly synchronizing transducer $$A_{p_0}=(\rd,\{0,1,\dots,n-1\},R_A,S_A,\pi_A,\lambda_A,p_0)$$ and an input letter $x$ that can be read from a state in $\Core(A_{p_0})$, there is a unique state $q$ of $\Core(A_{p_0})$ which satisfies $\pi_A(x,q)=q$.

The reader may easily verify that if $$B_{q_0}=(\rd,\{0,1,\dots,n-1\},R_B,S_B,\pi_B,\lambda_B,q_0)$$ is another strongly synchronizing transducer with $\Core(A_{p_0})$ strongly isomorphic to $\Core(B_{q_0})$, then this strong isomorphism of the cores is unique.
\end{remark}

\begin{theorem} \label{Thm:SynchronizingEndosMakesAMonoid} Let $\Mnr$ be the set of all endomorphisms $\tau:\CCnr\to\CCnr$ which are representable by strongly synchonizing transducers with all outputs finite.  The set $\Mnr$ forms a monoid under composition of maps.
\end{theorem}
\begin{proof}
First observe that if $\phi,\rho$ are two endomorphisms of $\CCnr$ that are representable by strongly synchronizing transducers then these transducers are $\omega$-equivalent to finite transducers, since the core of each of these transducers is finite, and one can only visit finitely many states on directed paths in these transducers from the initial states to the cores (else these transducers would admit accessible cycles containing states outside of their core states, which would violate the strong synchronizing condition).  Therefore, we will assume in the remainder of this argument that we have two initial transducers $A$ and $B$ representing the maps $\phi$ and $\rho$ respectively, which are minimal and strongly synchronizing, and so finite.  We will denote the output and transition functions of these transducers as $\lambda_A,\lambda_B,\pi_A,$ and $\pi_B$ in the obvious fashion.

We follow the direct product construction in \cite{GNSenglish} of the composition of two endomorphisms $\phi$ and $\rho$ in $\Mnr$, to show the result can be represented by a finite transducer with set of states given as the product sets of the states of $A$ and $B$, and with all outputs finite.  From the general product, we pass to its minimised equivalent transducer following this outline: first we may need to remove some inaccessible states, secondly, we might need to modify some (finite) outputs to remove states of incomplete response, and thirdly, we may need to identify some equivalent states (which can be seen as restricting futher within the set of direct-product states).

We claim that the resulting transducer $C$ is also strongly synchronizing.

To this end observe firstly that the resulting transducer (and as well the transducers $A$ and $B$) has no cycles of states with $\varepsilon$ outputs, lest the transducer map a point in the Cantor space $\CCnr$ to a finite string.  In particular, there is a number $N$ so that from any state $q$ in $A$, if we read an input of length $N$, the output will be at least length one.

Next, suppose that $A$ synchronizes at level $k$ while $B$ synchronizes at level $m$.  Increase $N$ if necessary, so that $mN>k$, notice that $A$ still has the property that upon reading an input of length $N$ from any state $q$, the output will be a non-trivial word.  We claim that $C$ synchronizes at level $mN$.

The reader can see this by considering how the output and transition functions of $C$ are defined.  If one considers some state $(a,b)$ as the active state for $C$, where $a$ is a state in $A$ and $b$ is a state in $B$, then for each letter $j\in\Xn$ we have $\lambda_C(j,(a,b)) = \lambda_B(\lambda_A(j,a),b)$ while $\pi_C(j,(a,b))=(\pi_A(j,a),\pi_B(\lambda_A(j,a),b)$.  As the transitions in $A$ synchronise on inputs of length $k<mN$, then as the first coordinate of the transitions of $C$ mirrors the transitions of $A$, we see that the first coordinate is completely determined by an input of length $mN$.  Similarly, as $A$ must produce a word of length at least $m$ on reading any word of length $mN$ from any state, and as the second coordinate of the transitions of $C$ mirrors the transitions of $B$ over the outputs of $A$, we see that the second coordinate must be synchronised as well, and our claim is supported.

Therefore, it is the case that $\Mnr$ is at least a semigroup, however, the set $\Mnr$ contains an endomorphism representing the identity map on $\CCnr$ (since this can be represented by a two-state, synchronous and strongly synchronizing transducer), and thus $\Mnr$ is a monoid.
\end{proof}

\begin{remark}\label{rem:coreInProductOfCores}
Observe that for two endomorphisms of $\Mnr$, represented by strongly synchronizing transducers $A$ and $B$ with all outputs finite, the transducer representing the product of these endomorphisms has as set of states a subset of the product of the states of $A$ and of $B$, and has its core occuring over a subset of the states arising in the product of the states in the cores of $A$ and $B$.
\end{remark}

\begin{notation}
We denote by $\Snr$ the set of all homeomorphisms of $\CCnr$ which are representable by strongly synchronizing transducers.  We further denote by $\Bnr$ the set of all homeomorphisms of $\CCnr$ which are representable by bi-synchronizing transducers.
\end{notation}

We observe in passing that $\Bnr\subset \Snr = \Mnr\cap \Rnr$.

Theorem \ref{Thm:SynchronizingEndosMakesAMonoid} has the following corollary.
\begin{cor}
The subset $\Bnr$ of $\Rnr$ forms a subgroup of $\Rnr$ under composition of homeomorphisms.
\end{cor}
We call the group $\Bnr$ the \emph{group of bi-synchronizing homeomorphisms of $\CCnr$}, or correspondingly \emph{the bi-synchronizing group} when the Cantor space being acted upon is clear.

We point out, from the proof of Theorem \ref{Thm:SynchronizingEndosMakesAMonoid}, that if one simply takes the full transducer product of two minimal transducers representing elements of $\Bnr$, then the result will be a connected transducer.  However, it can happen that some of the states of that product will be inaccessible.  Indeed, it can happen that a state corresponding to the product of two core states also might not be accessible (although, one can always read some input starting from this state to get into the core of the product transducer).  Thus, after computing any such general product, it is practically of great value to minimize the resulting transducer.

\subsection{Properties of strongly synchronizing and bi-synchronizing transducers}
In this subsection, we take a detour to unearth a more detailed structure theory for minimal, strongly synchronizing (or bi-synchronizing) transducers.

We begin our investigations by looking into how synchronization combines with the basic methods of minimization.
\begin{lemma}
Suppose that $0<m\in \N$ and that $A$ is an initial transducer that  represents a self-homeomorphism $\phi$ of $\CCnr$ and which synchronizes at level $m$.  Let $B$ be the minimal transducer $\omega$-equivalent to $A$ produced by the minimisation algorithm.  Then $B$ has finitely many states and synchronizes at level $m$ as well.
\end{lemma}

\begin{proof}
We suppose that $0<m\in \N$ and that $A_{q_0}=(\rd,\{0,1,\ldots ,n-1\}, R_1,S_1,\pi_A, \lambda_A,q_0)$ is an initial transducer that synchronizes at level $m$  and which represents a self-homeomorphism $\phi$ of $\CCnr$ for some $0<r<n\in\N$.

We consider the minimisation algorithm as discussed in Section \ref{s_reduction}.

The first operation is to remove states of incomplete response.  Recall that a state $q$ has incomplete response to an input letter $x$ if the word $z=\lambda(x,q)$ is smaller than the guaranteed eventual output of $q$ on reading long enough inputs with initial letter $x$.  As our initial states cannot be in any cycle within the transducer (since the $\rd$ letters are only processed from one state which is never revisited) the process of removing states of incomplete response is simply to inductively adjust all outputs from accessible states by ``back-propagating'' guaranteed responses as far as possible. Thus, transitions in the new transducer mirror transitions in the old transducer, and the synchronization level is preserved.

The next operation is to remove inaccessible states. It is immediate that this does not increase the synchronization level (it could decrease the synchronization level, if we remove an inaccessible state $q$ from which the transitions into $\Core(A)$ require long inputs).  Also, we observe that the number of states has not increased, and that $q_0$ is still the initial state of the resulting transducer, which we will denote by $A_1$.

We can observe that the transducer $A_1$ must now have only finitely many states:  the states in the core, (which is a finite set), and all the states one can visit on the way to the core from the initial state (again, a finite set, since after we have read an input of length $m$, we are in the core).

The third operation is to identify each set of equivalent states to a single state.   If two states $q_1$ and $q_2$ are to be identified, then the resulting state will be the image under the new synchronizing map of any word that would have resulted under the old map in either of the two initial states. If an input letter $x$ is processed by $A_1$ at the two states $q_1$ and $q_2$, then we will have that $\lambda_{A_1}(x,q_1)=\lambda_{A_1}(x,q_2)$ and while the two states $\pi_{A_1}(x,q_1)$ and $\pi_{A_1}(x,q_2)$ might be different in $A_1$, they will induce identical local maps, so $\pi_{A_1}(x,q_1)$ and $\pi_{A_1}(x,q_2)$ will also have to have been identified.  Thus, the synchronization level of the resulting transducer might become shorter, but it cannot be increased by this process.
\end{proof}

Grigorchuk, et al., in \cite{GNSenglish}, argue the uniqueness (up to what we call strong isomorphism in this article) of any two minimal finite transducers representing the same homeomorphism of Cantor space.  Therefore, we introduce the following notation.
\begin{notation}
Let $h\in(\Rnr\cup\Rn)$.  By $A_h$ we will mean a minimal initial finite transducer representing $h$.
\end{notation}

\begin{definition}
We set the notation
\[
\Onr\seteq\{Core(A_g)\mid g\in\Bnr\}
\]
for the set of cores of the minimal transducers representing elements of $\Bnr$, up to equivalence under strong isomorphism (we suppress this equivalence in the notation).
\end{definition}

\begin{definition}\label{lem:OnrProd}
Let $T_1, T_2\in\Onr$.  We define a product $T_1T_2\in\Onr$ as follows:
\begin{enumerate}
\item Compute the full transducer product $T_1\ast T_2$.
\item \label{def:remIncResp nr} Pick any state $q$ of $T_1 \ast T_2$. Apply the ``Remove States of Incomplete Response'' algorithm in the paper \cite{GNSenglish} to the initial transducer $(T_1\ast T_2)_{q}$ to produce an initial transducer $\overline{(T_1T_2)}_{{q}'}$.
\item \label{def:corePass nr} Pass to the core of $\overline{(T_1T_2)}_{{q}'}$ to produce the result $\overline{(T_1T_2)}^{\,\circ}$.
\item \label{def:idEquivStates nr} Identify equivalent states of $\overline{(T_1T_2)}^{\,\circ}$ to produce the result in $\Onr$, which we denote by $T_1T_2$.
\end{enumerate}
\end{definition}

We have the following structural lemma about the core of (minimal) transducers representing elements of $\Bnr$.
\begin{lemma} \label{lem:core structure}For any $g,h\in\Bnr$, we have:
 \begin{enumerate}
\item \label{point:core containment} $\Core(A_g)\Core(A_h)$ is strongly isomorphic to $\Core(A_{gh})$;
\item \label{point:core equivalences}if $\Core(A_g)$ is strongly isomorphic to $\Core(A_h)$, then there is an element $v\in G_{n,r}$ so that $gv=h$.
\item \label{point:cosetCore} if $g$ and $h$ are in the same coset of $\Gnr$ in $\Bnr$, then $\Core(A_{g})$ is strongly isomorphic to $\Core(A_{h})$.
\end{enumerate}
\end{lemma}

\begin{proof}\leavevmode

\begin{enumerate}
\item The first point is essentially a direct computation. Observe that the full transducer product of the initial transducers $A_{g}$ and $A_{h}$ is the initial transducer $A_{g}*A_{h}$ which is $\omega$-equivalent to $A_{gh}$ by \cite{GNSenglish}. Moreover, it contains as a sub-transducer the full transducer product of $\Core(A_g)$ with $\Core(A_h)$. The last three paragraphs of the proof of Theorem \ref{Thm:SynchronizingEndosMakesAMonoid} shows that $A_{g} \ast A_{h}$ is strongly synchronizing, and    $\Core(A_{g} \ast A_{h})$ is contained in $\Core(A_{g}) \ast \Core(A_{h})$. Thus $A_{gh}$, the minimal transducer representing the initial transducer $A_{g}\ast A_{h}$ contains a transducer strongly isomorphic to $\Core(A_g)\Core(A_h)$. However since $A_{gh}$ is synchronizing, it follows that $\Core(A_{gh})$ is strongly isomorphic to $\Core(A_g)\Core(A_h)$.

\item
The second point will follow from the characterisation of $\Gnr$ as the subgroup of $\Bnr$ consisting of those homeomorphisms which are represented by reduced transducers with core a single state which acts as the identity.

The argument below relies on the inversion algorithm of GNS (Proposition 2.21 of \cite{GNSenglish}). This algorithm, at the request of early readers of this manuscript, is given in detail in Appendix~\ref{appendixInversion} and applies to any finite non-degenerate transducer $A_\phi$ inducing a homeomorphism $\phi$ of Cantor space.  %That algorithm builds, for each state $q$ of the states of $A_{\phi}$, a set of ``inverse'' states of the form $p^{-1} = (w,q)$ where $w$ is a prefix of an element of $\Image(h_q)$ where also ${L}_{q}(w) = \epsilon$. (Recall our Notation \ref{notn-localMapAtState} that $h_q$ is a ``local map for $h$,'' and further that \cite{GNSenglish} uses the notation ${L}_{q}(w)$ to represent the greatest common prefix of the set $(U_w)(h_{q})^{-1}:= \{ \delta \in \CCn \mid  (\delta)h_{q} = w\concat\delta', \delta' \in \CCn\}$.  We retain the GNS notation ${L}_{q}$ here.)
Note that for a given $q$ the injectivity of $h_q$ implies that the set of such valid $w$ is finite.  Similarly, the transitions and outputs for the minimal transducer representing the inverse function depend precisely on the transitions and outputs of the original transducer.  In particular, since $A_g$ and $A_h$ (in our context) have strongly isomorphic cores, the transducers $A_{g^{-1}}$ and $A_{h^{-1}}$ created by this algorithm will have strongly isomorphic cores.  It is now the case that the product transducer $A_{g^{-1}}\times A_{h}$ must have its core strongly isomorphic to the core of the product $A_{h^{-1}}\times A_{h}$.  From this it follows that $A_{g^{-1}h}$ has core acting as the identity, thus $v\seteq g^{-1}h$ must be in $\Gnr$.
\item To see the third point, observe that if $g$ and $h$ are  in the same  coset, then there is an element $f \in \Gnr$ such that $g f = h$. Noting that, as $f \in \Gnr$, $\Core(A_{f})$ is the single state identity transducer.  We have that  $\Core(A_{g}) \Core(A_{f})$ is strongly isomorphic to $\Core(A_{g f})$ by Lemma \ref{lem:core structure}.\ref{point:core containment}.  However, as $\Core(A_{f})$ is the single state identity transducer, we have $$\Core(A_{g})\Core(A_f)=\Core(A_{g}) \Core(A_{1_{G_{n,r}}})$$  which (by Lemma \ref{lem:core structure}.\ref{point:core containment}) is strongly isomorphic to $\Core(A_g)$ (here we are taking $1_{\Gnr}$ as the identity element of $\Gnr$). Thus, it follows that  $\Core(A_{h})=\Core(A_{g f})$ is strongly isomorphic to $\Core(A_{h})$.

\end{enumerate}
\end{proof}

\begin{lemma}\label{lem:assoc}
Let $f,g,h\in \Bnr$, then $(\Core(A_f)\Core(A_g))\Core(A_h)$ is strongly isomorphic to $\Core(A_f)(\Core(A_g)\Core(A_h))$.
\end{lemma}
\begin{proof}
This follows from Lemma \ref{lem:core structure}.\ref{point:core containment} and the associativity of composition.
\begin{IEEEeqnarray*}{rCl}
(\Core(A_f)\Core(A_g))\Core(A_h)&=_{si}&\Core(A_{fg})\Core(A_h) \\ &=_{si}& \Core(A_{(fg)h})
\end{IEEEeqnarray*}
\begin{IEEEeqnarray*}{rCl}
\Core(A_{f(gh)})&=_{si}&\Core(A_{f})\Core(A_{gh})\\ &=_{si}&\Core(A_f)(\Core(A_g)\Core(A_h))
\end{IEEEeqnarray*}
{while $A_{(fg)h}=_{si}A_{f(gh)}$.}
\end{proof}

\begin{prop}
The set $\Onr$ together with the multiplication of Definition~\ref{lem:OnrProd} is a group.
\end{prop}
\begin{proof}
Observe that by Lemma \ref{lem:assoc} the  multiplication is associative, and, as in that proof, $\Onr$ admits an identity element and inverses under that multiplication using Lemma~\ref{lem:core structure}.\ref{point:core containment} and the fact that $\Bnr$ is a group.
\end{proof}

\begin{theorem}
The map $\mbox{Core}$ from $\Bnr$ to $\Onr$ defined by $g \mapsto \Core(A_{g})$ is a group homomorphism with kernel $\Gnr$.
\end{theorem}
\begin{proof}
 Recall that Lemma~\ref{lem:core structure}.(\ref{point:core equivalences}--\ref{point:cosetCore})  establishes a bijection between the  cosets of $\Gnr$ and the group $\Onr$, where the bijection is induced by the function $\mbox{Core}: \Bnr \to \Onr$ (passing to the core of the minimal representative transducer for a group element). Further, Lemma~\ref{lem:core structure}.\ref{point:core containment} shows that the function $\mbox{Core}$ is actually a group homomorphism.

All that remains then is to observe that the identity of $\Onr$ is the single state identity transducer, which is precisely the image of the coset $\Gnr$.  Thus $\Gnr$ is normal in $\Bnr$ and $\Onr\cong \Bnr/\Gnr$.
\end{proof}

\begin{cor}
The group $\Gnr$ is a normal subgroup of $\Bnr$ and the quotient group $\Bnr/\Gnr$ is isomorphic to $\Onr$.
\end{cor}

 \begin{remark}
Considering the nature of $\Onr$, we should mention that the elements can be far stranger than one might expect.  For instance,
 Figure \ref{Fig: example showing dependence on r} demonstrates an element of $\Onk{4}{3}$ which has no state that
 acts on $\CCk{4}$ as a homeomorphism, even though the element admits a completion to an element of $\Bnk{4}{3}$ which of course does act as a homeomorphism
 of $\CCk{4,3}$. This is one example of why having a combinatorially defined product for elements of $\Onr$ without having to pass up to $\Bnr$ is desirable.
\end{remark}

\subsection{Invertible rational endomorphisms of $\Gnr$ are strongly synchronizing}

In this subsection, we show that
\begin{quote}
every automorphism of $\Gnr$ is in $\Bnr$.
\end{quote}
Note this implies that $\Aut(\Gnr)\cong \Bnr/K$, where $K$ is the kernel of the conjugation action of $\Bnr$ on $\Gnr$. Recall that the kernel of this quotient map is the trivial subgroup (for any non-identity element $h \in \Homeo(\CCnr)$, we can find an element in $\Gnr$ which fails to commute with $h$ (e.g., a transposition of well chosen cones will do)). Thus, to prove Theorem \ref{thm:MainTheorem}, we only need to verify the displayed statement above.

 We require some further language.

\begin{definition}
Let $A_{q_0}=(X_A,Q_A,\pi_A,\lambda_A,q_0)$ be an initial, finite transducer, and let $p \in Q_{A}$ a  state such that there is a path from $q_0$ to $p$ with non-empty output for $A_{q_0}$, we call such states \emph{non-trivially accessible (in $A_{q_0}$)}.
\end{definition}

We also need to expand the notion of inverse homeomorphisms to transducers.

\begin{definition}
Let $A_{q_0}=\langle X_A,Q_A,\pi_A,\lambda_A,q_0\rangle$ be a minimal, initial, finite transducer representing a self-homeomorphism $h_{q_0}$ of the Cantor space $X_A^{\omega}$.  Let $B_{p_0^{-1}}=\langle X_B,Q_{B},\pi_{B},\lambda_{B},p_0^{-1}\rangle$ represent the finite minimal initial transducer (with an initial state denoted by $p_0^{-1}$) so that the induced homeomorphisms $h_{q_0}$ and $h_{p_{0}^{-1}}$ are inverse (so, in particular, $X_A=X_B$).  We will refer to $A$ and $B$ as \emph{inverse transducers}.
\end{definition}

If $A_{q_0}$ is a given minimal initial finite transducer, and one creates its inverse transducer $B_{p_0^{-1}}$ as in the definition above,  we will use a convention of writing elements in the set $Q_B$ with a superscript ``$-1$'' (e.g., a state denoted by $p^{-1}$ would represent a state in $Q_B$, whereas a state denoted simply as $q$ would represent a state in $Q_A$).  This is to avoid confusion in our arguments below.  Note that generally there is not a bijection between $Q_A$ and $Q_B$. See the appendix on inversion for details.

We can now prove the main result of this subsection.  The following lemma provides a key insight into why automorphisms of $\Gnr$ are strongly synchronizing.    Note in advance that the proof is lengthy: we will attempt to break it into meaningful stages to support the reader.

When we conjugate an element $\alpha\in G_{n,r}$ by a rational homeomorphism $\phi$, the element $\alpha$ will apply a prefix-exchange map ``in the middle'' of the conjugation (so, any triple of resulting states from the product has middle state acting as the identity, for long enough inputs).  Thus, any such conjugations arbitrarily pair states of the inverse conjugator to states of the conjugator in the resulting action, as we can ignore the identity state in the middle.  The resulting product will be in $G_{n,r}$ only if all such pairings result in ``local prefix exchange maps''.  The following lemma proves the converse also holds.  Recall for the proof below Definition \ref{def:rotation} of a rotation of a word.

\begin{lemma}\label{lem:MainTheoremAutIsSynch}
Let $A_{q_0}=\langle \Xn,Q_A,\pi_A,\lambda_A,q_0 \rangle$ be a minimal, initial, invertible finite transducer.  Let $p_0^{-1}$ represent the initial state of the minimal transducer $$B_{p_0^{-1}}=\langle \Xn,Q_B,\pi_B,\lambda_B,p_0^{-1}\rangle$$ which is the inverse transducer of  $A_{q_0}$. Suppose that for all states $p^{-1} \in Q_{B}$ and for all $q \in Q_A$ we have the initial transducer $B_{p^{-1}}A_{q}$ admits a complete prefix code $\vec{x}_{p^{-1},q}$ for $\Wne$ so that for all $\gamma \in \vec{x}_{p^{-1},q}$ there is $\delta_\gamma\in \Xn^{\ast}$ so that for all $\Gamma\in \Xn^{\omega}$ we have $(\gamma \Gamma)B_{p^{-1}}A_{q} = \delta_\gamma \Gamma$. Then $A_{q_0}$ is strongly synchronizing.

\end{lemma}
\begin{proof}
We shall show that there exists $\mathfrak{t}$, a finite set of synchronizing words for $A_{q_0}$, such that every word in $\Xn^{\omega}$ has a prefix belonging to $\mathfrak{t}$. We shall call such a set of words a base. (Note that the base $\mathfrak{t}$ we produce might admit pairs of words which are comparable.) The maximum length of a word in our base will be a synchronizing length for $A_{q_0}$.

One of the main ideas of the proof which follows is that if $A_q$ and $B_{p^{-1}}$ eventually act as inverses (on an infinite tail), independent of our choices of $p^{-1}$ and $q$, then after reading a long enough output of $B$  the transducer $A$ must start processing in such a way as to revert input strings to what they originally were when read into $B$.  This will require that after reading long enough inputs from any state of $A$, then the current state of $A$ should become determined.

{\flushleft {\it Recording parameters:}}\\
The following part records the length of input required so that the product of $B_{p^{-1}}$ with $A_q$ acts on long enough inputs as a prefix exchange map, for any choices of $p^{-1}$ and $q$.  The three parameters to note are $r\leq \bar{r}$ and $l$, where we will have that any input of length $\bar{r}$ will result in a prefix exchange in the product with final output word of length at least $l$.

Using the finiteness of the sets of states of the transducers $A_{q_0}$ and $B_{p_0^{-1}}$, set $r$ to be the minimal natural number so that for any $q \in Q_A$ and $p^{-1} \in  Q_B$, and for each $\gamma \in \Xn^r$ there is $\delta_\gamma\in \Wn$ so that for all $\Gamma\in \Xn^{\omega}$ we have  $(\gamma \Gamma)B_{p^{-1}}A_{q} = \delta_\gamma \Gamma$. (Note that if we were to take a larger value for $r$, then for the correspondingly longer words $\gamma$, the dependent $\delta_\gamma$ will become some extension of the original $\delta_\gamma$ by some suffix of $\gamma$.)  Further determine $l$ as the minimal length such that if, for all  $\gamma \in \Xn^{r}$ and some $q \in Q_A$ and $p^{-1} \in Q_B$  and for all $\Gamma\in \Xn^{\omega}$ we have $(\gamma \Gamma)B_{p^{-1}}A_{q} = \delta_\gamma \Gamma$, then $|\delta_\gamma| \le l$. That is, $l$ is the maximal length of an output prefix for all of the resulting (determined by value of $r$) prefix exchange maps.  Finally, let $\bar{r} \ge r$ be such that for any word $\bar{\gamma} \in \Xn^{\bar{r}}$ and for any $q \in Q_A$ and  $p^{-1} \in Q_B$,  we have $|(\bar{\gamma})B_{p^{-1}}A_{q}| \ge l$ (we need this as possible incomplete response in the composite might mean that reading an input $\eta$ of length $r$ could fail to produce a completely determined output $\zeta$ (which would eventually appear if we read longer input strings)).

{\flushleft {\it Finding a base for $A_{q_0}$ covering a local cone $U_{y_1\ldots y_t}$}:}

The following argument is somewhat convoluted.  First we use $r$ and $\bar{r}$ to build a partial base for $A_{q_0}$ covering $U_{y_1\ldots y_t}$.  We discover our constructed partial base may fail to completely cover $U_{y_1\ldots y_t}$.  Then, by choosing a larger $\bar{r}$, we are able to repeat our construction to grow our partial base so that it actually covers  $U_{y_1\ldots y_t}$.

\vspace{0.1 in}
{\centerline {(1): {\it Finding a base for $A_{q_0}$ partly covering the local cone $U_{y_1\ldots y_t}$}}}
\vspace{0.1 in}

Now fix a state $p^{-1}$ in $Q_B$. Since $p^{-1}$ is a state of the minimal inverse transducer for $A_{q_0}$, by the GNS algorithm $p^{-1} = (w,p)$ for some state $p \in Q_A$ and a word  $w$ such that $w$ is a prefix of an element of $\Image(p)$ and  ${L}_{p}(w) = \epsilon$.
(Recall that ${L}_{p}(w)$ is the greatest common prefix of the set $(U_{w})(h_{p})^{-1}:= \{ \delta \in \CCn \mid (\delta)h_{p} = w\delta', \delta' \in \CCn\}$.)
 Let $x_1x_2 \ldots x_r x_{r+1} \ldots x_{\bar{r}} \in \Xn^{\bar{r}}$, and let $y_1, \ldots, y_{t}$ be such that $(x_1x_2 \ldots x_{\bar{r}})B_{(w,p)} = y_1 \ldots y_{t}$.
 Let us now focus on the cone at $y_1 \ldots y_{t}$, as described in the title of this part of the proof.

  We do not yet know that $y_1 \ldots y_{t}$ is a synchronizing word for $A_{q_0}$: it may be the case that there are two states of $A_{q_0}$ such that, if we read $y_1 \ldots y_{t}$ from these states, the resulting states will be unequal.  We show that there is a base of extensions of $y_1 \ldots y_{t}$ which synchronise $A_{q_0}$.

Let $q_1$ and $q_2$ be any arbitrary states of  $A$. There are minimal indices $r_1$ and $r_2$, both less than or
equal to $r$, so that for all words $ \widetilde{x}_{r_{1}+1} \ldots \widetilde{x}_{\bar{r}}$ and
$ \widetilde{x}_{r_{2}+1} \ldots \widetilde{x}_{\bar{r}}$ in $\Xn^{*}$, and for any $\xi \in \CCn$ we have $$(x_1x_2 \ldots x_{r_1} \widetilde{x}_{r_{1}+1} \ldots
\widetilde{x}_{\bar{r}}\xi) B_{(w,p)}A_{q_1} =  \delta_1\widetilde{x}_{r_{1}+1} \ldots
\widetilde{x}_{\bar{r}}\xi$$  and $$(x_1x_2 \ldots x_{r_2}\widetilde{x}_{r_{2}+1} \ldots
\widetilde{x}_{\bar{r}}\xi)B_{(w,p)}A_{q_2} =  \delta_2 \widetilde{x}_{r_{2}+1} \ldots \widetilde{x}_{\bar{r}}\xi$$
 for $\delta_1$ and $\delta _2 \in \Xn^*$ with $|\delta_1| \le l$ and $|\delta_2| \le l$. Let $$(y_1\ldots y_{t})A_{q_1} = z_1\ldots z_{l_1} x_{r_1+1} \ldots x_{r_1+i}$$ and $$(y_1\ldots y_{t})A_{q_2} = u_1\ldots u_{l_2} x_{r_2+1} \ldots x_{r_2+j}.$$

Notice that it cannot be the case that $$(y_1\ldots y_{t})A_{q_1} = z_1\ldots z_{l_1} x_{r_1+1} \ldots x_{\bar{r}}\rho$$ for some $\rho \in \Xn^{+}$, since picking a word  $\rho'$ which is incomparable to $\rho$ in its first letter we will then have
$$(x_1 \ldots x_{\bar{r}} \rho')B_{(w,p)}A_{q_1} =  z_1\ldots z_{l_1} x_{r_1+1} \ldots x_{\bar{r}}\rho \zeta$$ for some $\zeta$, which is a contradiction as by minimality of $l_1$  we must then have $\rho = \rho'$. The same argument is valid for $(x_1 \ldots x_{\bar{r}} \rho')B_{(w,p)}A_{q_2}$.

Therefore we may assume that $r_1 +i \le \bar{r}$ and $r_2 +j \le \bar{r}$.

Now let $\pi_B(x_1 \ldots x_{\bar{r}}, (w,p)) = (v,s)$. Recall
from the GNS inversion algorithm  that $wx_1\ldots x_{\bar{r}} -
\lambda_A(y_1 \ldots y_t, p) = v$ and $\pi_A(y_1\ldots y_t,p) =
s$. Since $A_{q_0}^{-1}$ has only finitely many states we may
further assume that $\bar{r}$ is chosen large enough so that $v$ is a suffix of $x_1 \ldots
x_{\bar{r}}$ (this is a convenient, but not necessary, assumption). This is because for each state $s'$
of $A$, there
are only finitely many words $w'$ such that $w'$ is a prefix of
an element of $\Image(s)$ and ${L}_{s}(w') = \epsilon$. Hence, we
may pick $\bar{r}$ such that for any
state $p'$ of $A$ and any word $\gamma' \in \Xn^{\bar{r}}$, we have $|\lambda_{A}(\lambda_{B}(\gamma, (w,p)), p' )|$ is
also greater than the maximum size of any word $w'$
such that $(w',s)$ is a state of $B$. Moreover we also have
${L}_{p}(wx_1\ldots
x_{\bar{r}}) = y_1 \ldots y_t$.  Let $wx_1 \ldots x_k =
\lambda_A(y_1 \ldots y_t, p)$ for
some $k \le \bar{r}$. Let $m$ be minimal such that for any state
$q$  of $A$ and any word $\alpha \in \Xn^{m}$ we have
$|\lambda_A(\alpha, q )| \ge \bar{r} - k$. Now notice that there
is a maximal set $\alpha = \{\alpha_1, \ldots, \alpha_{m_1}\}$ a
subset of $\Xn^{m}$ such that the greatest common prefix of
the elements of $\alpha$ is the
empty word, and such that $\lambda_A(\alpha_{a}, s) =
x_{k+1}\ldots x_{\bar{r}}\rho_{a}$ for some $\rho_{a} \in
\Xn^{\ast}$ and $1 \le a \le m_1$. This is because
${L}_{p}(wx_1\ldots x_{\bar{r}}) = y_1\ldots y_t$ and
$\pi_A(y_1\ldots y_t,p) = s$. Further notice that the set
$\alpha$ depends only on  $x_1\ldots x_{\bar{r}}$, $(w,p)$ and
$(v,s)$, i.e. it is independent of the choice of $q_1$ and  $q_2$.

Fix an $\alpha_a \in \alpha$. Let $r_a = \pi_A(\alpha_{a}, s)$.
Recall that $\Image(r_a)$ represents the image of the map $\lambda_A(*,r_a)$ from
$\CCn$ to itself. Observe that
$${L}_{s}(v\rho_{a}\Image(r_{a})) = \alpha_a \mbox{ and }
(\rho_{a}\concat\Image(r_{a}))B_{(v,s)} = U_{\alpha_a}.$$ Now let $t_1
:=\pi_A(y_1\ldots y_t, q_1)$ and  $t_2 :=\pi_A(y_1\ldots y_t,
q_2)$. Observe that $$(x_1\ldots x_{\bar{r}}\rho_a\concat \Image(r_a))B_{(w,p)}
= y_1\ldots y_tU_{\alpha_a}.$$ Now since
$(y_1\ldots y_{t})A_{q_1} = z_1\ldots z_{l_1} x_{r_1+1} \ldots
x_{r_1+i}$ and
$(y_1\ldots y_{t})A_{q_2} = u_1\ldots u_{l_2} x_{r_2+1} \ldots
x_{r_2+j}$ it must be the case for any
$\rho \in \CCn$ that the string $$x_{r_1+i+1}\ldots x_{\bar{r}}\rho_a$$ is a prefix of $\lambda_A(\alpha_{a}\rho, t_1)$. In a similar way  $x_{r_2+j+1}\ldots x_{\bar{r}}\rho_{a}$ is a prefix of $\lambda_A(\alpha_{a}\rho, t_2)$. Now since we assumed
that $A_{q_0}$ has no states of incomplete response, it must also be
the case that $\lambda_A(\alpha_{a}, t_{\ast}) = x_{r_{\ast} +
\sharp +1}\ldots x_{\bar{r}}\rho_{a}$
for $(\ast,\sharp) \in \{(1,i),(2,j)\}$. This is because
otherwise
$\Image(\pi_A(\alpha_{a}, t_{\ast}))$, for $\ast \in \{1,2\}$, will have a non-trivial common prefix contradicting the fact that $A_{q_0}$ has no states of incomplete
response. Now let
$\rho' \in \Image(r_a)$ be arbitrary, and let $(x_1\ldots
x_{\bar{r}}\rho_{a}\rho')B_{(w,p)} =
y_1\ldots y_t \alpha_a \rho$, then we must have that
$(\rho)A_{ \pi_A(\alpha_{a}, t_{1}) } =
(\rho)A_{ \pi_A(\alpha_{a}, t   _{2}) } = \rho'$. Since
$\rho'$ was arbitrary and
$(\rho_a \concat \Image(r_a))B_{(v,s)} = U_{\alpha_a}$
we have that for all $\rho'$ in $\CCn$ there is some $\rho$ in $\CCn$ such that this holds.
Thus, $\pi_A(\alpha_a,t_1)$ and $\pi_A(\alpha_{a},t_2)$
are $\omega$-equivalent.

Now as $q_1$ and $q_2$ were arbitrary states of $A_{q_0}$  we
must have, for any
pair $q'_1, q'_2$ of states of $A_{q_0}$, that $\pi_A(y_1\ldots
y_t\alpha_a, q'_1 )$  and $\pi_A(y_1\ldots y_t\alpha_a, q'_2 )$ are
$\omega$-equivalent (and so equal by minimality) for all $\alpha_a \in \alpha$.
Therefore the set of words $\{y_1\ldots  y_t\alpha_{a} \mid 1 \le a \le m_1\}$
is a synchronizing set of words.

\vspace{0.1 in}
{\centerline {(2): {\it Completing the base for $A_{q_0}$ covering the local cone $U_{y_1\ldots y_t}$}}}
\vspace{0.1 in}

Now we need to consider the elements of
$\beta:=\Xn^{m}\backslash\alpha$.  Note that these elements are independent of
the choice of $q_1$ and $q_2$. Let $\beta:= \{\beta_{1}, \ldots,
\beta_{m_2}\}$. Let $\beta_{b} \in \beta$ be arbitrary. In this
paragraph  we will show that there is a positive integer $m'_1$ and a set of synchronizing
words  $\{y_1\ldots y_{t}\beta_{b}\eta_{c} \mid 1 \le c \le m_1'\}$
where the $\eta_c$'s all have the same length and form a complete antichain for words with prefix $y_1\ldots y_{t}\beta_{b}$.  By assumption we have that $|\lambda_A(\beta_b, s))|
\ge  \bar{r} - k$. Let $x_{k+1}' \ldots x'_{\bar{r}}\rho'_{b} :=
\lambda_A(\beta_b, s)$. Now $x_{1} \ldots
x_{k}x_{k+1}'\ldots x'_{\bar{r}}\rho'_{b}$ has length
greater than or equal to $\bar{r}$. Therefore we may repeat the
arguments in the paragraph above with $|x_{1} \ldots
x_{k}x_{k+1}'\ldots x'_{\bar{r}}\rho'_{b}|$ in place of $|\bar{r}|$ and $x_{1} \ldots
x_{k}x_{k+1}'\ldots x'_{\bar{r}}\rho'_{b}$ in place of $x_{1} \ldots
x_{k}x_{k+1}\ldots x_{\bar{r}}$. Notice
that $\lambda_A(x_{1} \ldots x'_{\bar{r}}\rho'_{b},(w,p))$ is
either a prefix of $y_{1}\ldots y_t$ or contains $y_{1}\ldots
y_{t}$ as a prefix. Let $\lambda_A(x_{1} \ldots
x'_{\bar{r}}\rho'_{b},(w,p)) := y_{1}'\ldots y_{t'}'$, and let
$(v',s') := \pi_A(x_{1} \ldots x'_{\bar{r}}\rho'_{b},(w,p))$.
After repeating the argument, we end up with a new
set of synchronizing words, however this new set of
synchronizing words now contains a subset  $\{y_1\ldots
y_{t}\beta_{b}\eta_{c} \mid 1 \le c \le m_1'\}$ where the $\eta_c$'s
all have the same size and form a complete antichain for words with prefix $y_1\ldots y_{t}\beta_{b}$. Take the
union of the sets $\{y_1\ldots y_{t}\beta_{b}\eta_{c} \mid 1 \le c
\le m_1'\}$ and $\{y_1\ldots y_t \alpha_{a} \mid 1 \le a \le m_1\}$
and let $\mathfrak{u}(y_1\ldots y_t)$ denote the union.
Continuing in this way across all the $\beta_{b} \in \beta$, and
letting $\mathfrak{u}(y_1\ldots y_t)$ denote the total union at
each stage, we see that $\mathfrak{u}(y_1\ldots y_t)$ is a base for the clopen
set $U_{y_1\ldots y_t}$.

{\flushleft {\it From local to global}:}

Now to finish to proof it suffices  to construct a finite set $M \subset \Wne$, and $D\in \N$ satisfying the following conditions:
\begin{enumerate}
  \item[(A)] for every word $\gamma \in M$, $\mathfrak{u}(\gamma)$ exists and is a base of synchronizing words for $A_{q_0}$ for the clopen set $U_\gamma$,
  \item[(B)] if $N\seteq \cup _{\gamma\in M}\mathfrak{u}(\gamma)$ then for every element of $\nu \in \Wne$ of length greater than or equal to $D$ there is an element of $N$ which is a prefix of a (possibly trivial) rotation of $\nu$.
  \end{enumerate}

\vspace{0.1 in}
{\centerline {(1): {\it These conditions suffice.}}}
\vspace{0.1 in}

 Let us suppose that we have such $M$, $N$, and $D$.  Set $\mathfrak{t}$ to be the set of all words of length $D$.  We will show that such a set $\mathfrak{t}$ is a base of synchronizing words.

If each word of length at least $D$ contains a proper contiguous substring which is an element of $\mathfrak{u}(\gamma)$ for some $\gamma\in M$ then all strings of length at least $D$ contain a synchronizing word for $A_{q_0}$, so $D$ is a synchronizing length for $A_{q_0}$.  If not, and assuming $A_{q_0}$ is not strongly synchronizing, then there is a word $u=u_1u_2\ldots u_m$ and two distinct states $s$ and $t$ so that $\pi_A(u,s)=s$ while $\pi_A(u,t)=t$ (this follows by considering the sets of pairs of states that one gets when reading long strings of inputs from some initial pair $(s',t')$, as eventually some pair $(s,t)$ must be repeated due to the finiteness of the automaton underlying the transducer).  Now, there is some minimal $k$ so that $k\cdot m \geq D$.  We observe that $u^k$ has a rotation which has a prefix in $N$, but then we must have that reading $u^{2k}$ will be a synchronizing word for $A_{q_0}$, so that $s=t$, which is a contradiction.

\vspace{0.1 in}
{\centerline {(2): {\it Constructing $M$ and $D$ as described.}}}
\vspace{0.1 in}

  Let  $\mathscr{P}$ be the set of all non-trivially accessible states of $B_{(\epsilon,q_0)}$, observing that $\mathscr{P}$ is not empty.  For each $(w,p) \in \mathscr{P}$  let $\mathscr{O}((w,p)):= \{\lambda_{B}(\varphi, (w,p))\mid \varphi \in \Xn^{\overline{r}} \}$. The arguments above demonstrate that for each word $\varphi \in \Xn^{\overline{r}}$, and each state $(w,p) \in \mathscr{P}$, there is a set $\mathfrak{s}(\lambda_{B}(\varphi, (w,p)))$ of synchronizing words which is a base for $A_{q_0}$ over the clopen set $U_{\lambda_{B}(\varphi, (w,p))}$.  Set $M\seteq \{ \gamma \in \Wn \mid \exists (w,p) \in \mathscr{P} : \gamma \in \mathscr{O}((w,p)) \}$ and correspondingly set $N\seteq \{\zeta \mid \exists\gamma\in M,\zeta\in\mathfrak{u}(\gamma)\}$.

  Set the following values
\begin{eqnarray*}
C &\seteq& \max_{\gamma\in M} |\gamma|,\\
D &\seteq& \max_{\zeta\in N}|\zeta|.
\end{eqnarray*}

Note that $D\geq C$ as each word in $N$ is an extension of a word in $M$.  We are now ready to verify that $D$ and $M$ as determined satisfy our two properties.

Let $\delta \in \xn^*$ so that $|\delta|\geq D$. Note that there is a $\rho \in \CCn$ such that, $(\rho)h_{q_0}^{-1}=(\rho)h_{(\epsilon,q_0)} = \delta\delta\delta\ldots=\delta^\omega$. Let $\rho_1$ be the minimal prefix of $\rho$ such that $\pi_{B}(\rho_1, (\epsilon,q_0)) = (w,p) \in \mathscr{P}$. Let $\rho_2 \in \xn^{\overline{r}}$ be such that $\rho_1\rho_2$ is a prefix of $\rho$. Observe that since $\pi_{B}(\rho_1, (\epsilon,q_0)) = (w,p) \in \mathscr{P}$ then we must have  $\lambda_{B}(\rho_1, (\epsilon,q_0)) \ne \epsilon$.  Observe also (by definition) that $\lambda_{B}(\rho_2,(w,p))\in M$. Therefore, by the definition of $D$ and as $(\rho)h_{(\epsilon,q_0)} = \delta^\omega$, we have $\lambda_{B}(\rho_2,(w,p))$ is a prefix of a (possibly trivial) rotation of $\delta$ . Since $\delta \in \xn^*$ was chosen arbitrarily amongst words of length at least $D$, it follows that for every element $\nu \in \xn^*$ of length at least $D$ there is an element of $M$ which is a prefix of a possibly trivial rotation of $\nu$. Recall that for all $\gamma \in M$, $\mathfrak{u}(\gamma) \subseteq N$ is a base for $U_{\gamma}$, in particular, any word of length at least $D$ which has a prefix $\gamma$ in $M$, actually has a prefix in $\mathfrak{u}(\gamma)\subset N$.
\end{proof}
If one considers the proof above, as we pass into consideration of states of $B$ which appear on cycles of the automaton when we consider inputs to $B$ of the form $\delta^\omega$, it is clear that one can ignore the states of $A$ and of $B$ which can be reached without writing any output in the construction of the synchronisation function of $A$.  Note as well that the construction of the synchronisation function follows just as well in other contexts; the following remark extracts the necessary hypotheses for the construction of these particular synchronisation functions from the proof above.

\begin{remark}\label{rem:MainTheoremAutIsSynch}
Let $A_{q_0}=(\rd,\Xn,R_A,S_A,\pi_A, \lambda_A,q_0)$ be a minimal, initial, invertible finite transducer for an element of $\Rnr$, and with inverse homeomorphism represented by the minimal, initial finite transducer $B_{p_0^{-1}}=(\rd,\Xn,R_B,S_B,\pi_B, \lambda_B,p_0^{-1})$. Let $\mathscr{P} \subset S_{B}$ be such that for any word $\delta \in \Wn$  there is a (possibly trivial) rotation $\phi$ of $\delta$, a state $p^{-1} \in \mathscr{P}$ and word $\rho \in \CCn$ such that $(\rho)h_{p^{-1}} = \phi^\omega$. Let  $\mathscr{Q}\subset S_{A}$ be such that for all states $p^{-1} \in \mathscr{P}$ and for all $q \in \mathscr{Q}$ we have the initial transducer $B_{p^{-1}}A_{q}$ admits a complete prefix code $\vec{x}_{p^{-1},q}$ for $\Wne$ so that for all $\gamma \in \vec{x}_{p^{-1},q}$ there is $\delta\in \Xn^{\ast}$ so that for all $\Gamma\in \CCn$ we have $(\gamma \Gamma)B_{p^{-1}}A_{q} = \delta \Gamma$. Then, there is a $k \in \N$ such that for any word $\gamma \in \Wn$ of length $k$ the function $\pi_{A}(\gamma, \centerdot): \mathscr{Q} \to S_{A}$ takes only one value.
\end{remark}

We are now ready to complete the proof of Theorem \ref{thm:MainTheorem}.

\begin{cor}\label{cor:MainTheoremAutIsSynch}
Let $h\in \Rnr$ be such that $h=h_{q_0}$ for $$A_{q_0}= (\rd,\Xn,R_A,S_A,\pi_A, \lambda_A,q_0)$$ a minimal, initial, invertible finite transducer, where $\Gnr^h\leq \Gnr$. Then $A_{q_0}$ is strongly synchronizing.
\end{cor}
\begin{proof}
Let $h$ and $A_{q_0}$ be as in the statement, and suppose $h^{-1}$ is represented by some minimal, initial finite transducer $$B_{p_0^{-1}}=(\rd,\Xn,R_B,S_B,\pi_B, \lambda_B,p_0^{-1}).$$ Let $\mathscr{P}$ be the set of non-trivially accessible states of $B_{p_0^{-1}}$ and let $\mathscr{Q}:= S_{A}$. We observe that elements of $S_{A}$, by definition, are accessible from the initial state $q_0$ of $A$ by a non-empty word.

Fix a pair $(p^{-1},q) \in  \mathscr{P} \times \mathscr{Q}$. Let $\gamma, \mu, \nu \in \Wnr$ be such that  $\pi_{B}(\gamma, p_0^{-1}) = p^{-1}$, $\lambda_{B}(\gamma, p_0^{-1}) = \mu \ne \epsilon$ and $\pi_{A}(\nu, q_0) = q$. Let $g \in \Gnr$ be such that $g$ acts on the cone $U_{\mu}$ as a prefix replacement map, replacing the prefix $\mu$ with $\nu$. Let  $C_{s_0} = (\rd,\Xn,R_C,S_C,\pi_C, \lambda_C,s_0)$ be the minimal transducer representing $g$.  Observe that when the word $\gamma$ is read from the state $(p_0^{-1}, s_0, q_0)$ of the product $A\ast B \ast C$, then the resulting state is $(p^{-1}, s,q )$ where $s$ is the unique state of $C_{s_0}$ such that $h_{s}: \CCn \to \CCn$ is the identity map. Thus, we may identify the state  $(p^{-1}, s,q )$ of $A\ast B \ast C$ with the state  $(p^{-1},q )$ of $A\ast B $. Now since $h_{p_0^{-1}} g h_{q_0} \in \Gnr$ by definition, it follows that the map $h_{(p^{-1},q)} = h_{p^{-1}}h_{q}$ must act on $\CCn$ as a prefix exchange map, since it is a local action of $h_{p_0^{-1}} g h_{q_0}$. However, as $(p^{-1}, q)$ was an arbitrary state of $\mathscr{P} \times \mathscr{Q}$ it follows that for  any pair $((p')^{-1}, q') \in \mathscr{P} \times \mathscr{Q}$ the map $h_{(p')^{-1}}h_{q}$ acts on $\CCn$ as a prefix exchange map. Therefore by Remark~\ref{rem:MainTheoremAutIsSynch}, it follows that there is a number $k \in \N$ such that for any word $\gamma \in \Wn$ of length $k$ the map $\pi_{A}(\gamma, \centerdot): S_A \to S_{A}$ takes only one value. Now observe that there is an integer $j$ such that for any word $\zeta \in \Wnr$ of length $j$, $\pi_{A}(\zeta, \centerdot)$ is a map from  $S_A \sqcup R_A$ to $S_{A}$. Hence we conclude that $A$ is synchronizing at length $j+k$.
\end{proof}

\section{Combinatorial properties of strongly synchronizing transducers}

In this section we explore combinatorial properties of strongly synchronizing transducers. In the first subsection we study cycles in strongly synchronizing transducers, we also investigate conditions for when such transducers possess states that induce homeomorphisms. Some of the results of that subsection are used in our analysis of the outer automorphism groups appearing in Section~\ref{sec:naturalQuotient}. In the second subsection we introduce an algorithm for detecting when an automaton is strongly synchronizing.

\subsection{Prime words, circuits, and homeomorphism states.}
We call a word $w\in\Wn$ a \emph{prime word} if there is no prefix $z<w$ and $k\in\N$, with $k>1$, so that $z^k=w.$  For all $w\in\Wn$, there is a shortest prefix $z$ of $w$ so that there is a positive integer $k$ so that $w=z^k$, we call $z$ the \emph{prime root of $w$}, and note that the prime root of $w$ is in fact a prime word.

\begin{definition}
Let  $A_{p_0}= (\rd,\{0,1,\ldots ,n-1\},R_A,S_A,\pi_A, \lambda_A,p_0)$ be a strongly synchronizing finite transducer representing a continuous map $\CCnr\to\CCnr$.  We say a word $w\in\Wn$ \emph{represents a basic circuit of $A_{p_0}$ (at $p$)} if there is state $p$ so that $\pi_A(w,p)=p$ and if $w_1$ is a non-trivial proper prefix of $w$ then $\pi_A(w_1,p)\not =p$.
\end{definition}

The following lemma consists of various points which are all easy exercises.  We will prove points \ref{point_primeRootsOnCores}, \ref{point_allWordssynchronize}, and \ref{point_commonSegment}, leaving the other points to the reader.
\begin{lemma}\label{lem_circuitFacts}
Suppose that $$A_{p_0}= (\rd,\Xn,R_A,S_A,\pi_A, \lambda_A,p_0)$$ is a strongly synchronizing finite transducer, with all outputs on finite inputs finite.
 \begin{enumerate}
\item Given a core state $p$ and a word $w$ so that $\pi(w,p)=p$, then if $z$ is the prime root of $w$, then $\pi(z,p)=p$.\label{point_primeRootsOnCores}
\item Given any word $w\in\Wn$, there is a unique state $p$ so that $\pi_A(w,p)=p$. \label{point_allWordssynchronize}
\item If $w$ represents a basic circuit of $A_{p_0}$, then the unique state $p$ so that $\pi_A(w,p)=p$ has the property that $p$ is in $\Core(A)$, and furthermore, the word $w$ is prime.
\item Let $p\in \Core(A)$, and let $\mathcal{W}_p$ represent the set of all words representing basic circuits of $A$ (at $p$).  Then, $\mathcal{W}_p$ is finite.
\item The sets $\mathcal{W}_p$, for $p$ running over the states in $\Core(A)$, partition the set of all words representing basic circuits of $A$.
\item Given a core state $p$, given any letter $x\in\Xn$, there is a word in $\mathcal{W}_p$ beginning with the letter $x$.
\item For any state $q$ the set of minimal elements of the set
$\{\lambda_{q}(w)\mid w \in X^*\}\backslash\{\varepsilon\}$ (under the partial order on words in $\Wn$, where $w_1\leqslant w_2$ if $w_1$ is a prefix of $w_2$)
has no common initial prefix. \label{point_commonSegment}
 \end{enumerate}
\end{lemma}
\begin{proof}
Suppose $m\in\N, w\in\Wn$, and $A_{p_0}= (\rd,\Xn,R_A,S_A,\pi_A, \lambda_A,p_0)$ a strongly synchronizing finite transducer, with all outputs on finite inputs finite, and where $A_{p_0}$ synchronizes at level $m$ for some non-negative integer $m$.

First suppose $p$ is a state of $\Core(A)$ and $\pi(w,p)=p$.  Let $z$ be the prime root of $w$, and $k$ be a positive integer so that $z^k=w$. Consider the word $z^m$.  It is immediate that $z^m$ is at least $m$ in length, and so it synchronizes $A_{p_0}$ to some state. Meanwhile, $\pi(w^m,p)=p$ and $w^m$ has suffix $z^m$, so $z^m$ must synchronize $A_{p_0}$ to $p$.  But now, $z^{m+1}$ also synchronizes $A_{p_0}$ to $p$, since $z^{m+1}$ has $z^m$ as a suffix.  Therefore, $\pi(z,p)=p$.  This shows point \ref{point_primeRootsOnCores}.  Release the variable $p$ now.

Since $|w|>0$ it is immediately clear that there is a length $m$ suffix of $w^m$ so in particular $w^m$ synchronizes $A_{p_0}$ to some state $p\in\Core(A)$.  It is then the case that $\pi(w^m,p)=p$.  Note that the prime root $z$ of $w$ is also the prime root of $w^m$, so by point \ref{point_primeRootsOnCores} the prime root $z$ of $w$ has $\pi(z,p)=p$.  It is then immediate that $\pi(w,p)=\pi(z^k,p)=p$, which shows point \ref{point_allWordssynchronize}.

Note that point \ref{point_commonSegment} follows from the fact that our transducer is minimal.  If there were a non-trivial common prefix $w$ for the output words from a state $q$, then the minimalisation procedure would have the transitions to $q$ having all their outputs ending with $w$ as a suffix, while $w$ would be removed from the prefixes of all of the outputs of $q$.
\end{proof}

The following lemma is effectively a continuation of Lemma \ref{lem_circuitFacts}, now under the context that our transducer is bi-synchronizing (and hence, acts bijectively on its target Cantor space).

\begin{lemma} \label{lem:CondBijLocalMaps}
Let $1\leqslant r<n\in\N$ and $h\in\Homeo(\CCnr)$ be represented by a finite minimal bi-synchronizing transducer $A$ with set of states $Q$.  We have the following.

 \begin{enumerate}
%\noindent Since $h$ is a homeomorphism we have the following restrictions:
\item Every finite word $w$ is obtained as a prefix of some output from a state in $Q$.\label{allInitialSegments}
%(since otherwise the image of $h$ is not clopen).

\item For any periodic equivalence class of the tail equivalence $\sim$ with (minimal) period $w$ there exists
a unique minimal $t\in \bbN$ and a unique circuit in $A$ with output $w^{t}$.\label{allFinalCycles}

\item If $q\in Q$ so that the local map $h_q:\CCn\to\CCn$ is bijective, then for all $p\in Q$ such that
$\pi(l,p)=q$ for some $l\in \Xn$, we have $\lambda(l,p)\ne \varepsilon$. \label{noEmptyapproaches}

\item If $p\in Q$ so that the local map $h_p:\CCn\to\CCn$ is a homeomorphism, and so that its outputs $\{\lambda(l,p)\mid l\in\Xn\}$ are all different from the empty word, then each such output word has length one, and the induced map $\lambda(\cdot,p):\Xn\to\Xn$ is a permutation.\label{homeoStateNoEmptyOutsHasLengthOneOutputs}
\item If $p\in Q$ so that the local map $h_p:\CCn\to\CCn$ is a homeomorphism, and so that its outputs $\{\lambda(l,p)\mid l\in\Xn\}$ are all of length one, then for all $q$ so that $\pi(l,p)=q$ for some $l\in\Xn$, the local map $h_q$ is also a homeomorphism.\label{homeoLengthOneOutputsTargetsHomeos}
 \end{enumerate}
\end{lemma}
\begin{proof}
To see Point \ref{allInitialSegments}, simply observe that if some finite word $w$ is not obtained as a segment of output of some state in $Q$, then the image of $h$ cannot be clopen, since some points will be missing from the image of $h$ in each basic cone of $\CCnr$.

Point \ref{allFinalCycles} is essentially the dual point to point \ref{point_allWordssynchronize} of Lemma~7.9.  One can work in a minimal transducer $B$ representing $h^{-1}$, to detect the word $v$ of output produced by reading $w$ on the cycle that accepts $w$ in $B$.  Since $w$ is prime, the word $v$ is also prime (otherwise, if $v=r^t$ for some $t>1$, we have that $B$ has a cycle that accepts $r$, and the output of $r$).  Then, there is a unique cycle $C$ of $A$ labelled by $v$, and one can detect the value $t$ by seeing how many times $w$ is produced upon reading $v$ from the state that begins the cycle $C$.

Point \ref{noEmptyapproaches} follows from the fact that $q$ can produce a full copy of $\CCn$ as output.  Let $l$ be a letter in $\Xn$ so that $\pi(l,p)=q$, and $x\in\Xn\backslash\{l\}$.  Then, there is an output from $p$ obtained by first reading $x$, and then reading any infinite word.  This output can also be obtained from $p$ by first reading $l$.  In particular, $h$ would fail to be injective.

Point \ref{homeoStateNoEmptyOutsHasLengthOneOutputs} follows as we must be able, for any given word $w\in\Wn$, to find some word $z\in\Wn$ so that $\lambda(z,p)$ has $w$ as a prefix.  Therefore, we need to be able to write words with the $n$ distinct one letter prefixes provided by our alphabet $\Xn$, so, the $n$ outputs of $p$ must begin with the $n$ distinct letters in $\Xn$.  Furthermore, if for some $j\in\Xn$, the output $w_j=\lambda(j,p)$ has length greater than one and initial letter $k$, then from the state $p$ there will be at least $n-1$ outputs beginning with the letter $k$ that cannot occur as prefixes of any output written from the state $p$.  In particular, the length of all outputs from state $p$ on one-letter inputs is one, and these outputs form a function $\lambda(\cdot,p):\Xn\to\Xn$ which must be a permutation.

Finally, point \ref{homeoLengthOneOutputsTargetsHomeos} follows as each of the states in the image of the transition function $\pi(\cdot,p):\Xn\to Q$ is responsible for a full Cantor set of outputs. (The state $\pi(j,p)$ must produce all prefixes of all words in the Cantor space $\CCn$ so that $p$ itself will admit all outputs with prefix the letter $\lambda(j,p)$.)
\end{proof}

\begin{lemma} \label{lem:AllOutputsImpliesHomeoState}
Let $1\leqslant r<n\in\N$ and $h\in\Bnr$ be represented by a finite minimal bi-synchronizing transducer $A$ with set of states $Q$.  If for some state $q$ in $Q$, in the initial transducer $A_{q}$
any finite word in $\Wn$ can be obtained as an initial segment of some output,
then the local map $h_q$ is a self-homeomorphism of $\CCn$.\end{lemma}

\begin{proof}
By definition, $\Range(\lambda_{q})$ is a clopen set in $\CCn$ and by our assumption,
every basic cone in $\CCn$ intersects $\Range(\lambda_{q})$, thus $$\Range(\lambda_{q})=\CCn.$$
\end{proof}

\begin{definition}
Let $h\in\Rnr\sqcup\Rn$ be represented by a minimal transducer $A_{q_0}$.  We call any state $q$ of $A_{q_0}$, where the map $h_q$ is a self-homeomorphism of $\CCn$, a \emph{homeomorphism state}.
\end{definition}

Given $h\in\Bnr$ for some $1\leqslant r<n\in\N$, represented by a minimal transducer $A$, we have seen above (see Lemma \ref{lem_circuitFacts}(\ref{point_allWordssynchronize})) that $A$ has, for each letter $l\in\Xn$, a unique state $q_l$ so that $\pi(l,q_l)=q_l$.  We call the state $q_l$ the \emph{$l$-loop state}.  We define the set

$$\LoopStates(A)=\left\{ q\in Q\mid q \mbox{ is the }l\mbox{-loop state for some }l\in\Xn\right\}.$$

We now give some sufficient conditions for a loop state and its neighbours to
be homeomorphism states.

\vspace{.1 in}

\begin{lemma}
Let $1\leqslant r<n\in\N$, and $h\in\Bnr$ be represented by a minimal transducer $A$.  For each $l\in\Xn$ let $q_l\in \LoopStates(A)$ be the $l$-loop state. Suppose that for some $l\in\Xn$ we have that for all $j\in \Xn\backslash\{l\}$ the word $\lambda(j,q_l)$ is not trivial and has no non-trivial common prefix with $\lambda(l,q_l)$. Then, the local map $h_{q_l}$ is a homeomorphism $h_{q_l}:\mathfrak{C}_n\to\mathfrak{C}_n$, and in this case we have the following:
 \begin{enumerate}
\item $|\lambda(j,q_l)|=1$, for each $j\in\Xn$,
\item the map $\lambda(\cdot,q_l):\Xn\to\Xn$ is a permutation, and
\item for each state $p_j\seteq\pi(j,q_l),$ the state $p_j$ is a homeomorphism state.
 \end{enumerate} \label{loop-homeo}
\end{lemma}
\begin{proof}
Assume that $q_l$ is the loop state for some $l\in\Xn$, and that for all ${j\in \Xn\backslash\{l\}}$,  the word $\lambda(j,q_l)$ is not trivial and has no non-trivial common prefix with $\lambda(l,q_l)$.
For all $j\in\Xn$, set  $w_j\seteq\lambda(j,q_l)$. Observe that $\{q_l\}$ represents the only circuit with output $w_l$.  Further, as the outputs from the state $q_l$ on reading one-letter inputs are never trivial, given any letter $k\in\Xn$, there is some $j\in\Xn$ so that $\lambda(j,q_l)$ begins with the letter $k$.  (Otherwise, the homeomorphism $h$ would not be surjective. This is because it would be impossible, for large enough values of $t$, to produce any infinite sequence of letters representing a point in Cantor space with a contiguous substring of the form $\lambda(l,q_l)^t\concat k$.)

Therefore, we see that from the state $q_l$ we can produce all strings of length $1$ as prefixes of outputs.  However, the argument above is generic, since given any word $w\in \Wn$, the only way for $h$ to produce infinite sequences representing points of Cantor space, with contiguous substrings of the form $(\lambda(l,q_l))^t\concat w$ for arbitrarily large $t$, is to have an infinite string of input with a substring of the form $l^m$ for some $s$ near $t$ (that is, for values of $t$ much larger than the synchronizing length $m$), followed by some word $z$ so that $\lambda(z,q_l)$ has prefix $w$.

Therefore, $q_l$ is a homeomorphism state by Lemma \ref{lem:AllOutputsImpliesHomeoState}.

Furthermore, as $q_l$ is a homeomorphism state with all transitions with non-empty outputs, we have by Lemma \ref{lem:CondBijLocalMaps} Point \ref{homeoStateNoEmptyOutsHasLengthOneOutputs} that all the outputs on those transitions are length one and that the map $\lambda(\cdot,q_l):\Xn\to\Xn$ is a permutation.  We also have that for all states $q$ in the set $\left\{\pi(j,q_l)\mid j\in\Xn\right\}$ that $A_q$ induces a homeomorphism on $\CCn$ by Lemma \ref{lem:CondBijLocalMaps} Point \ref{homeoLengthOneOutputsTargetsHomeos}.
\end{proof}

\subsection{Detecting if an automaton is strongly synchronizing}\label{subsection: detecting synchronicity}
In this section we outline a combinatorial method for detecting when an arbitrary finite automaton is strongly synchronizing.

This method consists of iteratively applying the `collapsing procedure', which we define below, until there are no more `collapses' to be performed.

\begin{definition}[Collapsing procedure]\label{Construction:Collapsing procedure}
Let $A = \gen{X, Q_A, \pi_A}$ be a finite automaton. For each state $q \in Q_{A}$ let $[q]$ be the set of states $p \in Q_{A}$ such that the functions $\pi_{A}(\centerdot,p): X \to Q_{A}$ and $\pi_{A}(\centerdot, q): X \to Q_{A}$ are equal. Let $Q_{A_1} := \{[q] \mid q \in Q_{A}\}$ and observe that  $Q_{A_1}$ is a partition of $Q_{A}$.  Form a new automaton $A_1 = \gen{ X, Q_{A_1}, \pi_{A_1}}$ where, for $i \in X$ and $[q] \in Q_{A_1}$, we set $\pi_{A_1}(i, [q]) = [\pi_{A_1}(i, q)]$.
\end{definition}

\begin{remark}\label{Remark:no collapse implies that A and A1 are equal}
Let $A$ be an automaton and let $A_{1}$ be the automaton resulting from applying the collapsing procedure to $A$ as above. Observe that if a collapse is possible then $|A_{1}| < |A|$, otherwise  $|A_{1}| = |A|$ and $A_1 =_{si} A$. The strong isomorphism arises as  for a state $q$ of $A$, the state $[q]$ of $A_{1}$ is the set $\{q\}$.
\end{remark}

Let $A = \gen{X, Q_A, \pi_A}$ be an  automaton. Form a sequence $(A_i)_{i \in \N}$ of automata where $A_{0} = A$ and such that, for $j \in \N_1$, $A_j = \gen{X, Q_{A_j}, \pi_{A_j}}$ is the result of applying the collapsing procedure to the automaton $A_{j-1}$. The set of states $Q_{A_j}$ of $A_j$ is a  partition of $Q_{A_{j-1}}$ the set of states of $A_{j-1}$. Since $Q_{A_1}$ is a partition of $Q_{A_{0}}$, then,  by induction, the set $Q_{A_j}$ of states of $A_{j}$ corresponds to a partition $\T{P}(Q_{A_j})$ of $Q_A$. Thus for $j \in \N_{1}$, we identify the  states of $A_{j}$ with the elements of this partition so that  states of $Q_{A_j}$ correspond to subsets of $Q_A$. For $q \in Q_A$ and $j \in \N$ we fix the notation $[q]_j$ for the state of $Q_{A_j}$  containing $q$; if $j=0$, set $[q]_0 := q$. By definition of the collapsing procedure, for $j \in \N$, and distinct states $[p]_{j}$ and $[q]_{j}$ of $A_{j}$, $[p]_{j+1} = [q]_{j+1}$ if and only if the functions $\pi_{A_j}(\centerdot, [p]_{j})$  and $\pi_{A_j}(\centerdot, [q]_{j})$ are equal. In particular we have the following claim:

\begin{claim}
For $x \in X$ and $q \in Q_{A}$, $\pi_{A_{j}}(x,[q]_{j}) = [\pi_{A}(x, q)]_{j}$.
\end{claim}
\begin{proof}
We proceed by induction. By construction the claim holds for $A_{1}$. Let $k \in \N_{1}$ and assume that the claim holds for $A_{i}$ for all $1 \le i < k$. Let $x \in X$ and  $[q]_{k}$ be a state of $A_{k}$. Observe that $[q]_{k} = \{ p \in Q_{A}\mid \pi_{A_{k-1}}(y, [p]_{k-1})= \pi_{A_{k-1}}(y, [q]_{k-1}) \mbox{ for all } y \in X\}$. However by the inductive assumption we have: $[q]_{k} = \{p \in Q_{A} \mid [\pi_{A}(y, p)]_{k-1} = [\pi_{A}(y, q)]_{k-1} \}$. Thus for any $p \in [q]_{k}$ and any $y \in X$, we have $[\pi_{A}(y,q)]_{k} = [\pi_{A}(y,p)]_{k}$ since $[\pi_{A}(y, p)]_{k-1} = [\pi_{A}(y, q)]_{k-1}$.
\end{proof}

Whenever we have an automaton $A$ and a sequence  $(A_i)_{i \in \N}$ of automata with $A_0 = A$ and such that each subsequent term of the sequence is obtained from the previous one by applying the collapsing procedure, we shall identify, as above, the  set  $Q_{A_j}$ of states of $A_{j}$ with a partition of $Q_{A}$ and elements of $Q_{A_j}$ with subsets of $Q_{A}$. Further observe that if $i,j \in  \N$ and $i < j$, then $|A_i| \le |A_j|$. Thus the sequence $(A_i)_{i\in \N}$ is eventually constant.

We have the following result.

\begin{lemma}\label{Lemma:two states of A are indentified in Ai if and only if they process all words of length i identically}
   Let $A = \gen{X, Q_A, \pi_A}$ be an automaton and  $(A_i)_{i \in \N}$ be the sequence such that $A_0 = A$ and each subsequent term is the automaton resulting from applying the collapsing procedure to the previous one. Let $p, q \in Q_A$, then  $[p]_i = [q]_i$ in $A_{i}$ for some $i \in \N$  if and only if  for all words $\Gamma \in X^{i}$ $\pi_{A}(\Gamma, p) = \pi_{A}(\Gamma, q)$.
   \end{lemma}
   \begin{proof}
  We proceed by induction on $i$.
   Let $i= 0$ and  $p, q$ be states of $A$  such that $[p]_0 = [q]_0$. Since $A_0 = A$, we have $p =q$. Moreover, for any $s,t \in Q_A$ such that  $\pi_{A}(\epsilon,s) = \pi_{A}(\epsilon, t)$ then since $s = \pi_{A}(\epsilon,s)$ and $t = \pi_{A}(\epsilon,t)$, we conclude that $s =t$. This establishes the base case.

   Assume that for $k \in \N_{1}$ and for all $i < k$, $i \in N$, the statement of the lemma holds.

   Let $p, q \in Q_{A}$ be such that $[p]_{k} = [q]_{k}$. Let $\Gamma \in X^{k-1}$ and $ x \in X$ be arbitrary. Since $[p]_{k} = [q]_{k}$, by construction of the Collapsing procedure it follows that for all $y \in X$, $[\pi_{A}(y,p)]_{k-1} = [\pi_{A}(y, q)]_{k-1}$. Thus if $p':= \pi_{A}(x,p)$ and $q' := \pi_{A}(x,q)$, then we must have, $[p']_{k-1} = [q']_{k-1}$. Therefore by the inductive assumption we have, $\pi_{A}(\Gamma, p') = \pi_{A}(\Gamma, q')$, from this it follows that $\pi_{A}(x\Gamma, p) = \pi_{A}(x\Gamma, q)$. Since $x \in X$ and  $\Gamma \in X^{k}$ were arbitrary, we conclude that the functions $\pi_{A}(\centerdot, p): X^{k} \to X^{k}$ and   $\pi_{A}(\centerdot, q): X^{k} \to X^{k}$ are equal.

   Now let  $p,q \in Q_{A}$ be such that the functions $\pi_{A}(\centerdot, p): X^{k} \to X^{k}$ and $\pi_{A}(\centerdot,q): X^{k} \to X^{k}$ are equal. Let $x \in X^{k}$ be arbitrary, and $p' = \pi_{A}(x, p)$ and $q' = \pi_{A}(x, q)$. Observe that the functions $\pi_{A}(\centerdot, p'): X^{k-1} \to X^{k-1}$ and $\pi_{A}(\centerdot,q'): X^{k-1} \to X^{k-1}$ are equal since the functions $\pi_{A}(\centerdot, p): X^{k} \to X^{k}$ and $\pi_{A}(\centerdot,q): xX^{k-1} \to xX^{k-1}$ are equal. Thus by the inductive assumption for $k-1$ we have that $[p']_{k-1} =[q']_{k-1}$. Therefore as $x \in X$ was arbitrary, we have that for all $y \in X$, $\pi_{A_{k-1}}(y, [p]_{k-1})= \pi_{A_{k-1}}(y, [q]_{k-1})$, and so $[p]_k = [q]_{k}$ as required.

    \end{proof}

We may restate the lemma above as follows:

\begin{lemma}\label{Lemma:two states of A are not identical in Ai-1  if and only if they read  at least one word of length i-1 differently}
   Let $A = \gen{X, Q_A, \pi_A}$ be an automaton and  $(A_i)_{i \in \N}$ be the sequence such that $A_0 = A$ and each subsequent term is the automaton resulting from applying the collapsing procedure to the previous one. Let $p, q \in Q_A$. Then $[p]_i \ne [q]_i$ in $A_{i}$ for some $i \in \N$ if and only if there is a word $\delta \in X^{i}$ such that $\pi_{A}(\delta, p) \ne \pi_{A}(\delta, q)$.
\end{lemma}

We thus have the following theorem characterising when a finite automaton is strongly synchronizing.

  \begin{theorem}\label{Theorem:A is synch if and only if it can be reduced to the single state transducer, sycn level is min number of applications of collapsing procedure needed to accomplish this.}
   Let $A = \langle X, Q_A, \pi_{A} \rangle$ be an automaton. Form the sequence $(A_i)_{i \in \N}$ where $A_0 = A$ and each subsequent term of the sequence is the result of applying the collapsing procedure to the preceding term. Let $k$ be minimal such that $|A_{k}| = |A_{k+1}|$. Then $A$ is strongly synchronizing if and only if $A_{k}$ consists only of a single state. Moreover, if $k$ is minimal such that $|A_{k}|=1$, then $k$ is the minimal synchronizing level of $A$.
   \end{theorem}
   \begin{proof}
   In what follows, for a given non-negative integer $k$, we use $\N_{k}$ to denote the set of all naturals greater than or equal to $k$.

   Observe that for all $l \in \N_{k}$ we have $A_{l} = A_{k}$ by Remark~\ref{Remark:no collapse implies that A and A1 are equal}. Thus if $|A_{k}| \ne 1$, then $|A_{l}| \ne 1$ for any $l \in \N_{k}$. Therefore for any $j \in \N_{k-1}$, there is a pair of states $q, p  \in Q_{A}$ such that $[q]_{j+1} \ne [p]_{j+1}$ and so, by Lemma~\ref{Lemma:two states of A are not identical in Ai-1  if and only if they read  at least one word of length i-1 differently}, for any $j \in \N_{k-1}$ there is a word $\delta \in X^{j+1}$ and states $p,q \in Q_A$ such that $\pi_{A}(\delta, q) \ne \pi_{j}(\delta, p)$. From this we conclude that $A$ is not strongly synchronizing.

   Now suppose that $|A_{k}| = 1$ then by Lemma~\ref{Lemma:two states of A are indentified in Ai if and only if they process all words of length i identically} we have that for any pair $p,q \in Q_{A}$ and any word $\Gamma \in X^{k}$, $\pi_{A}(\Gamma, p) = \pi_{A}(\Gamma, q)$ and so  $A$ is strongly synchronizing. Moreover since $k$ is minimal such that $|A_{k}| =1$, Lemma~\ref{Lemma:two states of A are not identical in Ai-1  if and only if they read  at least one word of length i-1 differently} guarantees that it is the minimal synchronizing level of $A$.
   \end{proof}

We conclude with some examples and non-examples.

\begin{example}
Consider the following automaton, $A$, below:
\begin{center}
\begin{tikzpicture}[shorten >=0.5pt,node distance=3cm,on grid,auto,semithick,every state/.style={fill=red,draw=none,circular drop shadow,text=white}]
   \node[state] (q_0)   {$q_0$};
   \node[state] (q_1) [below left=of q_0] {$q_1$};
   \node[state] (q_2) [below right=of q_0] {$q_2$};
    \path[->]
    (q_0) edge [loop above] node {$0|0$} ()
          edge  node [swap] {$2|0$} (q_2)
          edge  node {$1|2$} (q_1)
    (q_1) edge[bend left]  node  {$0|0$} (q_0)
          edge [loop left=of q_1] node {$1|2$} ()
          edge  node {$2|1$} (q_2)
    (q_2) edge[bend right]  node [swap] {$1|1$} (q_0)
          edge[bend left]  node {$0|2$} (q_1)
          edge [loop right=of q_2] node {$2|0$} ();
\end{tikzpicture}
\end{center}
Then $\bar{A}$ is as follows:
\begin{center}
\begin{tikzpicture}[shorten >=0.5pt,node distance=3cm,on grid,auto,semithick, every state/.style={fill=red,draw=none,circular drop shadow,text=white}]
   \node[state] (q_0)   {$q_0$};
   \node[state] (q_1) [below left=of q_0] {$q_1$};
   \node[state] (q_2) [below right=of q_0] {$q_2$};
    \path[->]
    (q_0) edge [loop above] node {$0$} ()
          edge  node [swap] {$2$} (q_2)
          edge  node {$1$} (q_1)
    (q_1) edge[bend left]  node  {$0$} (q_0)
          edge [loop left=of q_1] node {$1$} ()
          edge  node {$2$} (q_2)
    (q_2) edge[bend right]  node [swap] {$1$} (q_0)
          edge[bend left]  node {$0$} (q_1)
          edge [loop right=of q_2] node {$2$} ();
\end{tikzpicture}
\end{center}
After one iterate of the collapsing procedure we find that the states $q_0$ and $q_1$ are in the same equivalence class and $q_2$ is in a class on its own. The quotient automaton is shown below:
\begin{center}
\begin{tikzpicture}[shorten >=0.5pt,node distance=3cm,on grid,auto,semithick,every state/.style={fill=red,draw=none,circular drop shadow,text=white}]
   \node[state] (q_0)   {$q_0$};
   \node[state] (q_2) [below right=of q_0] {$q_2$};
    \path[->]
    (q_0) edge [loop above] node {$0$} ()
          edge  node [swap] {$2$} (q_2)
          edge [loop left]  node {$1$} ()
    (q_2) edge[bend right]  node [swap] {$1$} (q_0)
          edge[bend left]  node {$0$} (q_0)
          edge [loop right=of q_2] node {$2$} ();
\end{tikzpicture}
\end{center}
After the second iterate of the procedure the quotient automaton is as follows:
\begin{center}
\begin{tikzpicture}[shorten >=0.5pt,node distance=3cm,on grid,auto,every state/.style={fill=red,draw=none,circular drop shadow,text=white}]
   \node[state] (q_0)   {$q_0$};
    \path[->]
    (q_0) edge [loop above] node {$0$} ()
          edge [loop right] node {$2$} ()
          edge [loop left]  node {$1$} ();
\end{tikzpicture}
\end{center}
We can likewise perform, the same process on $\bar{A}^{-1}$, which  is given below:
\begin{center}
\begin{tikzpicture}[shorten >=0.5pt,node distance=3cm,on grid,auto, semithick,every state/.style={fill=red,draw=none,circular drop shadow,text=white}]
   \node[state] (q_0)   {$q_0$};
   \node[state] (q_1) [below left=of q_0] {$q_1$};
   \node[state] (q_2) [below right=of q_0] {$q_2$};
    \path[->]
    (q_0) edge [loop above] node {$1$} ()
          edge  node [swap] {$0$} (q_2)
          edge  node {$2$} (q_1)
    (q_1) edge[bend left]  node  {$0$} (q_0)
          edge [loop left=of q_1] node {$2$} ()
          edge  node {$1$} (q_2)
    (q_2) edge[bend right]  node [swap] {$1$} (q_0)
          edge[bend left]  node {$2$} (q_1)
          edge [loop right=of q_2] node {$0$} ();
\end{tikzpicture}
\end{center}
This automaton can also be collapsed to a single state automaton in 2 steps. One can check that $A$ is bi-synchronizing on the second level with map give by:
\begin{equation*}
f:\left\{ \quad \begin{matrix}
            00 \mapsto q_0 & 10 \mapsto q_0 & 20 \mapsto q_1 \\
            01 \mapsto q_1 & 11 \mapsto q_1 & 21 \mapsto q_0 \\
            02 \mapsto q_2 & 12 \mapsto q_2 & 22 \mapsto q_2
          \end{matrix} \right.
\end{equation*}

We now illustrate a non-example. Let $B$ be the following transducer:
\begin{center}
\begin{tikzpicture}[shorten >=0.5pt,node distance=3cm,on grid,auto,semithick, every state/.style={fill=red,draw=none,circular drop shadow,text=white}]
   \node[state] (q_0)   {$q_0$};
   \node[state] (q_1) [below left=of q_0] {$q_1$};
   \node[state] (q_2) [below right=of q_0] {$q_2$};
    \path[->]
    (q_0) edge [loop above] node {$0|0$} ()
          edge  node [swap] {$1|2$} (q_2)
          edge  node {$2|1$} (q_1)
    (q_1) edge[bend left]  node  {$1|0$} (q_0)
          edge [loop left=of q_1] node {$2|2$} ()
          edge  node {$0|1$} (q_2)
    (q_2) edge[bend right]  node [swap] {$2|2$} (q_0)
          edge[bend left]  node {$0|0$} (q_1)
          edge [loop right=of q_2] node {$1|1$} ();
\end{tikzpicture}
\end{center}
This transducer is not synchronizing at any level, for instance there is no single state into which the automaton enters after reading any finite string of zeroes. Consider $\bar{B}$ below:
\begin{center}
\begin{tikzpicture}[shorten >=0.5pt,node distance=3cm,on grid,auto,semithick,every state/.style={fill=red,draw=none,circular drop shadow,text=white}]
   \node[state] (q_0)   {$q_0$};
   \node[state] (q_1) [below left=of q_0] {$q_1$};
   \node[state] (q_2) [below right=of q_0] {$q_2$};
    \path[->]
    (q_0) edge [loop above] node {$0$} ()
          edge  node [swap] {$1$} (q_2)
          edge  node {$2$} (q_1)
    (q_1) edge[bend left]  node  {$1$} (q_0)
          edge [loop left=of q_1] node {$2$} ()
          edge  node {$0$} (q_2)
    (q_2) edge[bend right]  node [swap] {$2$} (q_0)
          edge[bend left]  node {$0$} (q_1)
          edge [loop right=of q_2] node {$1$} ();
\end{tikzpicture}
\end{center}
No two states of this automaton are related under the equivalence class given in the procedure, and so the quotient automaton returned after the first iterate of the collapsing procedure is exactly the  same as $B$. Therefore the procedure returns that the automaton is not strongly synchronizing.
\end{example}

  \section{The natural quotient $\Bnr\twoheadrightarrow \Onr$\label{sec:naturalQuotient}}

 We begin by defining, for a given $n$, a set $\SOn$ of non-initial minimal strongly synchronizing transducers.  We
show that $\SOn$ is a monoid and contains a group
$\On$ which has as subgroups the groups $\Onr$, for each $1\leq r<n$.  We prove Proposition \ref{prop: nesting} giving a nesting criteria for the subgroups
$\Onr$ inside $\On$. We then give
an example of an element of $\mathcal{O}_{4,3}$ which is not an
element of $\mathcal{O}_{4,r}$ for any $r \in \{1,2\}$, thus showing that $\Onr$ depends on $r$ (although, this remains open at the level of isomorphism). To close the section, we give a direct, combinatorial method of inverting elements in $\On$ which does not rely on passing to a completion in $\Bnr$, inverting, and then taking the core.

\subsection{Classes and sets} For what follows, we establish two conventions.  The first convention faces a technical problem that must be dealt with, whilst the second simply supports a nicety.

  The first convention we establish is that we will regularly confound a (finite) transducer with its equivalence class under strong isomorphism.  That is, if we refer to a transducer, we are really referring to the class of all transducers which are strongly isomorphic to it.  Similarly, if we refer to a set of transducers, then we are actually referring to a set of equivalence classes of transducers, where all of the transducers in a given equivalence class are strongly isomorphic.  If we refer to the states of a transducer, the reader will be assumed to have chosen a representative of the appropriate equivalence class of transducers (under strong isomorphism) for the discussion that follows.  We employ this language to avoid having to explicitly choose representative transducers and having to constantly discuss independence of choice of transducers at each turn.  On occasion, however, we will still use the congruence notation ``$=_{si}$'' f
  or a strong isomorphism between any two given particular transducers when we think it will aid the reader in catching some pointed observation.

  Thus, when we define below a monoid or group of transducers, the elements of that monoid or group will actually be the equivalence classes of the transducers specified.

The second convention is not strictly necessary, but it allows us to use equivalence classes of transducers (under strong isomorphism) where each equivalence class is actually a set.  To achieve this, we place the sets of states of our transducers on a firmer foundation.

Specifically, recall that the set $V_\omega$ from the von Neumann hierarchy of sets has the property that all of its elements are finite sets. Furthermore, as in Chapter 1 of Kunen's text \cite{kunen2011set}, by interpreting (in the natural cases) the elements of a direct product of two sets $X$ and $Y$ as sets of the form $\{x,\{x,y\}\}$, where $x\in X$ and $y\in Y$, one sees that $V_\omega$ has the following two properties:
  \begin{enumerate}
      \item for all $Q\in \mathscr{P}_{fin}(V_{\omega})$, we have $Q\in V_\omega$, and
      \item for all $X, Y\in V_\omega$, we have $X\times Y\in V_\omega$.
  \end{enumerate}
 Here, note that we are using $\mathscr{P}_{fin}(V_\omega)$ to represent the set of all finite subsets of $V_\omega$. The two properties above are exactly what is required in our discussions for working combinatorially with finite transducers (e.g, we might sometimes pass to subsets of the set of states, for instance when passing to the core of a synchronizing transducer, and secondly, the set of states of a product transducer is the direct product of the sets of states). From this point forward, when we discuss a transducer, we are implicitly assuming the set of states of $T$ is an element of $V_\omega$, and the alphabet of $T$ will be one of the sets $X_n$, or in the case of a transducer representing an element of $\Gnr$, the alphabet will be the union of the sets $\rd$ and $X_n$ for some $r$ and $n$ natural with $0<r<n$.  Thus, the reader may assume if they choose that all of our transducers come from a specific set of transducers (as opposed, e.g., to a class of transducers).  Note that in the discussions below we will treat direct products simply as sets of ordered pairs of elements, as in standard usage.

We can now proceed with the main discussion of the section.

  \subsection{Around a monoid $\SOn$}

We first define our monoid of interest.

\begin{definition}\label{def: semigroup On}
Let $\SOn$ be the set consisting of those transducers $T = \langle\Xn, Q_T, \pi_T, \lambda_T\rangle$ satisfying the following conditions:
\begin{enumerate}
\item $T$ is synchronizing at level $k$ for some natural $k$, \label{def: semigroup On cond1}
\item \label{def: semigroup On cond2}  $\core{T}=T$,
\item \label{def: semigroup Nn cond} $T$ is minimal, and
\item \label{def: semigroup On cond3}  for each state $q$ of $T$, the induced map $\lambda(\cdot,q):\CCn\to\CCn$ is injective and has clopen image.
\end{enumerate}
\end{definition}

Note that in the above context, for ease of notation we sometimes confound a state $q$ of $T$ with the corresponding map $\lambda(\cdot,q):\CCn\to\CCn$ (as we consider states to represent local maps).  This for instance leads to the appearance of notation such as ``$\Image(q)$'' which seems clear enough and which is lighter to read than the full ``$\Image(\lambda(\cdot,q))$.''

\newcommand{\coreit}[1]{\left(#1\right)^\circ}
The following definition simply generalises the multiplication of Definition~\ref{lem:OnrProd} to our broader context (of working with the transducers in $\SOn$).

\begin{definition}\label{lem:SOnProd}
Let $T_1, T_2\in\SOn$.  We define the product $T_1T_2\in\SOn$ as follows:
\begin{enumerate}
\item Compute the full transducer product $T_1\ast T_2$.
\item \label{def:remIncResp} Pick any state $q$ of $T_1 \ast T_2$. Apply the ``Remove States of Incomplete Response'' algorithm in the paper \cite{GNSenglish} to the initial transducer $(T_1\ast T_2)_{q}$ to produce an initial transducer $\overline{(T_1T_2)}_{{q}'}$ and perform the following operations:
\item \label{def:corePass} Pass to the core of $\overline{(T_1T_2)}_{{q}'}$ to produce the result $\overline{(T_1T_2)}^\circ$.
\item \label{def:idEquivStates} Identify equivalent states of $\overline{(T_1T_2)}^\circ$ to produce the result in $\SOn$, which we denote by $T_1T_2$.
\end{enumerate}
\end{definition}

The following lemma discusses aspects of the product above, including that it is well defined.  Its proof is straightforward but mildy tedious: we give a brief discussion of the key ideas.
\begin{lemma}\label{lem:nonCommutingCoreOps}
Let $T_1, T_2\in\SOn$. We have:
\begin{enumerate}
\item For a connected strongly synchronizing transducer, the process of removing states of incomplete response does not change the set of states that are in the core;
\item For a connected strongly synchronizing transducer without states of incomplete response, the processes of identifying equivalent states and of passing to the core commute;
\item \label{pt:nonCommuting}For a connected strongly synchronizing transducer, the processes of identifying equivalent states and of removing states of incomplete response generally do not commute;
\item The product transducer $T_1T_2$ defined in Definition \ref{lem:SOnProd} is well defined, and does not depend on the state $q$ chosen in step \ref{def:remIncResp} of that definition.
\end{enumerate}

\end{lemma}

{\flushleft{\it Discussion:}}\\
To see this lemma one should recall that the ``Remove states of incomplete response'' process simply adds a new initial state $\tilde{q}$ (which will not survive passage to the core as it has no incoming edges), and then it modifies the output words of transitions (without modifying destinations). Therefore, the transformation of a core state of incomplete response into a core state of complete response depends only on the other states within the core.  In particular, the multiplication process is independent of the choice of initial state used in part \ref{def:remIncResp} of Definition \ref{lem:SOnProd}.

Further, we observe the basic fact from topology that the composition of two injective continuous maps from $\CCn$ to itself, where each of these maps has image a clopen set, will be another injective map with image a clopen set.  Thus, each state in the product (before removing incomplete response) gives an injective map with clopen image.  We then observe that removing incomplete response from such a state just removes the greatest common prefix from the set of infinite words that are output from that state, and this does not impact injectivity or the fact that the image of that state is clopen.

In Figure \ref{fig-nonCommuteOps} below, we provide an example transducer with two core states where the processes of removing states of incomplete response, and identifying equivalent states, do not commute when applied to the core (proving Point \ref{pt:nonCommuting}).

\begin{figure}[htbp]
\begin{center}
 \begin{tikzpicture}[->,>=stealth',shorten >=1pt,auto,node distance=2.3cm,on grid,semithick,every state/.style={fill=red,draw=none,circular drop shadow,text=white}]
   \node[state] (a)   {$q_0$};
   \node[state,yshift=-.3cm] (b) [below left=of a] {$b$};
   \node[state,yshift=-.3cm] (c) [below right=of a] {$c$};
    \path[->]
    (a) edge  [out=240,in=75]node [swap]{$0|\epsilon$} (b)
        edge  [out=300,in=105]node {$1|\epsilon$} (c)
    (b) edge [out=165,in=195, loop] node [swap] {$0|0$} ()
        edge[out=15,in=165]  node  {$1|0$} (c)
    (c) edge[out=195,in=345]  node {$0|1$} (b)
        edge [out=15,in=345,loop]  node {$1|1$} ();
\end{tikzpicture}
\end{center}
\caption{A bi-synchronizing transducer representing the identity.\label{fig-nonCommuteOps}}
\end{figure}

$\square$

Note that $\SOn$ contains the sets $\Onr$ for each $1\leq r<n$, and the multiplication of $\Onr$ agrees with the multiplication of $\SOn$ over any of those sets.  In fact we have the following result.

\begin{prop}
The set $\SOn$  with the multiplication defined above is a monoid.
\end{prop}
\begin{proof}
This is a straightforward but somewhat tedious exercise in the definitions.  The central point for checking associativity is the independence of the set of core states that arises when one removes incomplete response from different choices of states starting that process.
\end{proof}

We have the following theorem.  Recall for what follows that $\Snr$ is the submonoid of $\Rnr$ corresponding to homeomorphisms representable by strongly synchronizing transducers (note $\Snr$ is not closed under inversion).

\begin{theorem}\label{thm:CoreExtensibility}
Let $T = ( \Xn, Q_T, \pi_{T}, \lambda_{T})$ be a transducer.  Then  $T$ is an element of $\SOn$ if and only if
there is a natural $r$ with $1\leq r<n$ and an initial transducer $$A_{q_0}=(\rd,\Xn, R, S, \pi_{A}, \lambda_{A},q_0)$$ (with $Q_T\subset S$) representing an element $h_{q_0}\in\Snr$ with $\core{A_{q_0}}=T$.\end{theorem}

\begin{proof}
Let $A_{q_0}$  be a minimal transducer representing an element of $\Snr$, so in particular $\core{A}$ is
synchronizing. Conditions (1) and (2) follow immediately from the definition of $\Snr$.  As $A_{q_0}$ represents a homeomorphism of Cantor space then (3) follows as well as each state produces a local map which must be injective and which can be expressed as an open projection of a restriction of a homeomorphism to clopen set, and thus has clopen image.

For the reverse implication, let $T \in  \SOn$. We now describe how to construct a strongly synchronizing transducer representing a homeomorphism of $\CCnr$, for appropriate $n$ and $r$, with core equal to $T$.

Let $q$ be a state of $T$.  Denote by
$\Image(q)$ the set $\{\lambda_{T}(x, q) \mid x \in \CCn\}$ (note $\Image(q)$ corresponds to the image of a local map $h_q$ in the case that $T$  were an initial transducer admitting a state $q$).
By assumption $\Image(q)$ is clopen, and so can be written
as a union of finitely many disjoint basic open sets $\bigcup_{1 \le i \le
k} U_{\eta_i}$ where $\eta_i \in \Xn^{*}$. Let $M =  \max_{1
\le i \le k}\{|\eta_i|\}$, and let $\nu \in  \Xn^{*}$ be a prefix of some element of  $\Image(q)$ such
that $|\nu| \ge M$. Let $N \in \N$ be such that for all words $\phi \in \Xn^N$ we have $|\lambda_{T}(\phi,q)| \ge M$. Let
$\phi_1, \ldots, \phi_j$ be all those elements of $\Xn^N$ for
which $\nu$ is a prefix of $\lambda_{T}(\phi_i, q)$ for all $1 \le i \le
j$. Now set, for all $1\leq i\leq j$, $\lambda_{T}(\phi_i, q) = \nu \rho_i$ for some $\rho_i
\in \Xn^{\ast}$. Let $\pi_{T}(\phi_i, q) =
p_i$ for $1 \le i \le j$.
Now since  $\Image(q)$  is
clopen, and since $\nu \ge \eta_l$ for some $1 \le l \le k$,
then we must have that $\bigcup_{1 \le i \le j} \rho_i \concat
\Image(p_i) = \CCn$ (here $\rho_i \concat \Image(p_i) = \{\rho_i x \mid x \in \Image(p_i)\}$). Note that since $q$ is injective  the $\rho_i \concat \Image(p_i)$ are pairwise disjoint.

Now set $m=(n-1)j$ and $r=n-1$.  Observe that there are maximal finite anti-chains
of $\Wnr$ of size $m$ and $n-1$ respectively. Let $\{ \xi_1,
\ldots, \xi_{m} \}$ and $\{ \zeta_1, \ldots, \zeta_{n-1} \}$ be
two such antichains. Since $m$ is divisible by $j$, we may
partition the first antichain into blocks of length $j$ as follows:  $\{ \xi_1,
\ldots, \xi_{m} \} = \sqcup_{0 \le i <
n-1} \{ \xi_{ij +1}, \ldots, \xi_{(i+1)j}\}$.
Then we may define a map $\alpha_{T} : \CCnr \to \CCnr$ by
$\xi_{ij + l} x \mapsto \zeta_i \rho_l \lambda_{T}(x, p_l)$ for $0 < l \le j$ and $0 \le i < n-1$. It follows that  $\alpha_{T}$ is a homeomorphism since $\bigcup_{1 \le i \le j} \rho_i \concat
\Image(p_i) = \CCn$ where, as observed above, this is a union of disjoint sets. We may now apply the procedure laid out in Section \ref{omega} to compute a minimal initial transducer in $\Rnr$ representing $\alpha_{T}$ where we observe that the resulting transducer is strongly synchronizing as processing any word $\xi_{t}$ will result in processing from a state in $T$ for all extensions.

\end{proof}

\begin{remark}\label{rem:productEquivalence}
Let $A_{p_0}$, $B_{q_0}$, $C_{p_0q_0}$ be minimal initial strongly synchronizing transducers respectively representing elements $h_{p_0}$, $h_{q_0}$ and $h_{p_0}h_{q_0}$ of $\Snr$, respectively, for some given $n$ and $r$. Since the core of the product is contained in the product of the cores, we see that the well-definedness of the product in $\SOn$ gives  \[\core{C_{(p_0,q_0)}}=\core{\core{A_{p_0}}\core{B_{q_0}}}.\]
\end{remark}

We can now specify the subgroup $\Onr$ of the monoid $\SOn$ as those elements of $\SOn$ which correspond to minimal cores of elements of $\Bnr$.  Specifically we obtain the following theorem.
\begin{theorem}
\label{thm:out classification}For $1<r<n\in\N$, we have $\Onr\cong\Outnr$.
\end{theorem}

\begin{proof}

  This follows from Remark \ref{rem:productEquivalence}, and from the fact that $\Gnr$ is  precisely the subgroup of $\Bnr$ consisting of the homeomorphisms that are representable by transducers with representative cores which act as the identity.

 \end{proof}

 Inversion of elements in the subgroup $\Onr$ currently depends on passing up to an element of $\Bnr$, inverting a homeomorphism, and passing back through the quotient to $\Onr$.  This is unsatisfactory: we will give a combinatorial process for inversion in Subsection \ref{sec:inversion} which depends only on the element of $\Onr$ as a transducer and produces the correct inverse.

We are now in position to define $\On$ as the subgroup of $\SOn$ generated by the union of the subgroups $\Onr$.

\subsection{Viable combinations, transducer completions, and nesting conditions on the groups $\Onr$}\label{sec:nesting}
In order to explore the relationships amongst the subgroups $\Onr$ of $\On$, we will need to provide some combinatorial constructions which will help us to understand when an element of $\Onk{n}{i}$ is also an element of $\Onk{n}{j}$, for $i\neq j$.

Figure \ref{Fig: example showing dependence on r} depicts an element $T$ of $\mathcal{L}_4\subset\Onk{4}{3}$ which has two very interesting properties.  Firstly, and not directly relevent to the discussion of this section, the reader may verify that $T$ has no state which acts as a homeomorphism state, even though it is synchronous and it is the core of an element in $\Bnk{4}{3}$ (discussed in Subsection \ref{ssec:dependence on r}).  Secondly, and motivating the discussion in this subsection, $T$ has no completion to a transducer representing an element of $\Bnk{4}{1}$ or of $\Bnk{4}{2}$, so $T\not\in (\Onk{4}{1}\cup\Onk{4}{2})$ even though $T\in \Onk{4}{3}$.

\begin{figure}[htbp]

 \begin{center}
 \begin{tikzpicture}[->,>=stealth',shorten >=1pt,auto,node distance=5cm,on grid,semithick,
                     every state/.style={fill=red,draw=none,circular drop shadow,text=white}]
    \node[state] (q_0) [xshift=0cm, yshift=0cm]  {$q_1$};
    \node[state] (q_1) [xshift=3cm, yshift=-2cm] {$q_4$};
    \node[state] (q_2) [xshift=3cm, yshift=-4cm] {$q_5$};
    \node[state] (q_3) [xshift=0cm, yshift=-5cm] {$q_3$};
    \node[state] (q_4) [xshift=6cm, yshift=-3cm] {$q_2$};
     \path[->]
     (q_0) edge [in=105,out=75,loop] node [swap]{$0|0$} ()
           edge  node {$1|0,3|1$} (q_1)
           edge  [out=225,in=135]node [swap]{$2|1$} (q_3)
     (q_1) edge[out=260,in=100]  node  [swap]{$1|2$} (q_2)
           edge [in=50,out=80, loop] node  {$3|3$} ()
           edge [out=10,in=80] node {$0|2$} (q_4)
           edge [out=170, in = 100] node [swap] {$2|3$} (q_3)
     (q_2) edge[out=90,in=270]  node [swap]{$3|1$} (q_1)
           edge node [swap]{$2|1$} (q_3)
           edge [in=330,out=300, loop] node [swap] {$1|2$} ()
           edge node [swap] {$0|2$} (q_4)
      (q_3) edge[out=125,in=235]  node [swap]{$0|0$} (q_0)
            edge [out=90,in=180]  node [swap]{$3|0$,$1|1$} (q_1)
            edge [in=285,out=255, loop] node [swap] {$2|1$} ()
      (q_4) edge[out=90,in=0]  node {$1|0,3|3$} (q_1)
                 edge [out=45,in=45]  node [swap]{$0|0$} (q_0)
                 edge [in=330,out=300] node {$2|3$} (q_3);
 \end{tikzpicture}
 \end{center}
 \caption{An element $T$ of $\mathcal{L}_{4}$ without homeomorphism states}
  \label{Fig: example showing dependence on r}
 \end{figure}

In Theorem $\ref{thm:CoreExtensibility}$ we build a homeomorphism $\alpha_{T}$ depending on some initial choice of state $q$ in a strongly synchronizing transducer $T$ of $\SOn$.  In general, it turns out that starting from an initial choice of state $q$ as above is too limited for producing the set of all possible completions to homeomorphisms.  The process below weakens our dependencies and enables us to determine the nesting structure for the subgroups $\Onr$.

\begin{definition}[Viable combinations]
For a transducer $T \in \SOn$ we call a tuple $\mathfrak{v}_{j}:=((\rho_1, \ldots, \rho_j), (p_1, \ldots, p_j))$ such that:
\begin{enumerate}
\item  $(\rho_1, \ldots, \rho_j) \in (\Xn^{\ast})^j$ and $(p_1, \ldots, p_j) \in Q_T^{j}$.
\item  $\rho_i \concat \Image(p_i) \cap \rho_j \concat \Image(p_j) = \emptyset$ for $i \ne j$.
\item $\bigsqcup_{1 \le i \le j} \rho_i \concat \Image(p_i) = \CCn$ \label{def: fill a con}
\end{enumerate}
a (full) \emph{viable combination for $T$}.

If condition (3) is replaced by the weaker condition that $$\bigsqcup_{1 \le i \le j} \rho_i \concat \Image(p_i)$$ is just a \emph{clopen} subset of $\CCn$, then we say $\mathfrak{v}_j$ a \emph{partial viable combination}.
\end{definition}

\begin{remark}\label{rem:viableToHomeo}
Notice that given a $T\in\widetilde{\mathcal{O}}_n$ and a viable combination $\mathfrak{v}_j$ for $T$ there is an induced a homeomorphism $h_{\mathfrak{v}_j}$ from $\mathfrak{C}_{n,j} \to \mathfrak{C}_{n}$ by $\dot{a} x \mapsto \rho_{a} \lambda_{T}(x, p_a)$ for all $\dot{a} \in \{\dot{1}, \ldots, \dot{j}\}$ and arbitrary $x \in \CCn$ (the index $a$ is the value of $\dot{a}$ with the ``dot'' removed).  {\bf Here, and throughout this section, we allow $j\geq n$.}
\end{remark}

\begin{definition}[Single expansions of viable combinations]
Let $T = \langle\Xn, Q, \pi, \lambda\rangle \in \SOn$ and let
$\mathfrak{v}_{j} = ((\rho_1, \ldots, \rho_j), (p_1, \ldots,
p_j))$ be  a viable combination for $T$. Fix an $i$ such that $1
\le i \le j$, and let $\rho_{i,l} := \rho_i\concat\lambda(l,
p_i)$ and $p_{i,l} := \pi(l, p_i)$ for $l \in \Xn$, then
\begin{eqnarray*}
\mathfrak{v}_{j+n-1} := (&(&\rho_1, \ldots, \rho_{i-1},
\rho_{i,0}, \ldots,\rho_{i,(n-1)}, \rho_{i+1}, \ldots, \rho_{j}
), \nonumber \\ &(&p_1, \ldots, p_{i-1}, p_{i,0}, \ldots, p_{i, (n-1)},
p_{i+1},\ldots, p_{j}))
\end{eqnarray*}
is called a \emph{single expansion of
$\mathfrak{v}_{j}$}.
\end{definition}

\begin{remark}
It follows from the fact that a viable combination induces a homeomorphism that a single expansion of a viable combination for an element $T \in \SOn$ results in a new viable combination for $T$. Therefore a sequence of single expansions applied to a viable combination for $T$ also results in a new viable combination for $T$.
\end{remark}
\begin{comment}
\begin{definition}[Expansions of viable combinations]\label{Def: expansions of viable combinations.}
Let $T \in \SOn$ and let $\mathfrak{v}_{j}$ be a viable combination for $T$, and $d$ a natural number.  Then we call any viable combination $\mathfrak
{v}_{j + d (n-1)}$ which is the result of a sequence of $d$ single expansions applied to $\mathfrak{v}_{j}$ an \emph{expansion of $\mathfrak{v}_{j}$}.
\end{definition}

\begin{definition}
For a transducer $T \in \SOn$ we say a partial viable combination $$\mathfrak{v}_{j}:=((\rho_1, \ldots, \rho_j), (p_1, \ldots, p_j))$$ \emph{arises from a state $q \in T$} if there is  a set $\{\phi_1, \ldots, \phi_j\}$ of incomparable words such that $\lambda_{T}(\phi_i, q) = \rho_i$ and $\pi_{T}(\phi_i, q) = p_i$.

\end{definition}

\begin{remark}
There are two subtleties here.  Firstly, we are not insisting that a partial viable combination $$\mathfrak{v}_{j}:=((\rho_1, \ldots, \rho_j), (p_1, \ldots, p_j))$$ arising from a state $q \in T$ has $\bigsqcup_{1 \le i \le j} \rho_i \concat \Image(p_i) = \Image(q)$.  Secondly, we observe that $T\in \Onr$ still does not guarantee that $\Image(q)$ is equal to $\CCn$ (i.e., in the case that $q$ is not a homeomorphism state), so it can be the case that every partial viable combination arising from $q$ is not a full viable combination.
\end{remark}
\end{comment}

\begin{definition}
For a transducer $T \in \SOn$ we say a full viable combination $$\mathfrak{v}_{j}:=((\rho_1, \ldots, \rho_j), (p_1, \ldots, p_j))$$ \emph{arises from a state $q \in T$} if there is a word $\eta \in X_{n}^{*}$ and  a set $\{\phi_1, \ldots, \phi_j\}$ of incomparable words, such that the following holds:
\begin{enumerate}
\item  $\lambda_{T}(\phi_i, q) = \eta\rho_i$, $\pi_{T}(\phi_i, q) = p_i$ and,
\item  $\sqcup_{1\le i \le j} \eta \rho_i \Image(p_i) = U_{\eta}$.
\end{enumerate}
\end{definition}

Note that implicit in the proof of Theorem~\ref{thm:CoreExtensibility} is a full viable combination arising from a state.

Now we may modify the construction given in the proof of Theorem \ref{thm:CoreExtensibility}.
\begin{lemma}\label{lem:viableCombinationsExtensions}
Let $1\leq r\leq n-1\in\N$, $T\in\SOn$ and $\mathfrak{v}$
denote the set of all viable combinations for $T$.  The transducer $T$ can be extended to an initial transducer $U_{q_0}$ representing an element of $\Snr$ with $\core{U}=T$ if and only if there exists $\mathfrak{v}_{j_1},
\ldots, \mathfrak{v}_{j_m}$ a sequence of viable combinations such that $r \equiv
\Sigma_{1 \le i \le m} j_i \equiv m \mod{n-1}$.
\end{lemma}
\begin{proof}
Let us assume that $T\in\SOn$, $\mathfrak{v}$
denotes the set of all viable combinations of $T$, $1\leq r\leq n-1$, and there exists $\mathfrak{v}_{j_1},
\ldots, \mathfrak{v}_{j_m}$ a sequence of viable combinations such that $r \equiv
\Sigma_{1 \le i \le m} j_i \equiv m \mod{n-1}$.  We will show that $T$ can be extended to a transducer $U$ as claimed.

First, set $j\seteq \Sigma_{1 \le i \le m} j_i$.
There exist finite complete antichains $\{ \xi_1, \ldots, \xi_{j}
\}$ and $\{\zeta_1, \ldots, \zeta_m\}$ of $\CCnr$.  Decompose $\{ \xi_1, \ldots, \xi_{j}
\}$ as $$\{ \xi_1, \ldots, \xi_{j}
\}:= \sqcup_{1 \le l < m} \{\rho_{l,1},\rho_{l,2},\ldots,\rho_{l,j_l}\},$$ where   $\rho_{l,k}$ denotes the string $\xi_{\Sigma_{1 \le i \le l-1} j_i
+k}$, when $1\leq l<m$ and $1\leq k\leq j_l$.  In particular we have \[\{ \xi_{\Sigma_{1 \le i \le l-1} j_i
+1}, \ldots, \xi_{\Sigma_{1 \le i \le l} j_i}\}=\{\rho_{l,1},\rho_{l,2},\ldots,\rho_{l,j_l}\}.\]

For a fixed $l$, $1 \le l < m$, consider the block $\{\rho_{l,1},\rho_{l,2},\ldots,\rho_{l,j_l}\}$, then $\alpha_{T}$ acts
on the space $\sqcup_{1 \le i \le j_{l} U_{\rho_{l,i}}}  \cong \mathfrak{C}_{n,j_l}$ as the map $h_{\mathfrak{v}_{j_l}}$ as in Remark \ref{rem:viableToHomeo} (note here, it is possible that $j_l>n$). Noting that $\mathfrak{C}_{n,j_l}\cong \CCn$ and
injecting the image to the cone $U_{\zeta_l}$, we can extend this definition across each of the blocks $\{\rho_{l,1},\rho_{l,2},\ldots,\rho_{l,j_l}\}$ for the full range of values of $l$ to get a well defined homeomorphism $\alpha_{T}$ of $\CCnr$ with finitely many local actions, and where the set of local actions corresponding to the states of $T$ are the only ones which appear infinitely often (so that $T$ is the core of the minimal transducer $V_{q_0}$ representing $\alpha_{T}$.

The other direction is technical but straightforward.  Assume that $T\in\SOn$, and $T$ extends to an initial minimal transducer $V_{q_0}$ with $\core{V_{q_0}}=T$ representing an element of $\Snr$, with initial state $q_0$ and output and transition functions $\lambda$ and $\pi$, respectively.  As $T$ is the core of $V_{q_0}$ it is already the case that if we look at all strings of $\Wnr$ of length $k+1$ where $k$ is the synchronizing length of $V_{q_0}$, then this provides a complete antichain for $\Wnr$ and reading any of these strings from the initial state of $V_{q_0}$ results in moving to a state in $T$.

Since core states of $V_{q_0}$ cannot write a letter from the alphabet $\rd$ we observe firstly that each word $\lambda(w,q_0)$ is non-empty for any $w\in \Wnr$ of length $k+1$, and further, as $h_{q_0}$ is a homeomorphism of $\CCnr$ that for each letter $\dot{a}\in \rd$ we have that the cone $U_{\dot{a}}$ decomposes as a disjoint union
\[
U_{\dot{a}}=\sqcup_{w\in X_a} \lambda(U_w,q_0)
\] where $X_a$ denotes the set of strings of length $k+1$ in $\Wnr$ so that if $w\in X_a$ we have $\lambda(w,q_0)$ is a string with first letter $\dot{a}$.  Let $Y_a\seteq(y_{w_1},y_{w_2},\ldots,y_{w_p})$ where we have $|X_a|=p$, and where we order $w_1<w_2<\ldots< w_p$ in the dictionary order on $X_a$, and where we further determine $y_{w_i}$ by the rule $\dot{a}y_{w_i}=\lambda(w_i,q_0)$ for each $w_i\in X_a$.  For  each $1\leq i\leq p$ set $q_i=\pi(w_i,q_0)$ the state reached in $V_{q_0}$ upon reading the word $w_i$ from $q_0$. Further set $Z_a=(q_1,q_2,\ldots,q_p)$, so that $(Y_a,Z_a)$ is a viable combination.  Now repeating this for each letter in $\rd$ gives us a set of viable combinations satisfying the modular arithmetic of the claim.

 \end{proof}

We may refer to a transducer $A_{q_0}\in\Snr$, which has core transducer $T$, as a \emph{completion of $T$ in $\Snr$}.   Note that if $A_{q_0}$ is bisynchronizing, then we might refine this language and say that $A_{q_0}$ is a completion of $T$ in $\Bnr$.

As the completions found in Lemma \ref{lem:viableCombinationsExtensions} create no cycles outside of the core transducer $T$ given to the process, thus, the resulting transducer in $\Snr$ is actually in $\Bnr$ when $T\in\On$ (that is, bi-synchronizing, not just strongly synchronizing).  Therefore, we obtain the following corollary.
\begin{cor}\label{cor:bisyncExtensions}
Let $1\leq r\leq n-1\in\N$, $T\in\On$ and $\mathfrak{v}$
denote the set of all viable combinations for $T$.  The transducer $T$ admits a completion to an initial transducer $U_{q_0}$ representing an element of $\Bnr$ if and only if there exists $\mathfrak{v}_{j_1},
\ldots, \mathfrak{v}_{j_m}$ a sequence of viable combinations such that $r \equiv
\Sigma_{1 \le i \le m} j_i \equiv m \mod{n-1}$.
\end{cor}

We are finally in position to discuss nesting properties for the subgroups $\Onr$ within $\On$.

Below, for integers $i$, $j$, and $n$, we write $i\mid j\mod{(n-1)}$ if there is an integer $k$
so that $i$ divides $j+k(n-1)$ in the ring of integers.
\begin{prop}\label{prop: nesting}
Suppose $0<i\leq j\leq n-1$ are natural numbers.  If $i \mid j \mod{(n-1)}$ then $\Onk{n}{i}\leq \Onk{n}{j}$.
\end{prop}
\begin{proof}
Let $0<i\leq j\leq n-1$ be natural numbers and assume $i \mid j \mod{(n-1)}$, and $A\in \Onk{n}{i}$ and $B\in \Bnk{n}{i}$ with core $A$.  We will show that $A\in\Onk{n}{j}$.

Since $i \mid j \mod{(n-1)}$, there is a complete antichain $\chi$ in $\wnr{n}{j}$ of length $mi$ for some $m>1$, where we may assume that every element of this antichain is a word of length at least two (i.e., not just a dotted letter).

Build a partition $\rho$ of $\chi$ into $m$ disjoint sets of cardinality $i$, and give each of these sets a bijection to the set of dotted letters from $0$ to $i-1$.  Construct a homeomorphism $\theta_{A,B}$ of $\CCnk{n}{j}$ as $m$ copies of $B$, by using a copy of $B$ for each element of the partition $\rho$, where we replace each dotted letter (from the perspective of $B$) by the string from the corresponding element of the partition under the appropriate bijection.  The homeomorphism $\theta_{A,B}$, when realised as a minimal transducer, has core $A$ and is an element of $\Bnk{n}{j}$.  In particular, $A\in \Onk{n}{j}$.

\end{proof}

Proposition \ref{prop: nesting} shows that subgroups $\Onr$ of $\On$ have a type of lattice structure, with the group $\Onk{n}{1}$ being minimal and a subgroup of the other groups while the group $\Onk{n}{n-1}$ is a container for all of the other groups and hence is $\On$.  In particular, we have.
\begin{prop}\label{Proposition: On is a group}
The set $\On$ under the multiplication inherited from $\SOn$ forms a group.
\end{prop}
\begin{proof}
This follows from Proposition \ref{prop: nesting} as we immediately have $\On=\Onk{n}{n-1}$, while $\Onk{n}{n-1}$ is a group.
\end{proof}

\subsection{On Lipschitz conditions, the groups $\On$, and subgroups of interest\label{s_lipschitz}}
Recall our notation that $\LBnr$ represents the subgroup of $\Bnr$ corresponding to the homeomorphisms of $\CCnr$ that are bi-lipschitz.  Also recall that we call $\Onr$ the image of $\Bnr$ under the quotient by its normal subgroup $\Gnr$, and the image of $\LBnr$ in $\Onr$ is the group $\Lnr$.  We further have the subgroup $\HBnr$ of $\LBnr$ consisting of homeomorphisms which can be represented by transducers with cores that are in fact synchronous, bi-synchronizing, and which contain a state which is a homeomorphism state, and we denote the similarly defined larger subgroup of $\On$ as $\Hn$.  (The condition above on the existence of a homeomorphism state is in fact necessary to create a subgroup for the synchronous bi-synchronizing core automata arising.  In \cite{BleakCameronOlukoya} the authors give an example of a transducer with all of the other properties which when squared becomes asynchronous after minimisation.)   In this subsection we explore some properties of these various groups.  The examples and discussion provided here imply all of the statements of Theorem \ref{thm:Examples}.

In what follows, we often discuss the core transducers which are simultaneously sub-transducers of initial transducers representing elements of the groups $\Snr$ and $\Bnr$, but also, they represent images under the quotients mentioned in the last paragraph.  These transducers do not naturally have initial states.  We show in Figure \ref{Fig: example showing dependence on r} of the following subsection an example of a transducer in $\mathcal{L}_{4,3}$ which has no homeomorphism state, even though it is in fact synchronous.  Thus, in these investigations, sometimes we focus on the combinatorics of the transducers as finite objects, and sometime we focus on their ``local dynamics'' as determined by their local maps.

For given $1\leqslant r<n\in\N$, recall that elements of $\Gnr$ are bi-Lipschitz. Also, due to its strong condition of being bi-synchronizing, any general element of $\Bnr$ might be expected to be bi-Lipschitz.  However, this first impression turns out to be false.

\begin{theorem}
For $1\leqslant r<n$, for $n>2$, the group $\Bnr$ contains elements which are not bi-Lipschitz.
\end{theorem}
\begin{proof}
The transducer $B$ depicted in Figure \ref{fig-Shayo} is bi-synchronizing (at level two), minimal, and is its own core.  However, it is not synchronous.  In this case we are demonstrating for $n=3$.

\begin{figure}[htbp]
\begin{center}
 \begin{tikzpicture}[->,>=stealth',shorten >=1pt,auto,node distance=2.3cm,on grid,semithick,
                    every state/.style={fill=red,draw=none,circular drop shadow,text=white}]
   \node[state] (a)   {$a$};
   \node[state] (b) [below left=of a] {$b$};
   \node[state] (c) [below right=of a] {$c$};
    \path[->]
    (a) edge [out=105,in=135,loop] node [swap]{$1|0$} ()
        edge [out=45,in=75,loop] node [swap]{$0|2$} ()
        edge  [out=185,in=85]node [swap]{$2|1$} (b)
    (b) edge [out=200,in=230, loop] node [swap] {$2|1$} ()
        edge[out=65,in=205]  node  [swap]{$0|00$} (a)
        edge [out=330,in=210] node [swap]{$1|\varepsilon$} (c)
    (c) edge[out=95,in=355]  node [swap]{$0|02$} (a)
        edge [out=115,in=335]  node {$1|2$} (a)
        edge [in=345,out=195] node [swap] {$2|01$} (b);
\end{tikzpicture}
\end{center}
\caption{A non-bi-Lipschitz core transducer of infinite order.\label{fig-Shayo}}
\end{figure}

Although the initial transducer $B_{a}$ is bi-synchronizing (at level two) and will induce a self-homeomorphism of $\mathfrak{C}_3$, it does not produce a Lipschitz map on Cantor space.  Consider the cycle of edges labelled by inputs $2$ and $0$ respectively (connecting $a$ to $b$ and then going back to $a$).  This cycle has length two in the directed graph underlying the transducer and so input word of length two.  However, the output word written on reading this input while traversing this cycle has length three.  In particular, the   point $p=20\overline{20}$ maps to $q=100\overline{100}$ in a non-Lipschitz fashion as points from ever-smaller basic cones about $20\overline{20}$ face an ever-increasing general contraction factor (a witness to this is the sequence of  points $((20)^k\overline{0})_{k\in \N}$, where the $k^{th}$ such point has distance $1/{2^{(2k+1)}}$ from $p$, but lands under $\phi$ at the point $(100)^k\overline{2}$ at distance $1/2^{3k+1}$ from $q$.

It is not hard to build similar examples for any alphabet of size $n$, for $n>2$, based on this example.  In particular, we can build elements of $\Bnr$ for any $n>2$ which are not bi-Lipschitz.

To see that we can build a bi-synchronizing transducer $\widehat{B}_{n,r}$ representing an element $\tilde{B}_{n,r}$ of $\Bnr$ which uses $B$ as part of its core, we will simply add a state to the transducer $B$ depicted above, and many transition edges, as described below.

Figure \ref{fig-Shayo-more} depicts the resulting transducer $\widehat{B}_{4,2}$ (so, for $n=4$ and $r=2$).  Details of the construction follow.

\begin{figure}[htbp]
\begin{center}
\begin{tikzpicture}[->,>=stealth',shorten >=1pt,auto,node distance=2.6cm,on grid,semithick,
                    every state/.style={fill=red,draw=none,circular drop shadow,text=white}]   \node[state] (q) {$q_0$};
   \node[state] (a) [below = of q]  {$a$};
   \node[state] (b) [below left=of a] {$b$};
   \node[state] (c) [below right=of a] {$c$};
    \path[->]
    (q) edge  [out=265,in=95]node [swap]{$_{\dot{0}|\dot{0}}$} (a)
        edge  [out=275,in=85]node {$_{\dot{1}|\dot{1}}$} (a)
    (a) edge [out=115,in=145,loop] node [swap]{$_{1|0}$} ()
        edge [out=35,in=65,loop] node [swap]{$_{0|2}$} ()
        edge [out=255,in=285,loop] node [swap]{$^{3|3}$} ()
        edge  [out=225, in=65] node  {$\!\!\!_{2|1}$} (b)
    (b) edge [out=200,in=230, loop] node [swap] {$2|1$} ()
        edge[out=90,in=180]  node  [swap, near end]{$\!\!\!\!\!\!_{0|00}$} (a)
        edge[out=100, in=170]  node  {${3/3}$} (a)
        edge [out=330,in=210] node [swap]{$1|\varepsilon$} (c)
    (c) edge[out=80,in=10]  node [swap]{$0|02$} (a)
        edge[out=90,in=0]  node [near end]{$_{3|03}\!\!\!\!\!\!\!$} (a)
        edge [out=115,in=315]  node {$_{1|2}\!\!\!$} (a)
        edge [in=345,out=195] node [swap] {$2|01$} (b);
\end{tikzpicture}
\end{center}
\caption{A non-bi-Lipschitz transducer $\widehat{B}_{4,2}$ of infinite order in $\mathcal{S}_{4,2}$.\label{fig-Shayo-more}}
\end{figure}

The added state will be an initial state $q_0$.  For each symbol $x\in\rd$,  the transitions will be given by the rules $\pi(x,q_0)=a$ and $\lambda(x,q_0)=x$.  That is, $q_0$ admits $r$ edges from $q_0$ to $a$, one for each input from $\rd$, and each transition acts as the identity transformation on its letter of $\rd$.

The remaining transitions are as follows.  For each symbol $x$ in $\Xn\backslash\{0,1,2\}$, set $\pi(x,a)=a$ and $\lambda(x,a)=x$.  Furthermore, set $$\pi(x,b)=a \mbox{ and }\lambda(x,b)=x.$$  Finally, set $\pi(x,c)=a$ and $\lambda(x,c)=0x$.

The reader can verify that the transducers so constructed are bi-synchronizing and represent self-homeomorphisms for the spaces $\CCnr$.
\end{proof}

As mentioned above, it is perhaps surprising that for $n>1$ and $1\leq r<n$, the group $\Bnr$ contains elements that are not bi-Lipschitz.  As in fact the subgroup $\LBnr$ of bi-lipschitz homeomorphisms within $\Bnr$ is proper, one might now go to the other extreme, and wonder that perhaps $\LBnr$ would not have a complicated image in $\Onr$ (observing of course that $\Gnr\leq \LBnr$).  However, while the bi-Lipschitz condition is a very strong further condition on the bi-synchronizing maps, in \cite{BleakCameronOlukoya} the authors are able to show that the group $\LBnr$ contains (most of) the complexity of the automorphism group of the (two-sided) full shift on $n$ letters.

\medskip

The example depicted in Figure \ref{fig-L2-Big} is of an element in $\mathcal{L}_2$ that is not in the subgroup $\mathcal{H}_2$, and which is of infinite order.

We observe that elements of $\On$ act on the quotient space $\Xn^{\Z}/\langle \sigma\rangle$, where $\sigma$ is the shift. (Note that a transducer representing an element of $\On$ has no well defined index to which it is writing; for instance, this is clear if it does not represent the core of a Lipschitz transformation. However, because of synchronisation, such a transducer will transform a bi-infinite string to another bi-infinite string; one always has sufficient history to know what to do with a new letter of input. Note that this operation takes any two shift-equivalent bi-infinite sequences to two bi-infinite sequences which are also shift equivalent if indices are assigned.)  From this perspective, we can verify that the transducer in Figure \ref{fig-L2-Big} is actually of infinite order by finding an infinite orbit of the induced action on this quotient space.  For this purpose, the equivalence class of the point $\ldots111(001)^\omega$ (take as representative the point defined using coordinate entries of value $1$ at all negative indices, and with its first entry of value $0$ occurring at index $0$), serves nicely. The reader can verify this by a simple induction tracing the orbit as given below.
\begin{align*}
\ldots111(001)^\omega&\mapsto\ldots111\cdot01\cdot(001)^\omega\mapsto\\
\ldots111\cdot0101\cdot(001)^\omega&\mapsto\ldots111\cdot(01)^3(001)^\omega\mapsto^k\\
\end{align*}

\vspace{-.35 in}

\centerline{${\ldots111\cdot(01)^{k+3}\cdot(001)^\omega}$}
\vspace{.1 in}

\begin{figure}[htbp]
\begin{center}
\begin{tikzpicture}[->,>=stealth',shorten >=1pt,auto,node distance=2.3cm,on grid,semithick,
                    every state/.style={fill=red,draw=none,circular drop shadow,text=white}]
   \node[at={(3,5)},state] (a) {$a$};
   \node[at={(-3,5)},state] (b)  {$b$};
   \node[at={(-5,0)},state] (c) {$q_0$};
   \node[at={(0,2.5)},state] (d)  {$d$};
   \node[at={(3.0,2.5)},state] (e) {$e$};
   \node[at={(5,0)},state] (f)   {$f$};
   \node[at={(0,0)},state] (g)  {$g$};
   \node[at={(0,-3)},state] (h)  {$h$};
   \node[at={(-3,-5)},state] (i){$i$};
   \node[at={(3,-5)},state] (j)  {$j$};
    \path[->]
    (a) edge [out=180,in=0]node [swap]{$0|\varepsilon$} (b)
        edge [out=275,in=85]node {${1|\varepsilon}\!\!\!$} (e)
    (b) edge node [swap] {$0|100$} (c)
        edge [out=275,in=140]node [near start] {$1|01\!\!\!$} (g)
    (c) edge [out=167,in=193, loop] node [swap]{$0|0$} ()
        edge [out=15,in=195] node [near start,swap]{$1/1$} (d)
    (d) edge node {$0|\varepsilon$} (a)
        edge node {$1|\varepsilon$} (g)
    (e) edge [out=95,in=265]node {$_{0|01\!}$} (a)
        edge node [near start, swap]{$_{1|101}\!$} (d)
    (f) edge [out=95,in=320]node [swap]{$0|1$} (a)
        edge node {$1|01$} (d)
    (g) edge node [near end,swap]{$_{1|11}\!$} (h)
        edge [out=330,in=105] node [near start]{$0|00$} (j)
    (h) edge [out=212,in=238, loop] node [swap]{$1/1$}()
        edge node [near start]{$\!\!\!\!{0|0}$} (j)
    (i) edge node {$0|00$} (c)
        edge [out=75,in=210] node [near end]{$1|1$} (g)
    (j) edge node {$0|\varepsilon$} (i)
        edge node {$1|\varepsilon$} (f);
\end{tikzpicture}
\end{center}
\caption{An element of $\mathcal{L}_2$ of infinite order, which is not in $\mathcal{H}_2$.\label{fig-L2-Big}}
\end{figure}

 It is not hard, using the method of the construction following the example in Figure \ref{fig-Shayo-more}, to increase the alphabet size to $n$ for any given integer $n>2$.

\medskip

We give two further example transducers, exploring the boundary cases when $n$ is small.

The first such example, depicted in Figure \ref{fig-Inf-P3}, is of a transducer representing an element of $\mathcal{H}_3$ which is of infinite order.  One can prove by induction that the point $\overline{p}\seteq(p_i)\in\{0,1,2\}^{\Z}$ given by the rules
$$p_i=\left\{\begin{matrix}2 \textrm{ if }i<0\,\\1  \textrm{ if }i\geqslant0.\end{matrix}\right.$$
witnesses an infinite orbit under the iterated action of the transducer of Figure \ref{fig-Inf-P3}.

\begin{figure}[htbp]
\begin{center}
 \begin{tikzpicture}[->,>=stealth',shorten >=1pt,auto,node distance=5cm,on grid,semithick,
                    every state/.style={fill=red,draw=none,circular drop shadow,text=white}]  \node [state] (A)                {$q_0$};
  \node [state] (B)  [right= of A] {$b$};
 \path [->]
 (A) edge [out=145,in=170, loop] node [swap]{$0/1$} (A)
 (A) edge [out=190,in=215, loop] node [swap]{$1/2$} (A)
 (A) edge                        node [swap]{$2/0$} (B)
 (B) edge [out=150,in=30]        node [swap]{$1/1$}(A)
 (B) edge [out=210,in=330]       node {$0/2$}(A)
 (B) edge [out=347,in=13,loop]  node [swap]{$2/0$}(B);
\end{tikzpicture}
\end{center}
\caption{An element of $\mathcal{H}_3$ with infinite order.\label{fig-Inf-P3}}
\end{figure}

Our final example, depicted in Figure \ref{fig-torsion-O2}, is of a non-bi-Lipschitz torsion element of $\mathcal{O}_{2,1}$, given as below.  One can check combinatorially that the transformation represented by this transducer has order two.

\begin{figure}[htbp]
\begin{center}
\begin{tikzpicture}[->,>=stealth',shorten >=1pt,auto,node distance=5cm,on grid,semithick,
                    every state/.style={fill=red,draw=none,circular drop shadow,text=white}]  \node [state] (A)                {$q_0$};
  \node [state] (B)  [above= of A] {$b$};
  \node [state] (C)  [right= of B] {$c$};
  \node [state] (D)  [below= of C] {$d$};
 \path [->]
 (A) edge [out=212,in=237, loop] node [swap]{$0/0$} (A)
 (A) edge [out=85,in=275]        node [swap]{$1/1$} (B)
 (B) edge [out=265,in=95]        node [swap]{$0/10$} (A)
 (B) edge                        node [swap]{$1/\varepsilon$} (C)
 (C) edge                        node {$0/0$} (A)
 (C) edge                        node {$1/11$} (D)
 (D) edge                        node [swap]{$0/0$} (A)
 (D) edge [out=313,in=337, loop] node [swap]{$1/1$} (D);
 \end{tikzpicture}
 \end{center}
\caption{A torsion non-bi-Lipschitz map in $\mathcal{O}_2$.\label{fig-torsion-O2}}
\end{figure}

\subsection{More on the dependence of $\Out(\Gnr)$ on $r$}\label{ssec:dependence on r}
In Proposition \ref{prop: nesting}, we give a condition determining the nesting of the subgroups $\Onr$ of $\On$.  In this subsection we give an example transducer $T$ representing an element of  $\mathcal{L}_{4,3}$  with no homeomorphism states.  The transducer $T$ admits no completion to a larger transducer representing an element of $\mathcal{B}_{4,1}$ or of $\mathcal{B}_{4,2}$.  This shows us, e.g., that the induced actions of the cores of the transducers representing elements of the groups $\mathcal{B}_{4,2}$ and $\mathcal{B}_{4,3}$ on the set $\{0,1,2,3\}^{\Z}$ are not the same.  Therefore, we highlight the possibility that neither of the groups $\mathcal{O}_{4,1}$ and $\mathcal{O}_{4,2}$ are isomorphic to  $\mathcal{O}_{4,3}$.

Consider the transducer $T$ in Figure~\ref{Fig: example showing dependence on r}, we shall show below that all possible completions of this transducer are in $\mathcal{S}_{4,3}$.

    \begin{lemma}\label{lem:Completions in 4,3}
     Let $T$ be the transducer in Figure~\ref{Fig: example showing dependence on r}. Every possible  completion of $T$ is in $\mathcal{S}_{4,3}$.
    \end{lemma}

\begin{proof}

 First observe that $\Image(q_1) \sqcup \Image(q_4) = \CCn$, $\Image(q_3) \sqcup \Image(q_4) = \CCn$ and $\Image(q_2) \sqcup \Image(q_5) = \CCn$. Moreover,  for $(q_i,q_j) \in  Q_T \times Q_T \backslash \{(q_1,q_4), (q_3,q_4), (q_2,q_5)\}$, the union $\Image(q_i) \cup \Image(q_j)$ is not a disjoint union, and also, is not equal to $\CCn$.

 Let $A_{q_0}$ be a completion of $T$ in  $\mathcal{S}_{4,r}$
 for $0 < r < 4.$ Let $l$ be such that for all words
 $\Gamma$ of length $l$ in $W_{4,3}$ we have $\pi(\Gamma, q_0)$ is a state of $T$, and for all words $\chi\in W_4$ of length $l$  and any state $q$ of $Q_A\backslash\{q_0\}$ we
 have $\pi(\chi, q)$ is a state of $T$. Let $\phi_1, \ldots,
 \phi_j$ be all words of length $l$ in $W_{4,3}$ (so in particular $j=3\cdot 4^{l-1}$). Let $\rho_i :=
 \lambda_{A}(\phi_i, q_0)$ and let $p_i := \pi_{A}(\phi_i,
 q_0)$, for $1 \le i \le j$. Let $\rho_{k}\concat x_{k,i} :=
 \rho_{k}
 \concat \lambda_{A}(i, p_k)$, and $p_{k,i} = \pi_{A}(i, p_k)$
 for $1 \le k \le j$ and $i \in X_4$. Observe that for any
  state $q$ of $T$, the map $\lambda_{T}(\centerdot, q): X_4
  \hookrightarrow X_4$ has image size two and the set of
  pre-images of a point in the image also has size 2. Moreover if
  $i_1 \ne i_2$ are elements of $X_4$ which have the same image
  under $\lambda_{T}(\centerdot, q)$ then $(\pi_{T}(i_1
  q),\pi_{T}(i_2 q)) \in \{(q_1, q_4),(q_4, q_1),(q_3, q_4),(q_4, q_3),(q_2, q_5), (q_5, q_2)\}$. This means that for a given $k$ with $1 \le k \le
   j$, $\rho_k \concat x_{k, i_1} = \rho_{k} \concat x_{k, i_2}$ for some distinct
   $i_1$ and $i_2$ in $X_4$. Moreover since $\Image(p_{k, i_1})
   \sqcup \Image(p_{k,i_2}) = \mathfrak{C}_{4}$, then $\rho_k'\concat x_{k', i}$ is
   incomparable to $\rho_k\concat x_{k, i_1}$ otherwise $A_{q_0}$ is not injective. Thus the set $\{\rho_{k}\concat x_{k,
    i} \mid 1 \le k \le j, i \in \Xn \}$ has size $2j$.
    Since $A_{q_0}$ is a homeomorphism, we must have that the
    set $\{\rho_{k}\concat x_{k,i} \mid 1 \le k \le j, i \in \Xn \}$
    forms an antichain for $\mathfrak{C}_{4,3}$ otherwise
    $\sqcup_{1 \le k \le j, \ i \in \Xn} \rho_{k} \concat
    x_{k,i} \concat \Image(p_{k,i})  \ne \mathfrak{C}_{4,3}$.
    We conclude that $2j \equiv r \mod 3$. However, since $j=3\cdot 4^{l-1}$ we see that
    $j \equiv 0\mod 3$. Thus, we can conclude $r$ is congruent to $0$ mod 3, and therefore $r = 3$.
    \end{proof}

   Having shown that $\Onk{4}{1}$ and $\Onk{4}{2}$ sit properly as subgroups of $\Onk{4}{3}$, one could now consider whether or not the $\Onr$ have their isomorphism type depending on $r$ for a fixed $n$.  We mention this question in the final section.

   \subsection{Constructing inverses within $\On$} \label{sec:inversion}
In this subsection, we will look at the combinatorial construction of inverses of elements in $\On$.

Here and until the statement of the next lemma, fix $T \in \On$.  Let $q \in T$ be an arbitrary state. Let
$\Image(q) = U_{\eta_1} \sqcup U_{\eta_2} \sqcup \ldots \sqcup
U_{\eta_l}$ for $\eta_j \in \Xn^{*}$. Fix $i$ such that $1 \le i \le l$ and let $\nu_i :=
{L}_{q}(\eta_i)$, $p_i:= \pi(\nu_i, q)$ and $w_i = \eta_i -
\lambda(\nu_i, q)$. Notice that by definition of the function
${L}_{q}$ we must have that $U_{w_i} \subseteq \Image(p_i)$.
Further observe that ${L}_{q}(\eta_ix) = {L}_{q}(\eta_i)\concat
{L}_{p_i}(w_i\concat x)$. Therefore we may recursively define sets $Q_j'$ by the rules $$Q_0':= \{(w_i,p_i) \mid 1 \le i  \le l\}$$ and for $k \ge 0$ let
\begin{IEEEeqnarray*}{rCl}
Q_{k+1}' :=  &\{& \left(wx - \lambda({L}_{p}(w\concat x), p),
\pi({L}_{p}(w\concat x), p)\right) \mid x \in \Xn, (w,p) \in
Q_{k}'\} \nonumber \\
&\cup& Q_k'.
\end{IEEEeqnarray*}
By the previous observation if $(w,p) \in
Q_{k}'$ for some $k$ then ${L}_{p}(w) = \epsilon$ and $U_{w}
\subseteq \Image(p)$. Now since each state $p \in T$ has clopen
image and $T$ has finitely many states, there are only finitely
many words $w \in \Xn^{\ast}$ such that ${L}_{p}(w) =
\epsilon$ and such that $U_{w} \subseteq \Image(p)$. Therefore
there is a $j \in \N$ such that $Q_j' =   Q_{j+1}'$. Set $Q':=
Q'_{j}$. Now let $T'= \langle\Xn, Q', \pi', \lambda'\rangle$ be the
transducer with state set $Q'$ and input and output function
defined as follows:
\begin{IEEEeqnarray*}{rCl}
\pi'(x, (w,p)) & = & \left(wx -
\lambda({L}_{p}(w\concat x), p),
\pi({L}_{p}(w\concat x), p)\right); \nonumber \\
\lambda'(x, (w,p)) & = &
{L}_{p}(w\concat x)
\end{IEEEeqnarray*}
for all $x \in \Xn$ and $(w,p) \in Q'$. We may extend the
definition of $\pi'$ and $\lambda'$ to elements of $\CCn$ in the
obvious way. Observe that since $\lambda'(\delta, (w,p)) =
{L}_{p}(w\delta) = (w\delta)(T_{p})^{-1}$ for $\delta \in
\CCn$ and $(w,p)$ a state of $Q'$, then $(w,p)$ induces an
injective map on $\CCn$.

Note that $T'$ has no states of incomplete response: if there is some state $(w,p)$ of $Q'$ such that the set of words $\{\lambda'(x, (w,p)) \mid x \in \Xn \}$ admits a non-trivial common prefix, then ${L}(p, w) \ne \epsilon$ which is not possible by our construction.

Denote by $T^{-1}$ the result of identifying equivalent states of $T'$.  Motivated by the following lemma, we call $T^{-1}$ the \emph{core inverse of $T$}.

\begin{lemma}\label{lemma: core inverse and core of inverse of completion coincide}
Let $T = \langle\Xn, Q, \pi, \lambda\rangle \in \On$, and $$A_{q_0} = (\dot{\bf{r}},\Xn, R,S,\pi_0, \lambda_0,q_0)$$ be a completion of $T$ representing an element of $\Bnr$ (where $Q \subseteq S$).  Further suppose $B_{q_1}$ is the minimal transducer representing the inverse of $A_{q_0}$. If $T^{-1} = \langle\Xn, Q', \pi',\lambda'\rangle$ is the core inverse of $T$ as constructed above then $\core{B_{q_1}}$ is $\omega$-equivalent to  $T^{-1}$.
\end{lemma}
\begin{proof}
Let $A_{(\epsilon, q_0)} = (\rd, \Xn, R', S', \lambda'_0, \pi'_0, (\epsilon, q_0))$ be the transducer constructed by the GNS inversion algorithm in \cite{GNSenglish}. Notice that $B_{q_1}$ is $\omega$-equivalent to $A_{(\epsilon, q_0)}$. Our strategy will be to show that after reading any sufficiently long input in $A_{(\epsilon, q_0)}$ the resulting active state is a state of $T'$.

Let $q$ be a state of $T$ and let $\Image(q) = U_{\eta_1} \sqcup
U_{\eta_2} \sqcup \ldots \sqcup
U_{\eta_l}$ for some $\eta_i \in \Xn^{*}$. Let $M := \max\{
\eta_i \mid 1 \le i \le l\}$ and let $\mu$ be any word in
$\Xn^{\ast}$ such that $\lambda(\mu, q) = \eta_i\concat \Delta$ for
some $\Delta \in \Xn^{\ast}$ and some $1 \le i \le l$. Let $\Gamma \in \Wnr$ be any word
such that $\pi_0(\Gamma, q_0) = q$, observe that since $q \in S$
it must be the case that $\lambda_0(\Gamma, q_0) \ne \epsilon$.
We therefore have that $\lambda_0(\Gamma\concat \mu, q_0) =
\lambda_0(\Gamma, q_0)\concat\lambda(\mu, q)$. Consider now the
set ${L}_{q_0}(\lambda_0(\Gamma, q_0)\concat \eta_i)$. Since
$U_{\eta_i} \subset \Image(q)$, either $A_{q_0}$ is not
injective or $\Gamma$ is a prefix of  ${L}_{q_0}(\lambda_0(\Gamma,
q_0)\concat \eta_i)$. Thus we have  ${L}_{q_0}(\lambda_0(\Gamma,
q_0)\concat \eta_i) = \Gamma \concat {L}_{q}(\eta_i)$ and
\begin{IEEEeqnarray*}{rCl}
\pi'_0(\lambda_0(\Gamma, q_0) \concat \eta_i, (\epsilon, q_0)) &=& (\lambda_0(\Gamma, q_0) \concat \eta_i - \lambda_0(\Gamma \concat{L}_{q}(\eta_i), q_0), \\ &\pi&(\Gamma \concat{L}_{q}(\eta_i), q_0  )) \\&=& ( \eta_i - \lambda({L}_{q}(\eta_i), q), \pi({L}_{q}(\eta_i), q  )).
\end{IEEEeqnarray*}

By construction,  $(\eta_i - \lambda({L}_{q}(\eta_i), q), \pi({L}_{q}(\eta_i), q  ))$ is a state of $T'$.  Since $\Gamma \in W_{n,r}$ was arbitrary such that $\pi_0(\Gamma, q_0) = q$ and since $\core{A_{q_0}} = T$, we may chose $\Gamma$ as long as  we please, so that $\lambda_0(\Gamma, q_0)\concat \eta_i$ can be made arbitrarily long. Since $A_{(\epsilon, q_0)}$ is strongly synchronizing, as $A_{q_0} \in \Bnr$, then it follows that for sufficiently long $\lambda_0(\Gamma, q_0)\concat \eta_i,$ $$\pi'_0(\lambda_0(\Gamma, q_0) \concat \eta_i, (\epsilon, q_0))$$ is a state of $\Core(A_{(\epsilon, q_0)})$ and also a state of $T'$.  Now, since the core of $B_{q_1}$ is connected and minimal, we have $T^{-1} = \Core(B_{q_1})$.
\end{proof}

\begin{cor}
For any $1\leq r<n$, if $T \in \Onr$ then $T^{-1} \in  \Onr$ and we have $TT^{-1}=1_{\Onr}=T^{-1}T$.
\end{cor}

\begin{remark}
For $T \in \On$, the lemma above demonstrates that if there is a completion of $T$ which is bi-synchronizing, then all completions of $T$ are bi-synchronizing. In particular all completions of $T$ belong to some $\Bnr$ for appropriate $r$ and the cores of their inverses are all strongly isomorphic to $T^{-1}$.
\end{remark}

\section{Root functions, and detecting flavours of synchronicity}\label{sec:synchronization}
The goal of this section is to explore further properties of the root function, and to determine when a specific homeomorphism of $\CCnr$ begins to act (locally) as a synchronous transducer.  While not central to the main results of the paper, we find these results to be of interest when working with the group $\Bnr$, and in particular when one is trying to discern conditions that force a general element in $\Bnr$ to actually reside in one of the subgroups $\LBnr$ or $\PBnr.$

We start this discussion with the next lemma.

\begin{lemma} \label{claim:propert}
Let $g,h\in \Homeo(\CCnr)$, and let $\nu, \eta \in \Wnre$. Then:
 \begin{enumerate}
\item
$\theta_{h}$ is monotonic.

\item
$\nu\theta_{h}=\eta\theta_{h}$ and
$h_{\nu}=h_{\eta}$ \;\;\; $\Longrightarrow$ \;\;\; $\nu=\eta$.

\item
$\theta_{h}$ is injective $\Longleftrightarrow$ $\theta_{h}$ is an automorphism of $\Wnre$.

\item
$(\nu\theta_{h})\theta_{g}\leqslant (\nu\theta_{hg})$.

\item
$(\nu\theta_{h})\theta_{g}=\nu\theta_{hg}$
$\Longleftrightarrow$ $h_{\nu} g_{_{(\nu\theta_{h})}}
=(hg)_{\nu}$.

\item
$(\nu\theta_{h})\theta_{h^{-1}}=\nu$
$\Longleftrightarrow$ $(h^{-1})_{(\nu\theta_{h})}
=(h_{\nu})^{-1}$ $\Longleftrightarrow$ $U_{\nu}h=U_{(\nu\theta_{h})}$.

\item
If $\nu\theta_{h}\leqslant \nu'\theta_{h}$ implies $\nu\leqslant \nu'$
for every $\nu' \in \Wnre$, then $U_{\nu}h=U_{(\nu\theta_{h})}$.
\item

Assume that $U_{\nu}h=U_{(\nu\theta_{h})}$ and
$\nu\theta_{h}<(\nu \concat a)\theta_{h}$ for all $a\in\Xn$.
Then there is a permutation $\pi\in S_n$ so that $U_{\nu\concat a}h=U_{((\nu\concat (a\pi))\theta_{h})}$ for all $a\in\Xn$.

\item
$U_{\nu'}h = U_{((\nu')\theta_{h})}$ holds for every $\nu' \in \Wnre$
if and only if $\theta_h \in \Aut(\Wnre)$.

\item
There is no common nontrivial initial segment for all minimal elements of the set
$\{ \tau\theta_{h_{\nu}} \mid \tau \in {W_{n}} \} \setminus \{ \varepsilon \}$.
 \end{enumerate}
\end{lemma}

\begin{proof}
$(1)$  For $\nu,\eta\in \Wnr$, we have
$$\nu\leqslant \eta \Longleftrightarrow
U_{\nu}\supseteq U_{\eta}
\Longleftrightarrow
U_{\nu}h\supseteq U_{\eta}h,$$
and the last inclusion implies that  $\Root(U_{\nu}h)\leqslant \Root(U_{\eta}h)$,
thus $\nu\theta_{h}\leqslant \eta\theta_{h}$.

$(2)$ $\nu\theta_{h}=\eta\theta_{h}$ and
$h_{\nu}=h_{\eta}$ imply that $U_{\nu}h=U_{\eta}h$,
which happens if and only if $U_{\nu}=U_{\eta}$, thus $\nu=\eta$.

$(3)$ One direction is trivial. In the other direction, assume that
$\theta_{h}$ is injective. We prove by induction on $l=|\nu|$ that
$|\nu|=|\nu\theta_{h}|$. That will imply that $\theta_{h}$ is a
bijection, and since it is also monotonic, the assertion follows.

The claim is obvious for $l = 0$.
Assume that for all $\eta \in \Wnre$ with $|\eta|<l$, $|\eta|=|\eta\theta_{h}|$
and $U_{\eta}h=U_{(\eta\theta_{h})}$.
Let $\eta,\nu\in \Wnre$ and $j\in \{0,1, \ldots, n-1\}$
such that $\nu=\eta\concat j$ and  $|\eta|=l-1$. Then by the induction hypothesis,
$$U_{\eta}h = U_{(\eta\theta_{h})}
= U_{\eta\concat 0}h\cup U_{\eta\concat 1}h \cup \ldots \cup U_{\eta \concat{n-1}}h.$$
So as $\theta_h$ is injective we must have that
$$\{ U_{\eta\concat 0}h, U_{\eta\concat 1}h, \ldots, U_{\eta\concat {n-1}}h \}=
\{ U_{(\eta\theta_{h}\concat 0)}, U_{(\eta\theta_{h}\concat 1)}, \ldots , U_{(\eta\theta_{h}\concat {n-1})} \}.$$
Thus $U_{\eta\concat j}h=U_{(\eta\theta_{h}\concat i)}$ for some
$i\in\{0,1\}$ which also implies that
$\nu\theta_{h}=\eta\theta_{h}\concat i$.
Hence $U_{\nu}h=U_{(\nu\theta_{h})}$ and $|\nu|=|\nu\theta_{h}|$.

$(4)$  Since $(U_{\nu})h\subseteq U_{(\nu\theta_{h})}$ we have that
$U_{\nu}(hg)=(U_{\nu}h)g\subseteq U_{(\nu\theta_{h})}g$.
Thus $(\nu\theta_{h})\theta_{g}\leqslant \nu\theta_{(hg)}$.

$(5)$ For every $y\in \CCnr$ we have on one hand that
$$(\nu\concat y)hg=\nu\theta_{hg}\concat (y)(hg)_{\nu}$$ and on the other
that $$((\nu\concat y)h)g=(\nu\theta_{h}\concat (y)h_{\nu})g=
(\nu\theta_{h})\theta_{g}\concat ((y)h_{\nu})g_{_{(\nu\theta_{h})}}.$$
Thus $(\nu\theta_{h})\theta_{g}=\nu\theta_{hg}$ $\Longleftrightarrow$
$(y)(hg)_{\nu}=((y)h_{\nu})g_{_{(\nu\theta_{h})}}$
for every $y\in \CCnr$.

$(6)$ The first equivalence is a special case of $(5)$ with $g=h^{-1}$.
For the second equivalence note that
$(\nu\theta_{h})\theta_{h^{-1}}=\nu$
$\Longrightarrow$ $U_{(\nu\theta_{h})}h^{-1}\subseteq U_{\nu}$ $\Longleftrightarrow$ $U_{(\nu\theta_{h})}\subseteq U_{\nu}h$.
But by definition $U_{(\nu\theta_{h})}\supseteq (U_{\nu})h$, thus
$U_{(\nu\theta_{h})}= (U_{\nu})h$.
The other direction is clear.

$(7)$
%Assume that for every $\eta \in \Wnre$, $\theta_{h}(\nu)\leqslant \theta_{h}(\eta)\Longrightarrow \nu\leqslant \eta$.
If $(U_{\nu})h\subsetneq U_{(\nu\theta_{h})}$, then there exists $\nu\nleq \eta\in \Wnre$
such that $U_{\eta}h\subsetneq U_{(\nu\theta_{h})}$, hence $\nu\theta_{h}\leqslant \eta\theta_{h}$.
%and therefore by the assumption $\nu\leqslant \eta$ a contradiction.

$(8)$ Assume that $(U_{\nu})h=U_{(\nu\theta_{h})}$ and
$\nu\theta_{h}<\nu \concat a\theta_{h}$ for all $a\in\Xn$.
By the first condition we get that
$$U_{(\nu\theta_{h})}=\coprod_{a\in \Xn}(U_{\nu \concat a})h.$$
So by the second condition, we must necessarily have $\nu\concat a\theta_{h}=\nu\theta_{h} \concat a\pi$
for all $a \in \Xn$ for some permutation $\pi$ of $\Xn$. Thus, for all
$a\in \Xn$, we have $U_{\nu\concat a}h=U_{(\nu\theta_{h}\concat \pi(a))}.$

$(9)$ If $\theta_h \in \Aut(\Wnre)$ then $\theta_h^{-1} = \theta_{h^{-1}}$. Therefore, by $(6)$, for every
$\nu' \in \Wnre$ we have $U_{\nu'}h=U_{(\nu')\theta_{h}}$.
Conversely, assume that for every $\nu' \in \Wnre$ we have $U_{\nu'}h=U_{(\nu')\theta_{h}}$.
Then by $(6)$, we have $\theta_h\theta_{h^{-1}}=Id$, which implies
that $\theta_h$ is injective, and therefore by $(3)$, $\theta_h \in \Aut(\Wnre)$.

$(10)$ This follows as $\tau\theta_{h_{\nu}}=\Root(U_{\tau}h_{\nu})$.
\end{proof}

\begin{lemma}
Let $h\in \Homeo(\CCnr)$ and $\nu,\eta,\tau\in \Wnr$ be such that $U_{\eta}\in \Dec(U_{\nu}h)$.
Then:

 \begin{enumerate}
\item
If $\eta\leqslant \tau$, then $\nu\leqslant \tau\theta_{h^{-1}}$.

\item
If $\tau<\eta$, then $ \tau\theta_{h^{-1}}<\nu$.

 \end{enumerate}
\end{lemma}

\begin{proof} $(1)$ Since $U_{\eta}\in \Dec(U_{\nu}h)$, we have $U_{\eta} \subset U_{\nu}h $,
hence $U_{\eta}h^{-1}  \subset U_{\nu}$. If $\eta\leqslant \tau$ then this implies
$U_{\tau}h^{-1}\subseteq U_{\nu}$, thus $\nu\leqslant \tau\theta_{h^{-1}}$.

$(2)$ If $\tau < \eta$ then $U_{\tau} \nsubseteq U_{\nu}h $, hence
$U_{\tau}h^{-1}\nsubseteq U_{\nu}$. But $ \tau\theta_{h^{-1}}\leqslant \eta\theta_{h^{-1}}$
and $\nu \leqslant \eta\theta_{h^{-1}}$ so we must have $ \tau\theta_{h^{-1}}<\nu$.
\end{proof}

\vspace{0.25cm}

When we are given a homeomorphism $h\in \Homeo(\CCnr)$, starting with the root function $\theta_{h}$ of $h$, we can build
another function $\bar{\theta}_{h}:\Wnre\longrightarrow \Wnre$,  which we call the \emph{local root function of $h$}.

\begin{definition}
Define $\varepsilon\bar{\theta}_{h} := \varepsilon$, and for every $\nu\in \Wnre$ and $l\in \{0,\ldots,n-1\}$,
define $(\nu\concat l)\bar{\theta}_{h} := (\nu\concat l)\theta_{h}-\nu\theta_{h}$.
\end{definition}

Note that this function essentially detects the suffix which is added to $\nu \theta_h$ when constructing $(\nu\concat l)\theta_h$.

\begin{lemma}
For all $h\in \Homeo(\CCnr)$, the following are equivalent:
 \begin{enumerate}
\item
$\theta_{h}\in \Aut(\Wnre)$,
\item
$\nu\bar{\theta}_{h}\in \{0,\ldots,n-1\}$ holds for every $\nu\in \Wnr$,
\item
$\nu\bar{\theta}_{h}\ne \varepsilon$ holds for every $\nu\in \Wnr$,
\item
$\theta_{h}$ is injective.
 \end{enumerate}
\end{lemma}

\begin{proof}
$(1)\Rightarrow (2)$: This follows from the monotonicity of $\theta_h$.

$(2)\Rightarrow (3)$: This is trivial.

$(3)\Rightarrow (1)$ Assume that for every $\nu\in \Wnr$, we have $\nu\bar{\theta}_{h}\ne \varepsilon$.
Then as $U_{\varepsilon}h=(\CCnr)h=\CCnr=U_{(\varepsilon\theta_{h})}$, by Lemma \ref{claim:propert} $(8)$,
we get that  for every $j\in \{0,\ldots,n-1\}$ there exists $l\in \{0,\ldots,n-1\}$
such that $j\bar{\theta}_{h}=l$ and  $U_{j}h=U_{l}$.
Now the claim follows by induction.

$(1)\Longleftrightarrow(4)$ By Lemma \ref{claim:propert} $(3)$.
\end{proof}

\vspace{0.25cm}

%%%%%

Recall that for a permutation $\pi \in \Sym(\{0,\ldots,n-1\})$, we defined ({\em c.f.,} Definition \ref{definition:pi-twist})
the maps $\widehat{\pi}_{n} : \CCn \rightarrow \CCn$  and $\widehat{\pi}_{n,r} : \CCnr \rightarrow \CCnr$ by twisting each entry
according to $\pi$ (except at the first index which remains unchanged in the case of $\CCnr$). The next lemma should be obvious.

\begin{lemma}
\label{prop:1}
An element $h\in \Homeo(\CCnr)$ coincides with $\widehat{\pi}$ up to a permutation of the first coordinate if and only if for every $\nu\in \Wnr$ we have that $U_{\nu}h$ is a cone $U_{\eta} \in \Banr$ and  $U_{\nu\concat l}h=U_{\eta\concat \pi(l)}$ for all $0\leqslant l\leqslant n-1$.
\end{lemma}

The following proposition detects when a homeomorphism of $\CCnr$ acts, on some cone, as an iterated permutation.
\begin{prop}
\label{pro:permut}
Let $h\in \Homeo(\CCnr)$ and $\nu\in \Wnr$.
If $h_\nu=h_{\nu\concat l}$ holds for every $0\leqslant l\leqslant n-1$,
then $h_\nu=\hat{\pi}_n$ for some $\pi \in \Sym(\{0,\ldots,n-1\})$.
\end{prop}

\begin{proof}
Suppose that $h\in \Homeo(\CCnr)$, $\nu\in\Wnr$, and $h_\nu=h_{\nu\concat l}$ for every $0\leqslant l\leqslant n-1$.

We first observe that $|\nu|\geq 1$ so if we extend $\nu$ using a letter $x\in\Xn$ we obtain $\nu\concat x\in\Wnr$.  Now, by assumption, the local map of $h$ on any sub-cone in $U_\nu$ is equal to the local map $h_\nu$, so the map $h_\nu$ can be represented by a one-state transducer which we will call $A$, with solitary state $q$.  As $h$ is injective we see that $h_\nu$ is injective as well.

We now argue that $h_\nu$ must be surjective onto $\CCn$, using compactness and continuity.  Suppose $h_\nu$ is not surjective.  Observe first that the image of $U_\nu$ under $h$ is both open and closed, and hence the complement of this set is both open and closed.  In particular, since $h_\nu$ is not surjective, there is a finite word $\Gamma\in \Wn$ so that the image of $h_\nu$ is disjoint from the cone $U_\Gamma$ in $\CCn$.  That is, for no input to $A$ will $A$ produce an output with prefix $\Gamma$.  Now, let $w$ be the output of $A$ from $q$ on reading the input `$0$'.  Clearly the point $\overline{w}$ is in the image of $h_\nu$, and there is a minimal length prefix $x$ so that $\nu\concat\overline{0} h=x\concat\overline{w}$.  However, there is a sequence of points $(p_i)\in \CCnr\backslash U_\nu$  so that $z_i=p_i h=x\concat w^{(i)\varXi}\concat\Gamma\concat\overline{w}$ for $\varXi$ some increasing function, since $h$ is surjective, but the points $w^{(i)\varXi}\concat\Gamma\concat\overline{w}$ are not in the image of $h_\nu$.  Thus, there is some point $p$ (a limit of a subsequence of $(p_i)$) in $\CCnr\backslash U_\nu$ so that $p\cdot h=x\concat\overline{w}$ as well, which contradicts the fact that $h$ is a homeomorphism.  Therefore, we may indeed assume that $h_\nu$ is surjective.

Our argument is actually complete: the transducer $A$ has all states (there is only one) simultaneously injective and surjective, and so by Proposition \ref{prop:AllStatesHomeosToSynch}, the state $q$ also must act locally as a permutation.
\end{proof}

\section{Open problems}

Apart from the general problem of understanding better the groups introduced in this paper, we mention a few specific projects in need of further attention (for $2 \leq m, n$).

   From Proposition \ref{prop: nesting} and basic number theory, we know that if $\gcd(n-1,r_1)= \gcd(n-1,r_2)$ then $\Onk{n}{r_1}=\Onk{n}{r_2}$.  In \cite{OlukoyaAutTnr} a form of converse is shown. The work of Pardo and also of Dicks and Mart\'inez-P\'erez (see \cite{Pardo,DicksMartinezPerez}) establishes that $G_{m,r}\cong G_{n,s}$ if and only if $m=n$ and  $\gcd(n-1,r)= \gcd(n-1,s)$. We wonder the following:

   \begin{question}
    Let $n\in\N$, and $1<r_1<r_2< n$. If $\gcd(n-1,r_1)= \gcd(n-1,r_2)$ then $\Onk{n}{r_1}=\Onk{n}{r_2}$ so $\Onk{n}{r_1}\cong\Onk{n}{r_2}$.  Is the converse true?
   \end{question}

\begin{question}
Is it true that $\mathcal{O}_n \not\cong\mathcal{O}_m$  for $m\neq n$?
What about for $\mathcal{L}$ and $\mathcal{H}$?
\end{question}

Further fundamental questions we can ask are as follows.
\begin{question}
Describe the normal subgroup structure of $\On\cong \Onk{n}{n-1}$.  For instance, does the commutator subgroup of $\On$ have finite index in $\On$?
\end{question}

A related question which might not be difficult is the following.

\begin{question}
If $n$ is odd, then
$\Onr$ has a double cover given by
\[
1\to \Gnr/\Gnr'\to \Anr/\Gnr'\to \Onr\to 1,
\]
where $\Gnr/\Gnr'$ is of order two.
Does this cover split?
\end{question}

The following question is a bit more technical, but its importance is clear from Section \ref{sec:synchronization} and from our following investigations \cite{BleakCameronOlukoya} into the automorphisms of the shift.

\begin{question}\label{qn: pres}
Find presentations for the groups $\mathcal{O}_n$, $\mathcal{L}_n$ and $\mathcal{H}_n$.
\end{question}
Note that for Question \ref{qn: pres}, we know from the found connections to the automorphism groups of the full shift (one-sided and two-sided, as mentioned in the introduction), that $\Ln$ and $\Hn$ are not finitely generated.  However, having meaningful infinite presentations can often be of more practical use than having finite presentations, so this question is likely of interest for all three families of groups.

\begin{question}
Find weaker conditions for a homeomorphism state than those given in
Lemma~\ref{loop-homeo}. (The conditions of the lemma are all used in the
proof but we do not know whether they are all necessary.)
\end{question}

\appendix
\section{The inversion construction}\label{appendixInversion}
In order to keep this work self-contained and as we have made heavy use of the inversion construction in certain proofs, we give here a more detailed exposition and justification of this construction. Our exposition is based on the treatment within the paper \cite{GNSenglish}, and the more detailed explanation given in the PhD thesis \cite{Olukoya2018}. We shall give the construction in the first instance for an arbitrary transducer which induces a homeomorphism from its input space to its output space. We then indicate how this approach can be modified for the specific transducers considered in this article, that is, those initial transducers inducing self-homeomorphisms of the Cantor space $\CCnr$ as defined in Subsection~\ref{Subsection:transducersCCnr}.

\subsection{General transducers}\label{Subsection:inversiongeneral}
Throughout this section  $X_{i}$ and $X_{o}$ shall be finite alphabets, and all transducers shall have input alphabet $X_{i}$ and output alphabet $X_o$, thus, we shall often not write this out explicitly. We shall also make the implicit assumption, unless we state otherwise, that transducers are accessible and non-degenerate.

\begin{notation}
As in the main text, for $a \in \{i,o\}$ and $\nu \in X_{a}^{\ast}$ we set $U_{\nu}:= \{ \nu \delta \mid \delta \in X_{a}^{\omega} \}$.
\end{notation}

\begin{definition}
Let $A_{q_0} = \gen{X_i, X_o, Q_A, \pi_{A}, \lambda_{A}}$ be an initial transducer. Then $A_{q_0}$ is called \emph{invertible} if the map $h_{q_0}: X_{i}^{\omega} \to X_{o}^{\omega}$ is a homeomorphism.
\end{definition}

We have the following straight-forward lemma.

\begin{lemma}\label{Lemma:statesofinverseinjectivewithclpnimage}
Let $A_{q_0} = \gen{X_i, X_o, Q_A, \pi_{A}, \lambda_{A}}$ be an invertible transducer. Then  for all states $q \in Q_{A}$, $h_{q}:X_{i}^{\omega} \to X_{o}^{\omega}$ is injective and  $\Image(q):= (X_{i}^{\omega})h_{q}$ is clopen.
\end{lemma}
\begin{proof}
Let $q \in Q_{A}$ be arbitrary. Since $A_{q_0}$ is accessible by assumption, it is clear that for any state $q \in Q_{A}$, $h_{q}$ is injective.

To see that $\Image(q)$ is clopen we proceed as follows. Since $A_{q_0}$ is accessible, there is a word $\Gamma \in X_{i}^{+}$, such that $\pi_{A}(\Gamma, q_0) = q$. Let $\Delta = \lambda_{A}(\Gamma, q_0)$. Now as $h_{q_0}$ induces a homeomorphism from $ X_{i}^{\omega}$ to $X_{o}^{\omega}$, and since $U_{\Gamma}$ is clopen, then $(U_{\Gamma})h_{q_0}$ is also clopen. Since $(U_{\Gamma})h_{q_0}$ is clopen, it is compact, and so we may find words $\nu_1, \ldots, \nu_{m} \in \xns$ such that $\bigcup_{1 \le i \le m} U_{\nu_i} = (U_{\Gamma})h_{q_0}$. Observe that since for all $1 \le i \le m$, $U_{\nu_i} \subset (U_{\Gamma})h_{q_0}$, then $\Delta \leqslant \nu_i$ for all $1 \le i \le m$. Set $\mu_i = \nu_i - \Delta$, then it follows that $\bigcup_{1 \le i \le m} U_{\mu_i} = \Image(q)$ yielding the result.
\end{proof}

We now define below a function which enables us,  by shifting the origin, to compute the preimage of an element of the output space in steps.

\begin{definition}\label{Def:Lfunction}
Let $A_{q_0}$ be a transducer and $q \in Q_{A}$ be a state. Define a map $L_{q}: X_{o}^{\ast} \to X_{i}^{\ast}$, by $L_{q}(\nu) = (\{\delta \in \CCn \mid (\delta)h_{q} \in U_{\nu} \})\Root$. In other words, $L_{q}(\nu)$ returns the longest common prefix of the pre-image of the cone $U_{\nu}$.
\end{definition}

We have the following results.

\begin{lemma}\label{Lemma:transitionfunctionworks}
Let $A_{q_0}$ be an invertible transducer, $(w,q) \in X_{o}^{\ast} \times Q_{A}$ be a pair such that $ U_{w} \subset \im(q)$ and $L_q(w) = \epsilon$, and let $x \in X_{o}$ be arbitrary. Then the pair $(v,p) \in X_{o}^{\ast} \times Q_{A}$ defined by $v:= wx - \lambda_{A}(L_{q}(wx),q)$ and $p:= \pi_{A}(L_{q}(wx), q)$ also satisfies $U_{v} \subset \Image(p)$ and $L_{p}(v) = \epsilon$.
\end{lemma}
\begin{proof}
We begin by showing that $U_{v} \subset \Image(p)$.

Set $u = \lambda_{A}(L_{q}(wx),q)$, and note that by definition of $v$, $wx = uv$. Let $\rho \in X_{o}^{\omega}$ be arbitrary. Since $U_{w} \subset \Image(q)$, there is a word $\delta \in X_{i}^{\omega}$ such that $(\delta)h_{q} = wx\rho$. By definition of the function $L_{q}$, there is a $\bar{\delta} \in X_{i}^{\omega}$ such that $\delta = L_{q}(wx) \bar{\delta}$. Now observe that as $\lambda_{A}(L_{q}(wx), q) = u$, then $(\bar{\delta})h_{p} = v\rho$. Since $\rho \in X_{o}^{\omega}$ we arbitrarily chosen, we have $U_{v}\subset \Image(p)$.

We now argue that $L_{p}(v) = \epsilon$. As $$\lambda_{A}(L_{q}(wx), q) = u \mbox{ and } \pi_{A}(L_{q}(wx), q) = p,$$ it must be the case that $L_{q}(wx) = uL_{p}(v)$ yielding the equality, $\epsilon = L_{p}(v)$.
\end{proof}

\begin{lemma}\label{Lemma:inverseisaccessible}
Let $A_{q_0}$ be an invertible transducer and $(w,q) \in X_{o}^{\ast} \times Q_{A}$ be a pair such that $ U_{w} \subset \im(q)$ and $L_q(w) = \epsilon$. Then there is a word $\Gamma \in X_{o}^{\ast}$ such that $\Gamma - \lambda_{A}(L_{q_0}(\Gamma), q) = w$ and $\pi_{A}(L_{q_0}(\Gamma), q_0) =q$.
\end{lemma}
\begin{proof}
Since $A_{q_0}$ is accessible, fix a word $\Delta \in X_{i}^{+}$ such that $\pi_{A}(\Delta, q_0) = q$. Set $\gamma = \lambda_{A}(\Delta, q_0)$ and consider the word $\gamma w$. Since $U_{w} \subset \Image(q)$, and $h_{q_0}$ induces a homeomorphism from its input space to its output space, it follows that $(\gamma w)L_{q_0} = \Delta L_{q}(w) = \Delta$. Thus setting $\Gamma = \gamma w$, the following equalities are valid:  $\Gamma - \lambda_{A}(L_{q_0}(\Gamma), q_0) = w$ and $\pi_{A}(L_{q_0}(\Gamma), q_0) =q$.
\end{proof}

\begin{lemma}\label{Lemma:finite}
Let $A_{q_0}$ be an invertible transducer and $q \in Q_{A}$. There are only finitely many words $w \in X_{o}^{\ast}$ such that $U_{w} \cap \Image(q) \ne \emptyset$ and $(w)L_{q} = \epsilon$.
\end{lemma}
\begin{proof}
Let $q$ be a state of $A$ and let $\T{N}(q):= \{q\} \sqcup\{ \pi_{A}(x, q) \mid x \in X_{i} \}$. For each $p \in \T{N}(q)$, let $C_{p} \subset X_{o}^{\ast}$ be minimal such  that, $\bigcup_{\nu \in C_{p}}U_{\nu} = \Image(p)$. Since $\Image(p)$ is clopen for any state $p \in Q_A$, such a set exists. For $p \in \T{N}(q)$ let $m_p = \max\{ |\nu| \mid \nu \in C_{p}\}$ and $M= \max_{p \in \T{N}(q)} \{ m_p\}$.

Observe that as $A_{q_0}$ is assumed to be non-degenerate and $|\T{N}(q)| < \infty$, there is a $j \in \N$ such that for any word $\gamma \in X_{i}^{j}$, and any $p \in \T{N}(q)$, $|\lambda_{A}(\gamma, p)| \ge M$.

Let $\gamma \in X_{i}^{j}$ and  $x \in X_{i}$. Let $\nu = \lambda_{A}(x, q)$, $p = \pi_{A}(x,q)$ and $\eta = \lambda_{A}(\gamma,p)$. Observe that as $|\eta| \ge M$, we must have $U_{\eta} \subseteq \Image(p)$, since there is a prefix of $\eta$ which is an element of $C_{p}$. Moreover, as  $q$ is injective,  $L_{q}(\nu\eta)$ must have prefix $x$, since if there a word $\xi \in \xns$ and a letter $y \in X_{i}$ with $y \ne x$ such that $\lambda_{A}(y\xi, q)$ has $\nu\eta$ as a prefix, then $(U_{y\xi})h_{q} \subseteq (U_{x})h_{q}$ a contradiction.

By choice of $j$, the set $I_{q}:=\{\lambda_{A}(\Gamma, q) \mid \Gamma \in \xn^{j+1} \} \subset X_{i}^{p}$ satisfies $\bigcup_{\mu \in(I_q)} U_{\mu} = \Image(q)$, moreover, since $\gamma \in X_{i}^{j+1}$ was chosen arbitrarily in the previous paragraph, for any $\mu \in I_{q}$ we have, $(\mu)L_{q} \ne \epsilon$.
\end{proof}

We are now in a position to construct the inverse transducer.

Let $A_{q_0}$ be an invertible transducer. Set $$Q_{A'}:= \{ (w,q) \in X_{o} \times Q_{A} \mid U_{w} \subset \Image(q), L_{q}(w)= \epsilon  \}.$$ Define functions $\pi_{A'}: X_{o} \times Q_{A'} \to Q_{A'}$ and $\pi_{A'}: X_{o} \times Q_{A'} \to X_{i}^{\ast}$ by, $$\pi'_{A}(x, (w,q)) = (wx - \lambda_{A}(L_{q}(wx),q), \pi_{A}(L_{q}(wx), q))$$ and $\lambda'_{A}(x, (w,q)) = L_{q}(wx)$. Then set $A'_{(\epsilon, q_0)} := \gen{X_{o}, X_{i}, Q_{A'}, \pi_{A'}, \lambda_{A'}}$.

We have the following result.

\begin{lemma}\label{Lemma:inversehasnostatesofincompleteresponse}
Let $A_{q_0}$ be an invertible transducer. Then the transducer $A'_{(\epsilon, q_0)} := \gen{X_{o}, X_{i}, Q_{A'}, \pi_{A'}, \lambda_{A}'}$ is well-defined, accessible, has no states of incomplete response, is non-degenerate, and satisfies $h_{(\epsilon, q_0)} = h_{q_0}^{-1}$.
\end{lemma}
\begin{proof}
That the transducer $A'_{(\epsilon, q_0)}$ is well-defined and accessible is a consequence of Lemmas~\ref{Lemma:transitionfunctionworks} and \ref{Lemma:inverseisaccessible}.

We now demonstrate that $A'_{(\epsilon, q_0)}$ is non-degenerate, has no states of incomplete response and does in fact induce the inverse of $h_{q_0}$.

All of these statements will follow from the following claim.

\begin{claim}
Let $\Gamma \in X_{o}^{\ast}$ and $(w,q) \in Q_{A'}$. Then $\lambda_{A'}(\Gamma, (w,q)) = L_{q}(w\Gamma)$ and $\pi_{A'}(\Gamma, (w,q)) = (w\Gamma - \lambda_{A}(L_{q}(w\Gamma),q), \pi_{A}(L_{q}(w\Gamma),q))$.
\end{claim}
\begin{proof}
We proceed by induction on the size of $\Gamma$. By definition of the output function, the cases $|\Gamma| = 0$ and $|\Gamma| =1$ are satisfied.

Assume that the claim hold for words of length $k$. Let $\Gamma \in X_{o}^{k}$ and $x \in X_{o}$. We have
\begin{IEEEeqnarray*}{rCl}
\lambda_{A'}(\Gamma x, (w,q) ) &=& \lambda_{A'}(\Gamma, (w,q))\concat\lambda_{A'}(x, \pi_{A'}(\Gamma, (w,q))) \nonumber \\&=& L_{q}(w\Gamma)\concat\lambda_{A'}(x, \pi_{A'}(\Gamma, (w,q))).
\end{IEEEeqnarray*}
 Let $(v,p) = \pi_{A'}(\Gamma, (w,q))$, then by the inductive assumption, $$v = w\Gamma - \lambda_{A}(L_{q}(w\Gamma,q) \mbox{ and } p = \pi_{A}(L_{q}(w\Gamma),q).$$ Thus we have that $L_{q}(w\Gamma)\concat L_{p}(vx) =(w\Gamma x)L_{q}$. Since $U_{v} \subset \Image(p)$, $h_{q}$ is injective, $w\Gamma = \lambda_{A}(L_{q}(w\Gamma),q)\concat v$ and $\pi_{A}(L_{q}(w\Gamma),q) = p$, then $(w\Gamma x)L_{q} = (w\Gamma)L_{q}\concat(vx)L_{p}$. Thus the equality: $$\lambda_{A'}(\Gamma x, (w,q) ) = L_{q}(w\Gamma x)$$ is valid.

 Now as $\pi_{A'}(\Gamma x, (w,q)) = \pi_{A'}(x,\pi_{A'}(\Gamma, (w,q)))$, we have $$\pi_{A'}(\Gamma x, (w,q)) = \pi_{A'}(x, (v,p)).$$ By definition $\pi'_{A}(x, (v,p))= (vx - \lambda_{A}(L_{p}(vx,p), \pi_{A}(L_{p}(vx),p)$. However we observe that  since $L_{q}(w\Gamma x)= L_{q}(w\Gamma)\concat L_{p}(vx)$, then $$w\Gamma x  - \lambda_{A}(L_{q}(w\Gamma x), q) = w\Gamma x -  \lambda_{A}(L_{q}(w\Gamma ), q) \concat \lambda_{A}(L_{p}(vx),p).$$ Therefore $$w\Gamma x  - \lambda_{A}(L_{q}(w\Gamma x), q) = vx - \lambda_{A}(L_{p}(vx),p).$$ Moreover, $$\pi_{A}(L_{q}(w\Gamma x), q) = \pi_{A}(L_{p}(vx), \pi_{A}(L_{q}(w\Gamma), q)  = \pi_{A}(L_{p}(vx),p).$$
Thus the equality $$\pi_{A'}(\Gamma x, (w,q)) = (w\Gamma x - \lambda_{A}(L_{q}(w\Gamma x), q), \pi_{A}(L_{q}(w\Gamma x), q) )$$ is valid.
\end{proof}

Now since for each state $q \in Q_{A}$, $h_{q}$ is injective and continuous, it follows that for a sequence $(\Gamma_i)_{i \in \N}$ such that $|\Gamma_{i}|$ tends to infinity as $i$ tends to infinity and $U_{\Gamma_i} \subset \Image(q)$, then $|(\Gamma_{i})L_{q}|$ tends to infinity also. Thus, it follows that $h_{(\epsilon, q_0)} = h_{q_0}^{-1}$. More generally, for a state $(w,q) \in Q_{A'}$, and a word $\delta \in X_{o}^{\omega}$, $(\delta)h_{(w,q)} = (w\delta)h_{q}^{-1}$.

To see that $A'_{(\epsilon, q_0)}$ has no states of incomplete response we observe that for $x \in X_{o}$ and a state $(w,q) \in Q_{A'}$:
\begin{IEEEeqnarray*}{rCl}
\lambda_{A'}(x, (w,q)) &=& L_{q}(wx) = ( (U_{wx})h_{q}^{-1})\Root \nonumber \\
&=& (\{ (\delta)h_{(w,q)} \mid \delta \in X_{o}^{\omega} \})\Root.
\end{IEEEeqnarray*}
\end{proof}

The lemma below is a consequence of Lemma~\ref{Lemma:finite}.

\begin{lemma}\label{Lemma:finitemeansfiniteinverse}
Let $A_{q_0}$ be a finite invertible transducer, then $A'_{(\epsilon,q_0)}$ is also finite.
\end{lemma}

\begin{remark}
We observe that given $A_{q_0}$ an invertible transducer, $A'_{(\epsilon, q_0)}$ is not far from minimal. More specifically, $A'_{(\epsilon, q_0)}$ can be made minimally by performing the operation of identifying its $\omega$-equivalent states. Although carrying out this operation potentially reduces the size of the transducer making it easier to compute with, in practise, the obfuscation resulting from the identification of $\omega$-equivalent states makes the minimal representative difficult to work with when trying to prove things. It is therefore much easier to work with the transducer $A'_{(\epsilon, q_0)}$ and then deduce conclusions about its minimal representative.
\end{remark}

\subsection{Transducers acting on $\CCnr$}
We now demonstrate how to adapt the construction in Subsection~\ref{Subsection:inversiongeneral} for transducers acting on $\CCnr$. We shall omit most proofs as they are not dissimilar to those given in the general case.  All notation shall be as  in the main text, and we shall implicitly assume throughout this section assume that transducers act on $\CCnr$, are connected and are non-degenerate unless we state otherwise. We recall that for a transducer $A_{q_0} = (\rd, \xn, R_{A}, S_{A}, \pi_{A}, \lambda_{A})$ acting on $\CCnr$, $Q_{A} = R_{A} \sqcup S_{A}$ and $q_0 \in R_{A}$.

We begin be extending the definitions of the previous subsection. We observe that as some of the states of a transducer acting on $\CCnr$ induce maps from $\CCn$ and others induce maps on $\CCnr$, the value $U_{\epsilon}$ might mean either $\CCnr$ or $\CCn$ depending on the map under consideration. This does not result in confusion as an element $h_{q}$ has domain $\CCnr$ if and only  $q= q_0$ while it has range $\CCnr$ if and only if $q \in R_{A}$ (see Subsection~\ref{Subsection:transducersCCnr}). Thus in the definition below of the map $L_{q}$, $L_{q}$ takes a value $\epsilon$ satisfying $U_{\epsilon} = \CCnr$ or $U_{\epsilon} = \CCn$ depending on whether or not $q = q_0$.

\begin{definition}
Let $A_{q_0} = (\rd, \xn, R_{A}, S_{A}, \pi_{A}, \lambda_{A})$ be a transducer acting on $\CCnr$. Then $A_{q_0}$ is called \emph{invertible} if $h_{q_0}: \CCnr \to \CCnr$ is a homeomorphism.
\end{definition}

\begin{definition}
Let $A_{q_0}$ be a transducer and $q \in Q_{A}\backslash\{q_0\}$ be a state. Define a map $L_{q}: \wnre{n}{r} \sqcup \Wne \to \Wne$, by $L_{q}(\nu) = (\{\delta \in \CCn \mid (\delta)h_{q} \in U_{\nu} \})\Root$. We also define $L_{q_0}: \wnre{n}{r} \to \wnre{n}{r}$ by $L_{q_0}(\nu) = (\{\delta \in \CCnr \mid (\delta)h_{q_0} \in U_{\nu} \})\Root$. In other words, for $q \in Q_{A}$, initial or not, $L_{q}(\nu)$ returns the longest common prefix of the pre-image of the cone $U_{\nu}$.
\end{definition}

\begin{remark}
Let $A_{q_0}$ be a transducer, then $L_{q}$ maps all elements of $\Wnre$ to $\epsilon$ if $q \in S_{A}$ and it maps all elements of $\Wne$ to $\epsilon$ if $q \in R_{A}$.
\end{remark}

The lemmas of the previous subsection are equally valid when restated in the context of transducers acting on $\CCnr$. As we mentioned above, the proofs only need to be adjusted to account for the action on $\CCnr$ and so we state the lemmas below without proof. However, we first require the following lemma which clarifies a confusion arising from the notation $U_{\epsilon}$ used when constructing the inverse.

\begin{lemma}\label{Lemma:distinguishingepsilon}
Let $A_{q_0}$ be an invertible transducer acting on $\CCnr$, and let $q \in Q_{A}$ be a state. Let $\epsilon$ and $\varepsilon$ both represent the empty word, where $\epsilon \in \Wnre$ and $\varepsilon \in \Wne$ so that $U_{\epsilon} = \CCnr$ and $U_{\varepsilon} = \CCn$. Then the pair $(\epsilon, q)$ satisfies $U_{\epsilon} \subset \Image(q)$ if and only if $q = q_0$ and $(\varepsilon, q)$ satisfies $U_{\varepsilon} \subset \Image(q)$ implies that  $q \in S_{A}$.
\end{lemma}
\begin{proof}
This is a straight-forward consequence of the the definition of transducers acting on $\CCnr$. Let $(\epsilon, q)$ for $q \in Q_{A}$ satisfy $U_{\epsilon} \subset \Image(q)$, then this implies that the range of $q = \CCnr$, thus $q \in R_{A}$. However as $A_{q_0}$ induces a homeomorphism and is accessible, then range of $q$ is all of $\CCnr$ if and only if $q = q_0$.

Now suppose $(\varepsilon, q)$ satisfies $U_{\varepsilon} \subset \Image(q)$. By definition of elements of $R_{A}$, there is a letter $\dot{a} \in \rd$ such that if $q \in R_{A}$, then $\Image(q) \subset U_{\dot{a}}$, thus it must be the case that $q \in S_{A}$.
\end{proof}

\begin{remark}
Given $A_{q_0}$ an invertible transducer acting on $\CCnr$, the Lemma above enables us to write a pair $(\epsilon, q)$ for $q \in Q_{A}$ without ambiguity, as $U_{\epsilon}$ represents  $\CCnr$ if and only if $q = q_0$.  Therefore we will only use the $\epsilon$ symbol for the empty word going forward.
\end{remark}

\begin{lemma}\label{Lemma:statesofinverseinjectivewithclpnimage CCnr}
Let $A_{q_0} = (\rd, X_n, R_{A}, S_{A}, \pi_{A}, \lambda_{A})$ be an invertible transducer acting on $\CCnr$. Then  for all states $q \in Q_{A}$, $h_{q}$ is injective and  $\Image(q)$ is clopen.
\end{lemma}

\begin{lemma}\label{Lemma:transitionfunctionworks CCnr}
Let $A_{q_0}$ be an invertible transducer, $(w,q) \in (\Wnre \sqcup \Wne ) \times Q_{A}$ be a pair such that $ U_{w} \subset \im(q)$ and $L_q(w) = \epsilon$, and let $x \in \rd \sqcup \Xn$ be arbitrary such that $x \in \rd$ if $w = \epsilon$ and $q= q_0$ and $x \in \Xn$ otherwise. Then the pair $(v,p) \in (\Wnre \sqcup \Wne ) \times Q_{A}$ defined by $v:= wx - \lambda_{A}(L_{q}(wx),q)$ and $p:= \pi_{A}(L_{q}(wx), q)$ also satisfies $U_{v} \subset \Image(p)$ and $L_{p}(v) = \epsilon$.
\end{lemma}

\begin{lemma}\label{Lemma:inverseisaccessible CCnr}
Let $A_{q_0}$ be an invertible transducer and $$(w,q) \in (\Wnre \sqcup \Wne ) \times Q_{A}$$ be a pair such that $ U_{w} \subset \im(q)$ and $L_q(w) = \epsilon$. Then there is a word $\Gamma \in \Wnre$ such that $\Gamma - \lambda_{A}(L_{q_0}(\Gamma), q_0) = w$ and $\pi_{A}(L_{q_0}(\Gamma), q_0) =q$.
\end{lemma}

\begin{lemma}\label{Lemma:finite CCnr}
Let $A_{q_0}$ be an invertible transducer and $q \in Q_{A}$, then there are only finitely many words $w \in \Wnre \sqcup \Wne$ such that $U_{w} \cap \Image(q) \ne \emptyset$ and $L_{q}(w)= \epsilon$.
\end{lemma}

We may now construct the inverse transducer.

Let $A_{q_0}$ be an invertible transducer. Set $$Q_{A'}:= \{ (w,q) \in (\Wnre \sqcup \Wne) \times Q_{A} \mid U_{w} \subset \Image(q), L_{q}(w)= \epsilon  \}.$$ Define functions $\pi_{A'}: (\rd \sqcup \Xn) \times Q_{A'} \to Q_{A'}$ and $\lambda_{A'}: (\rd \sqcup \Xn) \times Q_{A'} \to \Wnre \sqcup \Wne$ by $$\pi'_{A}(x, (w,q)) = (wx - \lambda_{A}(L_{q}(wx),q), \pi_{A}(L_{q}(wx), q)),$$ and $$\lambda'_{A}(x, (w,q)) = L_{q}(wx) $$ for $x \in \Xn$ if $q \in Q_{A}\backslash \{q_0\}$ and $x \in \rd$ otherwise. Set $S_{A'} := \{ (w,q) \in  Q_{A'} \mid q \ne q_0 \}$ and set $R_{A'} := Q_{A'} \backslash S_{A'}$. This definition is justified since if $(w,q_0) \in R_{A'}$ and $x \in \rd \sqcup \xn$ ($x \in\rd$ if $w = \epsilon$ and $x \in \xn$ otherwise) are such that $L_{q_0}(wx) \ne \epsilon$, then $\pi_{A}(L_{q_0}(wx), q_0) \ne q_0$.
Finally set $A'_{(\epsilon, q_0)} := (\rd,\Xn,R_{A'}, S_{A'}, \pi_{A'}, \lambda_{A'})$.

\begin{lemma}\label{Lemma:inversehasnostatesofincompleteresponse CCnr}
Let $A_{q_0}$ be an invertible transducer. Then the transducer $A'_{(\epsilon, q_0)} := (\rd,\Xn,R_{A'}, S_{A'}, \pi_{A'}, \lambda_{A'})$ is well-defined, accessible, has no states of incomplete response, is non-degenerate, and satisfies $$h_{(\epsilon, q_0)} = h_{q_0}^{-1}.$$ If $A_{q_0}$ is finite then so also is $A'_{(\epsilon, q_0)}$.
\end{lemma}
Observe that this construction may produce a transducer with equivalent states, but these may be identified resulting in a minimal transducer.

\bibliographystyle{amsplain}
\bibliography{ploiBib}

\end{document}